\documentclass[onefignum,onetabnum]{siamart171218}
\usepackage[left=2.5cm,right=2.5cm, top=2cm, bottom=2.5cm]{geometry}

\usepackage{braket,amsfonts}
\usepackage[english]{babel}
\usepackage[utf8]{inputenc}
\usepackage{amsmath}
\usepackage{graphicx}
\usepackage{amssymb}
\usepackage{tikz-cd}
\usepackage{mathrsfs}
\usepackage[colorinlistoftodos]{todonotes}
\usepackage{enumitem}
\usepackage{yfonts}
\usepackage{algorithmic}
\usepackage{cite}

\usepackage{tikz}
\usetikzlibrary{matrix}
\usepackage{filecontents}
\usepackage{ulem}
\usepackage{mathtools}
\usepackage{adjustbox}
\usepackage{lipsum}
\usepackage[section]{placeins}

\usepackage{bm}
\usepackage{graphicx,epstopdf}
\usepackage[caption=false]{subfig}

\usepackage{amsopn}

\newcommand{\modify}[1]{{{#1}}}

\title{Accelerating Alternating Least Squares for Tensor Decomposition by Pairwise Perturbation
}

\author{Linjian Ma\thanks{Department of Computer Science, University of Illinois at Urbana-Champaign, Urbana, IL, 61801 (\email{lma16@illinois.edu}, \email{solomon2@illinois.edu}).}
\and Edgar Solomonik\footnotemark[1]
}

\headers{Accelerating Tensor Decomposition by Pairwise Perturbation}{Linjian Ma and Edgar Solomonik}

\ifpdf
\hypersetup{ pdftitle={Pairwise Perturbation} }
\fi

\date{}
\newtheorem{thm}{Theorem}[section]
\newtheorem{lem}[thm]{Lemma}

\newtheorem{cor}[thm]{Corollary}

\newcommand{\T}{T}

\newcommand{\tsr}[1]{\pmb{\mathcal{#1}}}
\newcommand{\fvcr}[1]{\textit{#1}}
\newcommand{\vcr}[1]{\mathbf{#1}}
\newcommand{\mat}[1]{\mathbf{#1}}
\newcommand{\defeq}{=}
\newcommand{\tnrm}[1]{{\left\| #1 \right\|}_2}
\newcommand{\fnrm}[1]{{\left\| #1 \right\|}_F}
\newcommand{\tinf}[1]{\inf\left\{\vnrm{\fvcr{f}_{#1}}\right\}}
\newcommand{\vnrm}[1]{{\left\| #1 \right\|}_2}

\newcommand{\cpbrak}[1]{\left[\!\left[ #1 \right]\!\right]}
\newcommand{\biggvnrm}[1]{{\Bigg\| #1 \Bigg\|}_2}
\newcommand{\inti}[2]{\{{#1},\ldots, {#2}\}}
\newcommand{\M}{M}

\newcommand{\bigast}{\mathop{\scalebox{2.}{\raisebox{-0.2ex}{$\ast$}}}}%

\begin{document}
\maketitle

\begin{abstract}
The alternating least squares algorithm for CP and Tucker decomposition is dominated in cost by the tensor contractions necessary to set up the quadratic optimization subproblems.
We introduce a novel family of algorithms that uses perturbative corrections to the subproblems rather than recomputing the tensor contractions.
This approximation is accurate when the factor matrices are changing little across iterations, which occurs when alternating least squares approaches convergence.
We provide a theoretical analysis to bound the approximation error.
Our numerical experiments demonstrate that the proposed pairwise perturbation algorithms are easy to control and converge to minima that are as good as alternating least squares.
The experimental results
show improvements of up to 3.1X with respect to state-of-the-art alternating least squares approaches for various model tensor problems and real datasets.
\end{abstract}



\section{Introduction}
Tensor decompositions provide general techniques for approximation and modeling of high dimensional data~\cite{kolda2009tensor,grasedyck2013literature,cichocki2016tensor,hao2014nonnegative,perros2015sparse,carroll1970analysis}.
They are fundamental in methods for computational chemistry~\cite{benedikt2013tensor,hummel2017low,hohenstein:044103}, physics~\cite{ORUS2014117}, and quantum information~\cite{ORUS2014117,HUCKLE2013750}.
Tensor decompositions are performed on tensors arising both in the context of numerical-PDEs (e.g. as part of preconditinioners~\cite{pazner2018approximate}) as well as in data-driven statistical modeling~\cite{anandkumar2014tensor,kolda2009tensor,liu2013tensor,1593664}.
The alternating least squares (ALS) method, which is most commonly used to compute many of these tensor decompositions, has become a target for parallelization~\cite{karlsson2016parallel,hayashi2017shared}, performance optimization~\cite{chakaravarthy2017optimizing,schatz2014exploiting}, and acceleration by randomization~\cite{battaglino2017practical}.
We propose a new algorithm, \textit{pairwise perturbation}, that asymptotically accelerates ALS iteration complexity for CP and Tucker decomposition by leveraging an approximation that is provably accurate for well-conditioned problems and is effective when the algorithm approaches the optimization local minima.

Each iteration of ALS is a sweep over quadratic optimization subproblems for each individual factor matrix composing the decomposition. 
For both CP and Tucker decomposition, computational cost of each sweep is dominated by the tensor contractions needed to setup the quadratic optimization subproblem for every factor matrix. These contractions are redone at every ALS sweep since they involve the factor matrices, all of which change after each sweep. We propose to circumvent these contractions in the scenario when the factor matrices are changing only slightly at each sweep, which is expected when ALS approaches a local minima. Our method approximates the setup of each quadratic optimization subproblem by computing perturbative corrections to the right-hand side due to the change in each factor matrix since a previous ALS sweep. To do so, pairwise perturbative operators are computed that propagate the change to each factor matrix to the subproblem needed to update each other factor matrix. Computing these operators costs slightly more than a typical ALS sweep. These operators are then reused to {\it approximately} perform more ALS sweeps until the changes to the factor matrices are deemed large, at which point, regular ALS sweeps are performed. Once the updates performed in these regular sweeps are again small,
the pairwise operators are recomputed. Each sweep computed approximately in this way costs asymptotically less than a regular ALS sweep.

For CP decomposition, CP-ALS~\cite{carroll1970analysis,harshman1970foundations} is widely used as it is robust and makes a relatively large amount of progress for the amount of computation required~\cite{kolda2009tensor} (although alternatives based on gradient and subgradient descent are also competitive~\cite{acar2011scalable}). Within CP-ALS, the computational bottleneck of each sweep involves an operation called \textit{the matricized tensor-times Khatri-Rao product} (MTTKRP). Similarly, the costliest operation in the ALS-based Tucker decomposition (Tucker-ALS) method is called \textit{the tensor times matrix-chain} (TTMc) product. 
For an order $N$ tensor with modes of dimension $s$, approximated computation of ALS sweeps via pairwise perturbation reduces the cost of that sweep from $O\left(s^NR\right)$ to {$O\left(s^2R+sR^2\right)$} for a rank-$R$ CP decomposition and from $O\left(s^NR\right)$ to $O\left(s^2R^{N-1}\right)$ for a rank-$R$ Tucker decomposition.

To quantify the accuracy of the pairwise perturbation algorithm, in Section~\ref{sec:error}, we provide an error analysis for both MTTKRP and TTMc operations. 
For both operations, we first view the ALS procedure in terms of pairwise updates, pushing updates to least-squares problems of all factor matrices as soon as any one of them is updated.
This reformulation is algebraically equivalent to the original ALS procedure.
If the relative change to each factor matrix since pairwise perturbation operators were constructed is bounded by $O(\epsilon)$, we can bound the absolute error of the way pairwise perturbation propagates updates in MTTKRP/TTMc calculations due to changes in any one of the other factor matrices.
For order three tensors, this absolute error bound yields a relative error bound that depends on a matrix condition number.
For the TTMc operation in Tucker decomposition, we derive a 2-norm relative error bound for the  
overall TTMc calculations (as opposed to updates thereof) of $O\left(\epsilon^2\right)$ that holds when the residual of the Tucker decomposition is somewhat less than the norm of the original tensor.
We also derive a Frobenius norm error bound of $O\left(\epsilon^2(s/R)^{N/2}\right)$ for TTMc, which only assumes that \textit{higher-order singular value decomposition}  (HOSVD)~\cite{de2000multilinear,tucker1966some} is performed to initialize Tucker-ALS (which is typical). In addition, in the appendix, we show that for the CP decomposition, if the factor matrices have changed by $O(\epsilon)$ in norm, the relative error in pairwise perturbation for the overall MTTKRP calculation is bounded by a term that scales with $O\left(\epsilon^2\right)$ and a tensor condition number.
However, we demonstrate that in the worst case scenario, for decomposition of any large tensor, this tensor condition number can be infinite.

In order to evaluate the performance benefit of pairwise perturbation, in Section~\ref{sec:exp}, we compare per ALS sweep and full decomposition performance using a NumPy-based~\cite{harris2020array} sequential implementation. 
Our microbenchmark results compare the performance of one CP-ALS sweep with different input tensor sizes.
We consider the initialization sweep, in which the pairwise perturbation operators are calculated, as well as the approximated sweep, in which the operators are not recalculated, of the pairwise perturbation algorithm.
{These results show that the approximated pairwise perturbation sweeps are up to $6.3$X faster than one ALS sweep with the dimension tree algorithm~\cite{li2017model,chakaravarthy2017optimizing,kaya2019computing,kaya2017high,phan2013fast,vannieuwenhoven2015computing,ballard2018parallel} for an order three tensor with dimension size 960, and up to 33.0X faster than one ALS sweep for an order six tensor.}
We then study the performance and numerical behavior of pairwise perturbation for the decomposition of synthetic tensors and application datasets.
Our experimental results show that pairwise perturbation achieves fitness as high as standard ALS, {and achieves speed-ups of up to 3.1X for CP decomposition and up to 1.13X for Tucker decomposition} with respect to state of the art ALS algorithms.

We also evaluate the performance of pairwise perturbation based on a distributed-memory parallel implementation on many nodes of an Intel KNL system (Stampede2) using Cyclops Tensor Framework~\cite{solomonik2014massively} and ScaLAPACK~\cite{Dongarra:1997:SUG:265932} libraries.
Our experimental results show that pairwise perturbation achieves fitness as high as standard ALS, and achieves speed-ups of up to 1.75X with respect to a standard ALS implementation on top of the Cyclops library on Stampede2.

\section{Background}
This section first outlines the notation used throughout this paper, then outlines the basic alternating least square algorithms for both CP and Tucker decomposition. 
\subsection{Notation and Definitions}
Our analysis makes use of tensor algebra in both element-wise equations and specialized notation for tensor operations~\cite{kolda2009tensor}.
For vectors, bold lowercase Roman letters are used, e.g., $\vcr{x}$. For matrices, bold uppercase Roman letters are used, e.g., $\mat{X}$. For tensors, bold calligraphic fonts are used, e.g., $\tsr{X}$. An order $N$ tensor corresponds to an $N$-dimensional array with dimensions $s_1\times \cdots \times s_N$. 
Elements of vectors, matrices, and tensors are denotes in parentheses, e.g., $\vcr{x}(i)$ for a vector $\vcr{x}$, $\mat{X}(i,j)$ for a matrix $\mat{X}$, and $\tsr{X}(i,j,k,l)$ for an order 4 tensor $\tsr{X}$.
Columns of a matrix $\mat{X}$ are denoted by $\vcr{x}_i =  \mat{X}(:,i)$. 
The mode-$n$ matrix product of an order $N$ tensor $\tsr{X} \in \mathbb{R}^{s_1\times \cdots \times s_N}$ with a matrix $ \textbf{A}\in \mathbb{R}^{J\times s_n}$ is denoted by $\tsr{X}\times_n \mat{A}$, with the result having dimensions $s_1\times\cdots\times s_{n-1}\times J\times s_{n+1}\times\cdots\times s_N$.
{The mode-$n$ vector product of $\tsr{X}$ with a vector $ \vcr{v}\in \mathbb{R}^{s_n}$ is denoted by $\tsr{X}\times_n \mat{v}^T$, with the result having dimensions $s_1\times\cdots\times s_{n-1}\times s_{n+1}\times\cdots\times s_N$.}
Matricization is the process of unfolding a tensor into a matrix. Given a tensor $\pmb{\mathcal{X}}$ the mode-$n$ matricized version is denoted by $\textbf{X}_{(n)}\in \mathbb{R}^{s_n\times K}$ where $K=\prod_{m=1,m\neq n}^N s_m$. 
We generalize this notation to define the unfoldings of a tensor $\tsr{X}$ with dimensions $s_1 \times \cdots \times s_N$ into an order $M+1$ tensor, $\tsr{X}_{(i_1,\ldots,i_M)}\in \mathbb{R}^{s_{i_1} \times \cdots \times s_{i_M}\times K}$, where $K=\prod_{i\in \{1,\ldots, N\} \setminus \{i_1,\ldots, i_M\}} s_i$, e.g.,
$
\tsr{X}(j,k,l,m) = \tsr{X}_{(1,3)}\left(j,l,k+(m-1)s_2\right).
$
We use parenthesized superscripts as labels for different tensors, e.g., $\tsr{X}^{(1)}$ and $\tsr{X}^{(2)}$ are generally unrelated tensors.

The Hadamard product of two matrices $\mat{U}, \mat{V} \in \mathbb{R}^{I\times J}$ resulting in matrix $\mat{W} \in \mathbb{R}^{I\times J}$ is denoted by $\mat{W} = \mat{U} \ast \mat{V}$, where $\mat{W}(i,j)= \mat{U}(i,j)\mat{V}(i,j)$. The outer product of K vectors $\vcr{u}^{(1)}, \ldots , \vcr{u}^{(K)}$ of corresponding sizes $s_1, \ldots , s_K$ is denoted by $\tsr{X} = \vcr{u}^{(1)} \circ \cdots \circ \vcr{u}^{(K)}$ where $\tsr{X} \in \mathbb{R}^{s_1\times\cdots\times s_K}$ is an order $K$ tensor. The Kronecker product of vectors $\vcr{u} \in \mathbb{R}^I$ and $\vcr{v} \in \mathbb{R}^J$ is denoted by $\vcr{w} = \vcr{u} \otimes \vcr{v}$ where $\vcr{w} \in \mathbb{R}^{IJ}$.
For matrices $\mat{A}\in \mathbb{R}^{I\times K}$ and $\mat{B}\in \mathbb{R}^{J\times K}$, their Khatri-Rao product results in a matrix of size $(IJ)\times K$ defined by
$
   \mat{A}\odot \mat{B} = \left[\vcr{a}_1\otimes \vcr{b}_1,\ldots, \vcr{a}_K\otimes \vcr{b}_K\right] .
$

\subsection{Tensor Norm} 

The spectral norm of any tensor $\tsr{T}\in\mathbb{R}^{s_1\times\cdots s_N}$ is
\[
{
\tnrm{\tsr{T}} \defeq  \max_{\substack{{\forall i \in \{2,\ldots, N\}, \vcr{x}^{(i)} \in \mathbb{R}^{s_{i}} } \\ {\vnrm{\vcr{x}^{(2)}}= \cdots=\vnrm{\vcr{x}^{(N)}}}=1}} 
\Big \|\tsr{T} \bigtimes_{i\in\{2,\ldots, N\}}\vcr{x}^{(i)T}\Big\|_2,
}
\]
where $\tsr{T}$ is contracted with {$\vcr{x}^{(i)}$} along its $i$th mode. The spectral tensor norm corresponds to the magnitude of the largest tensor singular value~\cite{lim2005singular}.
Computing the spectral norm is NP-hard~\cite{Hillar:2013}, but can usually be done in practice by specialized variants of ALS~\cite{friedland2013best}.
The spectral norm is invariant under reordering of modes of $\tsr{T}$.
Lemma~\ref{lem:norm} shows submultiplicativity of this norm for the multilinear multiplication.

\begin{lemma}
\label{lem:norm}
Given any tensor $\tsr{T}\in\mathbb{R}^{s_1 \times \cdots \times s_N}$ and matrix $\mat{M} \in \mathbb{R}^{s_N\times R}$, if  $ \tsr{V} = {\tsr{T} \times_N \mat{M}^T}$ then
$\tnrm{\tsr{V}}\leq \tnrm{\tsr{T}}\tnrm{\mat{M}}$.
\end{lemma}
\begin{proof}
There exist unit vectors {$\vcr{x}^{(2)},\ldots, \vcr{x}^{(N)}$} such that
\[
\tnrm{\tsr{V}}
={\biggvnrm{
\tsr{V} \bigtimes_{i\in\{2,\ldots, N\}}\vcr{x}^{(i)T}}}
=\biggvnrm{
{
\tsr{T} \bigtimes_{i\in\{2,\ldots, N-1\}}\vcr{x}^{(i)T} \times_N \left(\mat{M}\vcr{x}^{(N)}\right)^T
}
}.
\]
Let 
\(
\vcr{z} = \mat{M}{\vcr{x}^{(N)}} 
\), so $\vnrm{\vcr{z}}\leq\vnrm{\mat{M}}$.
If $\vnrm{\vcr{z}}=0$, then $\tnrm{\tsr{V}}=0$, the inequality holds. Otherwise, since
\begin{align*}
\biggvnrm{
{
\tsr{T} \bigtimes_{i\in\{2,\ldots, N-1\}}\vcr{x}^{(i)T} \times_N \vcr{z}^T
}
} 
\leq 
\biggvnrm{
{
\tsr{T} \bigtimes_{i\in\{2,\ldots, N-1\}}\vcr{x}^{(i)T} \times_N \vcr{z}^T
}}\frac{\vnrm{\mat{M}}}{\vnrm{\vcr{z}}}
\leq \tnrm{\tsr{T}}\vnrm{\mat{M}},
\end{align*}
the inequality still holds.
\end{proof}

\subsection{CP Decomposition with ALS}
\label{sec:bg-cp}
The CP tensor decomposition~\cite{hitchcock1927expression,harshman1970foundations} is a higher-order generalization of the matrix singular value decomposition (SVD). 
The CP decomposition is denoted by
   \[
    \tsr{X} \approx \cpbrak{\mat{A}^{(1)}, \cdots , \mat{A}^{(N)} }, \quad \text{where} \quad  
    \mat{A}^{(i)} = \left[ \mat{a}_1^{(i)}, \cdots , \mat{a}_R^{(i)} \right],
    \]
and serves to approximate a tensor by a sum of $R$ tensor products of vectors,
  \[
  \tsr{X} \approx \sum_{r=1}^{R} \mat{a}_r^{(1)}\circ \cdots \circ \mat{a}_r^{(N)}.
  \]
The CP-ALS method alternates among quadratic optimization problems for each of the factor matrices $\mat{A}^{(n)}$, resulting in linear least squares problems for each row,
\[
 \mat{A}^{(n)}_{\text{new}}\mat{P}^{(n)}{}^T \cong \mat{X}_{(n)},
\]
where the matrix $\mat{P}^{(n)}\in \mathbb{R}^{I_n \times R}$, where $I_n = s_1\times \cdots \times s_{n-1}\times s_{n+1}\times \cdots \times s_N$, is formed by Khatri-Rao products of the other factor matrices,
\[
    \mat{P}^{(n)}=\mat{A}^{(1)} \odot \cdots \odot  \mat{A}^{(n-1)}  \odot  \mat{A}^{(n+1)} \odot \cdots \odot \mat{A}^{(N)}.
\]
These linear least squares problems are often solved via the normal equations~\cite{kolda2009tensor}.
We also adopt this strategy here to devise the pairwise perturbation method.
The normal equations for the $n$th factor matrix are
     \[
     \mat{A}^{(n)}_{\text{new}}\boldsymbol{\Gamma}^{(n)}= \mat{X}_{(n)}\mat{P}^{(n)},
     \]  
where $\mat{\Gamma}\in\mathbb{R}^{R\times R}$ can be computed via
   \begin{align*}
\boldsymbol{\Gamma}^{(n)}=\textbf{S}^{(1)}\ast\cdots\ast \textbf{S}^{(n-1)}  \ast \textbf{S}^{(n+1)}\ast\cdots\ast \textbf{S}^{(N)}, \quad
     \text{with each} \quad \textbf{S}^{(i)} = \textbf{A}^{(i)T}\textbf{A}^{(i)}.
        \end{align*}
These equations also give the $n$th component of the optimality conditions for the unconstrained minimization of the nonlinear objective function,
   \[
f\left(\textbf{A}^{(1)}, \ldots , \textbf{A}^{(N)}\right) \defeq
 \frac{1}{2}\fnrm{\pmb{\mathcal{X}}-\cpbrak{ \textbf{A}^{(1)}, \cdots , \textbf{A}^{(N)} }}^2, 
   \]
   for which the $n$th component of the gradient is
   \[
      \frac{\partial f}{\partial \textbf{A}^{(n)}}
      = \textbf{G}^{(n)}=\textbf{A}^{(n)}\boldsymbol{\Gamma}^{(n)}-\textbf{X}_{(n)}   \mat{P}^{(n)}
      = \left(\mat{A}^{(n)}-\mat{A}^{(n)}_{\text{new}}\right)\boldsymbol{\Gamma}^{(n)}.
   \]
Algorithm~\ref{alg:cp_als} presents the basic ALS method described above, keeping track of the Frobenius norm of the $N$ components of the overall gradient to ascertain convergence.

\begin{algorithm}
    \caption{\textbf{CP-ALS}: ALS procedure for CP decomposition}
\label{alg:cp_als}
\begin{algorithmic}[1]
\small
\STATE{\textbf{Input: }Tensor $\tsr{X}\in\mathbb{R}^{s_1\times\cdots s_N}$, 
stopping criteria \modify{$\Delta$}
}
\STATE{Initialize $[\![ \textbf{A}^{(1)}, \ldots , \textbf{A}^{(N)} ]\!]$ as uniformly distributed random matrices within $[0,1]$, initialize $\mat{G}^{(n)}\leftarrow \mat{A}^{(n)}$, $ \mat{S}^{(n)} \leftarrow\mat{A}^{(n)T}\mat{A}^{(n)}$ for $n\in\{1,\ldots,N\}$
}
\WHILE{$\sum_{i=1}^{N}{\fnrm{\mat{G}^{(i)}}}>\modify{\Delta\|\tsr{X}\|_F}$}
\FOR{\texttt{$n\in \inti{1}{N} $}}
\STATE\label{line3}{$\boldsymbol{\Gamma}^{(n)}\leftarrow\textbf{S}^{(1)}\ast\cdots\ast \textbf{S}^{(n-1)}\ast \textbf{S}^{(n+1)}\ast\cdots\ast \textbf{S}^{(N)} $}
\STATE{Update $ \textbf{\M}^{(n)}$ based on the dimension tree algorithm shown in Figure~\ref{fig:dt}}
\STATE{$   \textbf{A}^{(n)}_\text{new} \leftarrow \textbf{\M}^{(n)}\boldsymbol{\Gamma}^{(n)}{}^{\dagger} $}
\STATE{$   \mat{G}^{(n)}\leftarrow (\mat{A}^{(n)}-\mat{A}^{(n)}_\text{new})\mat{\Gamma}^{(n)}$}
\STATE{$   \textbf{A}^{(n)} \leftarrow\textbf{A}^{(n)}_\text{new}$}
\vspace{.01in}
\STATE{$   \textbf{S}^{(n)} \leftarrow {\textbf{A}^{(n)}{}^T}\textbf{A}^{(n)} $}
\ENDFOR
\ENDWHILE
\RETURN $[\![ \textbf{A}^{(1)}, \ldots , \textbf{A}^{(N)} ]\!]$
\end{algorithmic}
\end{algorithm}
The \textit{Matricized Tensor Times Khatri-Rao Product} or MTTKRP computation, $\mat{M}^{(n)}=\mat{X}_{(n)}\mat{P}^{(n)}$, is the main computational bottleneck of CP-ALS\cite{ballard2017communication}.
 The computational cost of MTTKRP is $\Theta(s^NR)$ if $s_n=s$ for all $n\in\inti{1}{N}$. 
 {With the dimension tree algorithm, which will be detailed in Section~\ref{sec:bg-dt}, the computational complexity for all the MTTKRP calculations in one ALS sweep is $4s^{N}R$ to leading order in $s$.}
The normal equations worsen the conditioning, but are advantageous for CP-ALS, since $\mat{\Gamma}^{(n)}$ can be computed and inverted in just $O(s^2R+R^3)$ cost and the MTTKRP can be amortized by dimension trees.
If QR is used instead of the normal equations, the product of $\mat{Q}$ with the right-hand sides would have the cost $2s^NR$ and would need to be done for each linear least squares problem, increasing the overall leading order cost by a factor of $N/2$.

\subsection{Tucker Decomposition with ALS}
\label{sec:bg-tucker}

\begin{algorithm}
\caption{\textbf{Tucker-ALS}: ALS procedure for Tucker decomposition}\label{tuckerals}
\begin{algorithmic}[1]
\label{alg:tucker-als}
\small
\STATE{\textbf{Input: }Tensor $\tsr{X}\in\mathbb{R}^{s_1 \times \cdots\times s_N}$, 
decomposition ranks $\{R_1,\ldots,R_N\}$, 
stopping criteria \modify{$\Delta$}}
\STATE{Initialize $[\![\tsr{G}; \mat{A}^{(1)}, \ldots , \mat{A}^{(N)} ]\!]$ using HOSVD, initialize $\tsr{F}\leftarrow\tsr{G}$}
\WHILE{$\fnrm{\tsr{F}}>\modify{\Delta\|\tsr{X}\|_F}$} {
	\FOR{{$n\in \inti{1}{N} $}} {
\STATE{Update $ \tsr{Y}^{(n)}$ based on the dimension tree algorithm}
  \STATE{$   \textbf{A}^{(n)} \leftarrow R_n$  leading left singular vectors of $\textbf{Y}^{(n)}_{(n)} $} 
	}\ENDFOR
  \STATE{$\tsr{G}_\text{new}\leftarrow
  \tsr{Y}^{(N)}\times_{N}\mat{A}^{(N)T}$}
  \STATE{$\tsr{F}\leftarrow\pmb{\mathcal{G}}_\text{new}-\pmb{\mathcal{G}}$}
  \STATE{$\tsr{G}\leftarrow\pmb{\mathcal{G}}_\text{new}$}
}
\ENDWHILE
\RETURN $[\![\pmb{\mathcal{G}}; \textbf{A}^{(1)}, \ldots , \textbf{A}^{(N)} ]\!]$ 
\end{algorithmic}
\end{algorithm}
In this section we review the ALS method for computing a low-rank Tucker decomposition of a tensor~\cite{tucker1966some}. Tucker decomposition approximates a tensor by a core tensor contracted by matrices with orthonormal columns along each mode. The Tucker decomposition is given by 
  \[
  \tsr{X} \approx \cpbrak{\tsr{G}; \mat{A}^{(1)}, \ldots , \mat{A}^{(N)}  }\defeq \tsr{G}\times_1\textbf{A}^{(1)}\times_2\textbf{A}^{(2)}\cdots\times_N\textbf{A}^{(N)} .
  \]
The corresponding element-wise expression is
\begin{align*}
\tsr{X}\left(x_1,\ldots,x_N\right) \approx  \sum_{\{z_1,\ldots, z_N\}}\tsr{G}\left(z_1,\ldots,z_N\right)\prod_{r\in\{1,\ldots,N\}} \mat{A}^{(r)}\left(x_r,z_r\right).
\end{align*}
The core tensor $\tsr{G}$ is of order $N$ with dimensions (Tucker ranks) $R_1\times \cdots \times R_N$ (throughout error and cost analysis we assume each $R_n=R$ for $n\in \inti{1}{N}$).
The matrices $\textbf{A}^{(n)} \in \mathbb{R}^{s_n\times R_n} $ have orthonormal columns.

The \textit{higher-order singular value decomposition} (HOSVD)~\cite{de2000multilinear,tucker1966some} computes the leading left singular vectors of each one-mode unfolding of $\tsr{X}$, providing a good starting point for the Tucker-ALS algorithm. The classical HOSVD computes the truncated SVD of $\mat{X}_{(n)}\approx\mat{U}^{(n)}\mat{\Sigma}^{(n)}\mat{V}^{(n)T}$ and sets $\mat{A}^{(n)} = \mat{U}^{(n)}$ for $n\in\inti{1}{N}$. The interlaced HOSVD~\cite{vannieuwenhoven2012new,hackbusch2012tensor} instead computes the truncated SVD of 
\[\mat{Z}^{(n)}_{(n)} =\mat{U}^{(n)}\mat{\Sigma}^{(n)}\mat{V}^{(n)T} \quad \text{where} \quad \tsr{Z}^{(1)}=\tsr{X} \quad \text{and}\quad \mat{Z}^{(n+1)}_{(n)} = \mat{\Sigma}^{(n)}\mat{V}^{(n)T}.\]
The interlaced HOSVD is cheaper, since the size of each $\tsr{Z}^{(n)}$ is $s^{N-n+1}R^{n-1}$.

The ALS method for Tucker decomposition~\cite{andersson1998improving,de2000best,kolda2009tensor}, which is also called the \textit{higher-order orthogonal iteration} (HOOI), then proceeds by fixing all except one factor matrix, and computing a low-rank matrix factorization to update that factor matrix and the core tensor.
To update the $n$th factor matrix, Tucker-ALS factorizes
    \[
    \tsr{Y}^{(n)}=\tsr{X}\times_1\mat{A}^{(1)T}\cdots\times_{n-1}\mat{A}^{(n-1)T}\times_{n+1}\mat{A}^{(n+1)T}\cdots\times_{N}\mat{A}^{(N)T},
    \]
which is called the \textit{Tensor Times Matrix-chain} or TTMc, into a product of an matrix with orthonormal columns $\mat{A}^{(n)}$ and the core tensor $\tsr{G}$, so that
    \(
    \mat{Y}^{(n)}_{(n)} \approx \mat{A}^{(n)}\mat{G}_{(n)} .
    \)
This factorization can be done by taking $\mat{A}^{(n)}$ to be the $R_n$ leading left singular vectors of $\mat{Y}^{(n)}_{(n)}$.
This Tucker-ALS procedure is given in Algorithm~\ref{tuckerals}.

As in previous work~\cite{oseledets2009breaking,choi2018high}, our implementation computes these singular vectors by finding the left eigenvectors of
the Gram matrix
    \(
    \mat{W} = \mat{Y}^{(n)}_{(n)}\mat{Y}^{(n)T}_{(n)}.
    \)
Computing the Gram matrix sacrifices some numerical stability, but avoids a large SVD and provides consistency of the signs of the singular vectors 
across ALS sweeps.

{
\subsection{The Dimension Tree Algorithm}
\label{sec:bg-dt}
For CP-ALS, the tensor contractions for MTTKRP can be amortized across the linear least squares problems necessary for a given ALS sweep (for loop iteration in Algorithm~\ref{alg:cp_als}).
Such amortization techniques are referred to as dimension tree algorithms and a variety of dimension trees have been studied to minimize costs~\cite{li2017model,chakaravarthy2017optimizing,kaya2019computing,kaya2017high,phan2013fast,vannieuwenhoven2015computing,ballard2018parallel}.
As our analysis focuses on leading order cost in $s$, simple binary dimension trees are an optimal choice.
These dimension trees for $N=3,4$ are illustrated in Figure~\ref{fig:als_cp3},\ref{fig:als_cp4}. We define the partially contracted MTTKRP intermediates $\tsr{M}^{(i_1,i_2,\ldots,i_m)}$ therein as follows,
\begin{equation}
\tsr{M}^{(i_1,i_2,\ldots,i_m)} = \tsr{X}_{(i_1,i_2,\ldots,i_m)}\bigodot_{j\in {\inti{1}{N}}\setminus\{i_1,i_2,\ldots, i_m\}}\mat{A}^{(j)}.
\label{eq:tensors-cp} 
\end{equation}
Elementwise,
\begin{align*}
\tsr{M}^{(i_1,i_2,\ldots,i_m)}(x_{i_1},x_{i_2},\ldots,x_{i_m},k) =
\sum_{\{x_1,\ldots, x_N\}\setminus \{x_{i_1},x_{i_2},\ldots,x_{i_m}\}}\tsr{X}(x_1,\ldots,x_N)\prod_{r\in\{1,\ldots,N\} \setminus \{i_1,i_2,\ldots,i_m\}} \mat{A}^{(r)}(x_r,k),
\end{align*}
where $\tsr{M}^{(1,\ldots, N)}$ is the input tensor $\tsr{X}$. 
The first level contractions (contractions between the input tensor and one factor matrix) can be done via matrix multiplications between the reshaped input tensor and the factor matrix.
These contractions have a cost of $O\left(s^NR\right)$ and are generally the most time-consuming part of ALS. Other contractions (transforming one intermediate into another intermediate) can be done via batched matrix-vector products, and the complexity of an $i$th level contraction is $O\left(s^{N+1-i}R\right)$. Because two first level contractions are necessary for the construction of tree dimension tree, as is illustrated in Figure~\ref{fig:als_cp3},\ref{fig:als_cp4}, to calculate all the $\textbf{\M}^{(n)}$ in one ALS sweep, to leading order in $s$, the computational complexity is $4s^{N}R.$

For Tucker-ALS, The \textit{Tensor Times Matrix-chain} or TTMc that computes each $\tsr{Y}^{(n)}$ is the main computational bottleneck of Tucker-ALS~\cite{kaya2016high} and can also be amortized by the dimension tree. The intermediates for Tucker dimension tree are the partially contracted TTMc, $\tsr{Y}^{(i_1,i_2,\ldots,i_m)}$, defined as follows,
\[
\tsr{Y}^{(i_1,i_2,\ldots,i_m)} = \tsr{X}\bigtimes_{j\in {\inti{1}{N}}\setminus\{i_1,i_2,\ldots, i_m\}}\mat{A}^{(j)}{}^T,
\]
where $\tsr{X}$ is contracted with all the matrices $\mat{A}^{(j)}$ except $\mat{A}^{(i_1)},\ldots,\mat{A}^{(i_m)}$.
Each contraction can be done via matrix multiplications, and the complexity of an $i$th level contraction is $O\left(s^{N+1-i}R^i\right)$. Similar to CP-ALS, to calculate all the $\tsr{Y}^{(n)}$ in one ALS sweep, to leading order in $s$, the computational complexity is $4s^{N}R.$
}

\section{Pairwise Perturbation Algorithms}
\label{sec:pp}
We now introduce a pairwise perturbation (PP) algorithm to accelerate the ALS procedure when the iterative optimization steps are approaching a local minimum.
{We first derive the approximation for order three tensors, then generalize the algorithm to order $N$ tensors.}
The key idea of the pairwise perturbation method is to compute \textit{pairwise perturbation operators}, which correlate a pair of factor matrices.
These tensors are then used to repeatedly update the quadratic subproblems for each tensor.
As we will show, these updates are provably accurate if the factor matrices do not change significantly since their state at the time of formation of the pairwise perturbation operators.

\subsection{Pairwise Perturbation for Order Three Tensors}

 \begin{figure}[htb]
\centering
\subfloat[ALS dimension tree with $N=3$]{\includegraphics[width=.32\textwidth, keepaspectratio]{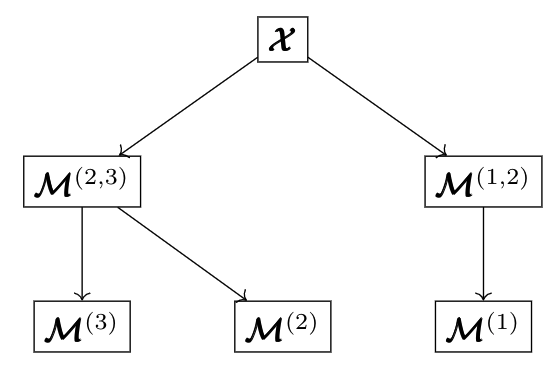}\label{fig:als_cp3}}
\subfloat[ALS dimension tree with $N=4$]{\includegraphics[width=.4\textwidth, keepaspectratio]{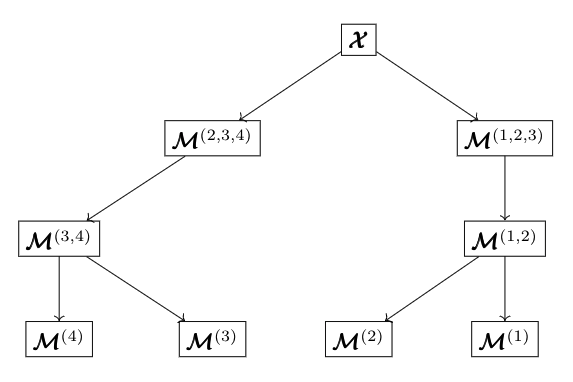}\label{fig:als_cp4}}

\subfloat[PP dimension tree with $N=3$]{\includegraphics[width=.32\textwidth, keepaspectratio]{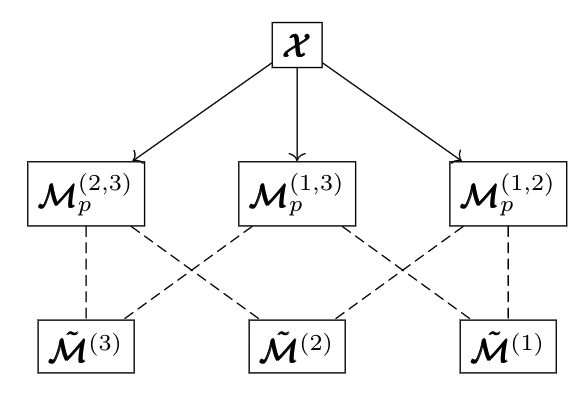}\label{fig:pp_cp3}}
\subfloat[PP dimension tree with $N=4$]{\includegraphics[width=.55\textwidth, keepaspectratio]{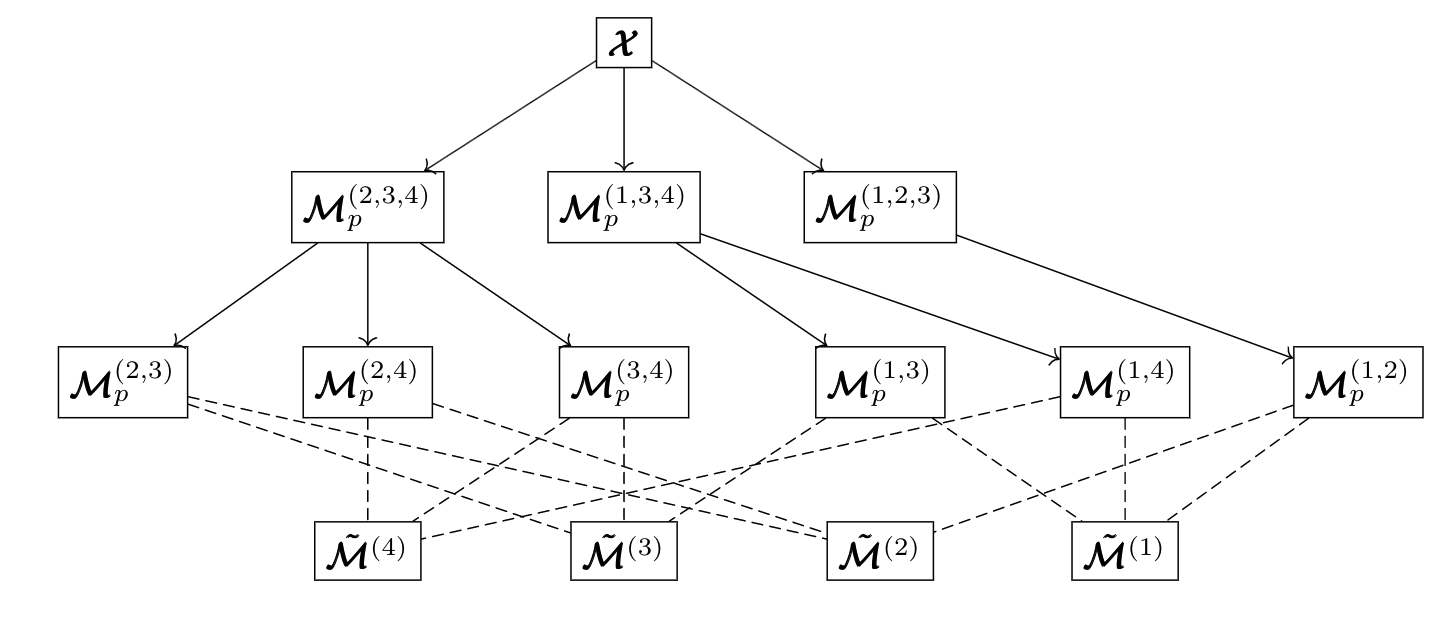}\label{fig:pp_cp4}}
\caption{{
Dimension trees for ALS and pairwise perturbation. In (c)(d), the solid arrows denote the data dependencies in building pairwise perturbation operators, and is calculated in the PP initialization step. The dashed lines denote the data dependencies in the PP approximated step calculations.
}
}
\label{fig:dt}
\end{figure}

\subsubsection{CP-ALS}
The pairwise perturbation procedure for CP-ALS approximates the MTTKRP outputs. Consider an order three equi-dimensional tensor with size in each mode $s$ and CP rank $R$, the first mode MTTKRP can be expressed as 
$ \mat{M}^{(1)} = \mat{X}_{(1)}\left(\mat{A}^{(2)} \odot \mat{A}^{(3)}\right)$. Let $\mat{A}_p^{(n)}$ denote the $\mat{A}^{(n)}$ calculated with regular ALS at some number of sweeps prior to the current one. Then $\mat{A}^{(n)}$ at the current sweep can be expressed as 
 \[
   \mat{A}^{(n)} = \mat{A}_p^{(n)} + d\mat{A}^{(n)},
 \]
 and $\mat{M}^{(1)}$ can be expressed as 
  \begin{equation}
 \label{eq:cp-pp-m1}
 \mat{M}^{(1)} =
 \underbrace{
 \mat{X}_{(1)}\left(\mat{A}_p^{(2)} \odot \mat{A}_p^{(3)}\right) + 
  \mat{X}_{(1)}\left(\mat{A}_p^{(2)} \odot d\mat{A}^{(3)}\right) + 
  \mat{X}_{(1)}\left(d\mat{A}^{(2)} \odot \mat{A}_p^{(3)} \right)
  }_{\mat{U}^{(1)}}
  + 
  \mat{X}_{(1)}\left(d\mat{A}^{(2)} \odot d\mat{A}^{(3)}\right).
 \end{equation}
The pairwise perturbation procedure for CP-ALS approximates $\mat{M}^{(1)}$ with $\Tilde{\mat{M}}^{(1)}=\mat{U}^{(1)} + \mat{V}^{(1)}$, where $\mat{U}^{(1)}$ is the first three terms in Equation~\ref{eq:cp-pp-m1} and $\mat{V}^{(1)}$ approximates the final term through approximating the input tensor $\tsr{X}$ by its approximate CP decomposition,
{\small
\begin{align*}
\mat{X}_{(1)}\left(d\mat{A}^{(2)} \odot d\mat{A}^{(3)}\right) 
\approx \mat{V}^{(1)}  =\left(\cpbrak{\mat{A}^{(1)}, \mat{A}^{(2)}, \mat{A}^{(3)} }\right)_{(1)}\left(d\mat{A}^{(2)} \odot d\mat{A}^{(3)}\right) 
= \mat{A}^{(1)} \bigg( \left(\mat{A}^{(2)T} d\mat{A}^{(2)}\right) * \left(\mat{A}^{(3)T} d\mat{A}^{(3)}\right) \bigg),
\end{align*}
}
which can be calculated with the cost of $O\left(sR^2\right)$. The remaining error term is 
\[\left(\tsr{X} - \cpbrak{\mat{A}^{(1)}, \mat{A}^{(2)}, \mat{A}^{(3)} }\right)_{(1)}\left(d\mat{A}^{(2)} \odot d\mat{A}^{(3)}\right).\]
Therefore, the norm of the error scales as $O\left(C\epsilon^2\right)$ if each $\tnrm{d\mat{A}^{(i)}}\leq \epsilon$ and the decomposition residual norm is bounded by $C$.

The approximated MTTKRP, $\Tilde{\mat{M}}^{(1)}$, can be rewritten as a function of $\tsr{M}_p^{(i_1,i_2,\ldots,i_m)}$, which is defined in the same way as $\tsr{M}^{(i_1,i_2,\ldots,i_m)}$ in Equation~\ref{eq:tensors-cp} except that $\tsr{X}$ is contracted with $\mat{A}_p^{(j)}$ for $j\in \inti{1}{N} \setminus \{i_1,i_2,\ldots,i_m\}$,
\begin{align*}
    \Tilde{\mat{M}}^{(1)}(x,k) =  \mat{M}_p^{(1)}(x,k) + \sum_{y=1}^{s}\tsr{M}_p^{(1,2)}(x,y,k) d\mat{A}^{(2)}(y,k) + \sum_{y=1}^{s}\tsr{M}_p^{(1,3)}(x,y,k) d\mat{A}^{(3)}(y,k) + \mat{V}^{(1)}(x,k).
\end{align*}
PP has two steps: the initialization step, where the terms $\mat{M}_p^{(1)}$ and pairwise perturbation operators $\tsr{M}_p^{(1,2)}$, $\tsr{M}_p^{(1,3)}$ are calculated, and the approximated step, where these terms are used in the equation above to calculate $\Tilde{\mat{M}}^{(1)}$. Using the dimension tree structure shown in Figure~\ref{fig:pp_cp3}, the initialization step for all the three modes can be done with the leading order cost of $6s^3R$, 1.5X the cost of the ALS dimension tree. Each approximated step for all the modes can be done with the leading order cost of $3\left(4s^2R + 6sR^2\right)$ overall.

\subsubsection{Tucker-ALS}
We derive a similar pairwise perturbation algorithm for order three Tucker-ALS.
The first mode of TTMc can be expressed as 
$ \tsr{Y}^{(1)} = \tsr{Y} \times_2 \mat{A}^{(2)T} \times_3 \mat{A}^{(3)T}$. PP approximates $\tsr{Y}^{(1)}$ with 
\[
\Tilde{\tsr{Y}}^{(1)} = \tsr{X} \times_2 \mat{A}_p^{(2)T} \times_3 \mat{A}_p^{(3)T} + \tsr{X} \times_2 \mat{A}_p^{(2)T} \times_3 d\mat{A}^{(3)T} + \tsr{X} \times_2 d\mat{A}^{(2)T} \times_3 \mat{A}_p^{(3)T},
\]
and the error term is $\tsr{X} \times_2 d\mat{A}^{(2)T} \times_3 d\mat{A}^{(3)T}$.
The expression above can be rewritten as a function of $\tsr{Y}_p^{(i_1,i_2,\ldots,i_m)}$, which is defined in the same way as $\tsr{Y}^{(i_1,i_2,\ldots,i_m)}$ except that $\tsr{X}$ is contracted with $\mat{A}_p^{(j)}$ for $\tsr{Y}_p^{(i_1,i_2,\ldots,i_m)}$,
\[
\Tilde{\tsr{Y}}^{(1)} = \tsr{Y}_p^{(1)} + \tsr{Y}_p^{(1,2)} \times_2 d\mat{A}^{(2)T} + \tsr{Y}_p^{(1,3)} \times_3 d\mat{A}^{(3)T} .
\]
Using the dimension tree structure, the initialization step for all the three modes can be done with the leading order cost of $6s^3R$, 1.5X the cost of the ALS dimension tree. Each approximated step for all the modes can be done with the leading order cost of $12s^2R^2$ overall.

\subsection{General Pairwise Perturbation Algorithm} We now generalize PP to order $N$ tensors.

\subsubsection{CP-ALS}
The MTTKRP in $n$th mode, $\mat{M}^{(n)}$, can be expressed as 
  \[ 
    \mat{M}^{(n)}=\mat{X}_{(n)}\bigodot_{i=1,i \neq n}^N \left(\mat{A}_p^{(i)} + d\mat{A}^{(i)}\right).
 \]
 $\mat{M}^{(n)}$ can be expressed as a function of $\tsr{M}_p^{(i_1,i_2,\ldots,i_m)}$ as
 follows,
\begin{align*}
   \mat{M}^{(n)}(y,k)
    =&\mat{M}_p^{(n)}(y,k) + \sum_{i=1,i\neq n}^{N}\sum_{x=1}^{s_i}\pmb{\mathcal{M}}_p^{(i,n)}(x,y,k) d\mat{A}^{(i)}(x,k) + \\
    &\sum_{i=1,i\neq n}^{N}\sum_{j=i+1,j\neq n}^{N}\sum_{x=1}^{s_i}\sum_{z=1}^{s_j}\tsr{M}_p^{(i,j,n)}(x,z,y,k) d\mat{A}^{(i)}(x,k)d\mat{A}^{(j)}(z,k) + \cdots.
\end{align*}
From the above expression we observe that, except the first two terms, all terms include the contraction between tensor $\tsr{M}_p^{(i_1,i_2,\ldots,i_m)}$ and at least two matrices $d\mat{A}^{(i)}$, so that their norm scales quadratically with the norm of the perturbative updates $d\mat{A}^{(i)}$. Therefore, their norm scales as $O\left(\epsilon^2\right)$ if $\tnrm{d\mat{A}^{(i)}}\leq \epsilon$.
{The pairwise perturbation algorithm obtains an effective approximation by keeping the first two terms (these terms are illustrated in Figure~\ref{fig:pp_cp4} for an order four tensor), and approximating the input tensor using its approximate CP decomposition in the third term to lower the error to a greater extent,
\begin{equation}
\label{eq:ppupdate}
   \Tilde{\mat{M}}^{(n)}(y,k)
    =\mat{M}_p^{(n)}(y,k) + \sum_{i=1,i\neq n}^{N}\sum_{x=1}^{s_i}\tsr{M}_p^{(i,n)}(x,y,k) d\mat{A}^{(i)}(x,k) + \sum_{i,j=1,i,j\neq n, i\neq j}^{N}\mat{V}^{(n,i,j)}(y,k),  
\end{equation}
\vspace{-2mm}
\begin{align*}
 \text{where} \quad
    \mat{M}_p^{(n)}=\mat{X}_{(n)}\bigodot_{i=1,i \neq n}^N \mat{A}_p^{(i)}, \quad
    \tsr{M}_p^{(i,n)}=\tsr{X}_{(i,n)}\bigodot_{j\in\inti{1}{N}\setminus\{i,n\}}^N \mat{A}_p^{(j)}, \\ 
    \text{and} \quad 
    \mat{V}^{(n,i,j)} = \mat{A}^{(n)} \bigg(  \left(\mat{A}^{(i)T} d\mat{A}^{(i)}\right) * \left(\mat{A}^{(j)T} d\mat{A}^{(j)}\right) * \bigast_{k=1, k\neq i,j,n}^{N}\left(\mat{A}^{(k)T}\mat{A}^{(k)}\right) \bigg).
\end{align*}
We evaluate the benefit of including the $\mat{V}^{(n,i,j)}$ correction in Section~\ref{sec:numpyexp}.
Given $\tsr{M}_p^{(i,n)}$ and $\mat{M}_p^{(n)}$, calculation of $\Tilde{\mat{M}}^{(n)}$ for $n\in \inti{1}{N}$ requires $2N^2\left(s^2R+sR^2\right)$ operations overall. Further, we show in Section~\ref{sec:err-cp} that the column-wise relative approximation error of $\Tilde{\mat{M}}^{(n)}$ with respect to ${\mat{M}}^{(n)}$ is small if each
 \(
\vnrm{d\vcr{a}_k^{(n)}}/\vnrm{\vcr{a}_k^{(n)}} \text{   for } n\in \inti{1}{N}, k\in\inti{1}{R}
 \)
is sufficiently small. Algorithm~\ref{alg:cp_als_pp} presents the PP-CP-ALS method described above.
}

\begin{algorithm}
    \caption{\textbf{PP-CP-ALS}: Pairwise perturbation procedure for CP-ALS}
\label{alg:cp_als_pp}
\begin{algorithmic}[1]
\small
\STATE{\textbf{Input:} tensor $\tsr{X}\in\mathbb{R}^{s_1\times\cdots\times s_N}$, 
       stopping criteria \modify{$\Delta$}, PP tolerance $\epsilon<1$}
\STATE{Initialize $[\![ \textbf{A}^{(1)}, \ldots , \textbf{A}^{(N)} ]\!]$ as uniformly distributed random matrices within $[0,1]$, initialize $\mat{G}^{(n)}, d\mat{A}^{(n)}\leftarrow \mat{A}^{(n)}$, $ \mat{S}^{(n)} \leftarrow\mat{A}^{(n)T}\mat{A}^{(n)}$ for $i\in\{1,\ldots,N\}$
}
\WHILE{$\sum_{i=1}^{N}{\fnrm{\mat{G}^{(i)}}}>\modify{\Delta\|\tsr{X}\|_F}$}
  \IF{$\forall ~i \in \inti{1}{N}, {\fnrm{d\mat{A}^{(i)}}}<\epsilon{\fnrm{\mat{A}^{(i)}}}$}
    \STATE{Compute $\tsr{M}^{(i,n)}_p,\mat{M}^{(n)}_p$ for $i,n\in \inti{1}{N}$ via dimension tree in Section~\ref{subsec:cost}}
    \FOR{\texttt{$n\in \inti{1}{N} $}}
      \STATE{$\mat{A}_p^{(n)} \gets \mat{A}^{(n)}$, $d\mat{A}^{(n)} \gets \mat{O}$}
    \ENDFOR
    \WHILE{$\sum_{i=1}^{N}{\fnrm{\mat{G}^{(i)}}} > \modify{\Delta\|\tsr{X}\|_F}$ and $\forall ~i \in \inti{1}{N}, \modify{\fnrm{d\mat{A}^{(i)}}<\epsilon\fnrm{\mat{A}^{(i)}}}$}
      \FOR{\texttt{$n\in \inti{1}{N} $}}
        \STATE{$\boldsymbol{\Gamma}^{(n)}\leftarrow\mat{S}^{(1)}\ast\cdots\ast \mat{S}^{(n-1)}\ast \mat{S}^{(n+1)}\ast\cdots\ast \mat{S}^{(N)} $}
        \STATE{
        Update $\Tilde{\mat{M}}^{(n)}$ based on Equation~\ref{eq:ppupdate}
         }
        \STATE{$   \mat{A}^{(n)}_\text{new} \leftarrow \Tilde{\mat{M}}^{(n)}\boldsymbol{\Gamma}^{(n)}{}^{\dagger} $}
        \STATE{$   \mat{G}^{(n)}\leftarrow (\mat{A}^{(n)}-\mat{A}^{(n)}_\text{new})\mat{\Gamma}^{(n)}$}
        \STATE{$   \mat{A}^{(n)} \leftarrow\mat{A}^{(n)}_\text{new}$}
        \vspace{.01in}
        \STATE{$   \mat{S}^{(n)} \leftarrow {\mat{A}^{(n)}}{}^T\mat{A}^{(n)} $}
        \STATE{$d\mat{A}^{(n)} = \mat{A}^{(n)}_\text{new}-\mat{A}^{(n)}_p$}
      \ENDFOR
    \ENDWHILE
  \ENDIF
  \STATE{Perform regular ALS sweep as in Algorithm~\ref{alg:cp_als}, taking $d\mat{A}^{(n)} \gets \mat{A}^{(n)}_\text{new}-\mat{A}^{(n)}$ for each $n\in\inti{1}{N}$}
\ENDWHILE
\RETURN $[\![ \mat{A}^{(1)}, \ldots , \mat{A}^{(N)} ]\!]$
\end{algorithmic}
\end{algorithm}

\subsubsection{Tucker-ALS}
We derive a similar pairwise perturbation algorithm for Tucker-ALS.
Similar to the expression for $\mat{M}^{(n)}$ in CP-ALS, $\tsr{Y}^{(n)}$ can be expressed as 
\[ 
    \tsr{Y}^{(n)}=\tsr{X}\bigtimes_{i=1,i \neq n}^N \left(\mat{A}_p^{(i)}{}^T + d\mat{A}^{(i)}{}^T\right).
\]
The expression above can be rewritten as a function of $\tsr{Y}_p^{(i_1,i_2,\ldots,i_m)}$, 
 $$
\tsr{Y}^{(n)} {=} \tsr{Y}_p^{(n)} {+} \sum_{i=1,i\neq n}^{N}\tsr{Y}_p^{(i,n)}\times_i d\mat{A}^{(i)}{}^T
+ \sum_{i=1,i\neq n}^{N}\sum_{j=i+1,j\neq n}^{N}\tsr{Y}_p^{(i,j,n)}\times_i d\mat{A}^{(i)}{}^T\times_j d\mat{A}^{(j)}{}^T + {\cdots} .
 $$
The pairwise perturbation algorithm again takes only the first order terms in $d\mat{A}^{(i)}$, computing
{\small
\begin{align*}
   \Tilde{\tsr{Y}}^{(n)}
    =\tsr{Y}_p^{(n)} + \sum_{i=1,i\neq n}^{N}\tsr{Y}_p^{(i,n)}\times_i d\mat{A}^{(i)}{}^T, 
\quad\text{where} \quad 
    \tsr{Y}_p^{(n)}
    =\tsr{X}\bigtimes_{l=1,l\neq n}^{N} \mat{A}_p^{(l)}{}^T
\quad \text{and} 
\quad    \tsr{Y}_p^{(i,n)}
    =\tsr{X}\bigtimes_{j\in \inti{1}{N}\setminus\{i,n\}} \mat{A}_p^{(j)}{}^T.
 \end{align*}
 }
Given $\tsr{Y}_p^{(i,n)}$ and $\tsr{Y}_p^{(n)}$, $\Tilde{\tsr{Y}}^{(n)}$ for $n\in\inti{1}{N}$ can be calculated with $2N^2s^2R^{N-1}$ cost overall.
In Section~\ref{sec:error-tucker}, we show that the relative Frobenius norm approximation
 error of $\Tilde{\pmb{\mathcal{Y}}}^{(n)}$ with respect to ${\pmb{\mathcal{Y}}}^{(n)}$ is small, so long as each
 \(
\fnrm{d\textbf{A}^{(n)}}/\fnrm{ \textbf{A}^{(n)}}
 \)
is sufficiently small. Algorithm~\ref{alg:tucker_als_pp} presents the PP-Tucker-ALS method described above.

\begin{algorithm}
    \caption{\textbf{PP-Tucker-ALS}: Pairwise perturbation procedure for Tucker-ALS}
\label{alg:tucker_als_pp}
\begin{algorithmic}[1]
\small
\STATE{\textbf{Input:} tensor $\tsr{X}\in\mathbb{R}^{s_1\times\cdots\times s_N}$, 
decomposition ranks $\{R_1,\ldots,R_N\}$, 
       stopping criteria \modify{ $\Delta$}, PP tolerance $\epsilon$}
\STATE{Initialize $[\![\tsr{G}; \mat{A}^{(1)}, \ldots , \mat{A}^{(N)} ]\!]$ using HOSVD,
initialize $d\mat{A}^{(n)}\leftarrow \mat{A}^{(n)}$ for $i\in\{1,\ldots,N\}$, initialize $\tsr{F}\leftarrow\tsr{G}$}
\WHILE{$\fnrm{\tsr{F}}>\modify{\Delta \|\tsr{X}\|_F}$} 
  \IF{$\forall ~i \in \inti{1}{N}, {\fnrm{d\mat{A}^{(i)}}}<\epsilon{\fnrm{\mat{A}^{(i)}}}$}
    \STATE{Compute $\tsr{Y}^{(i,n)}_p,\tsr{Y}^{(n)}_p$ for $i,n\in \inti{1}{N}$ via dimension tree in Section~\ref{subsec:cost}}
    \FOR{\texttt{$n\in \inti{1}{N} $}}
      \STATE{$\mat{A}_p^{(n)} \gets \mat{A}^{(n)}$, $d\mat{A}^{(n)} \gets \mat{O}$}
    \ENDFOR
    \WHILE{$\fnrm{\tsr{F}}>\modify{\Delta\|\tsr{X}\|_F}$ and \modify{$\fnrm{\tsr{F}}<\modify{\epsilon\|\tsr{X}\|_F}$}}
      \FOR{\texttt{$n\in \inti{1}{N} $}}
        \STATE{$ \tsr{Y}^{(n)}\leftarrow\pmb{\mathcal{Y}}_p^{(n)} + \sum_{i=1,i\neq n}^{N}\pmb{\mathcal{Y}}_p^{(i,n)}\times_i d\textbf{A}^{(i)}$ }
        \STATE{$   \textbf{A}^{(n)} \leftarrow R_n$ leading left singular vectors of $\textbf{Y}^{(n)}_{(n)}$}
        \STATE{$   d\textbf{A}^{(n)} \leftarrow\textbf{A}^{(n)}-\textbf{A}_p^{(n)} $}
      \ENDFOR
        \STATE{$\tsr{G}_\text{new}\leftarrow
  \tsr{Y}^{(N)}\times_{N}\mat{A}^{(N)T}$}
        \STATE{$\tsr{F}\leftarrow\pmb{\mathcal{G}}_\text{new}-\pmb{\mathcal{G}}$}
        \STATE{$\tsr{G}\leftarrow\pmb{\mathcal{G}}_\text{new}$}
    \ENDWHILE
  \ENDIF
  \STATE{Perform regular ALS sweep as in Algorithm~\ref{alg:tucker-als}, taking $d\mat{A}^{(n)} \gets \mat{A}^{(n)}_\text{new}-\mat{A}^{(n)}$ for each $n\in\inti{1}{N}$}
        \STATE{$\tsr{G}_\text{new}\leftarrow
  \tsr{Y}^{(N)}\times_{N}\mat{A}^{(N)T}$}
        \STATE{$\tsr{F}\leftarrow\pmb{\mathcal{G}}_\text{new}-\pmb{\mathcal{G}}$}
        \STATE{$\tsr{G}\leftarrow\pmb{\mathcal{G}}_\text{new}$}
\ENDWHILE
\RETURN $[\![ \tsr{G}; \mat{A}^{(1)}, \ldots , \mat{A}^{(N)} ]\!]$
\end{algorithmic}
\end{algorithm}

\subsubsection{Dimension Trees for Pairwise Perturbation Operators}
\label{subsec:cost}

Computation of the pairwise perturbation operators $\tsr{M}^{(i,n)}_p$ and of $\tsr{M}^{(n)}_p$ can benefit from amortization of common tensor contraction (Khatri-Rao product or multilinear multiplication) subexpressions.
In the context of ALS, this technique is known as dimension trees and has been successfully employed to accelerate TTMc and MTTKRP.
The same trees can be used for both CP and Tucker, although the tensor intermediates and contraction operations are different (Khatri-Rao products for CP and multilinear multiplication for Tucker).
We describe the trees for CP decomposition, computing each $\tsr{M}^{(i,n)}_p$ and $\tsr{M}^{(n)}_p$.
Figure~\ref{fig:pp_cp3},\ref{fig:pp_cp4} describes the dimension tree for $N=3,4$.
Our tree constructions assume that the tensors are equidimensional, if this is not the case, the largest dimensions should be contracted first.

The main goal of the dimension tree is to perform a minimal number of contractions to obtain each $\tsr{M}^{(i,n)}_p$.
Each matrix $\tsr{M}^{(n)}_p$ can be simply obtained by a contraction with $\tsr{M}^{(i,n)}_p$ for any $i\neq n$.
Each level of the tree for $l=1,\ldots, N-1$ should contain intermediate tensors containing $N-l+1$ uncontracted modes belonging to the original tensor (the root is the original tensor $\tsr{X}=\tsr{M}^{(1,\ldots, N)}$).
For any pair of the original tensor modes, each level should contain an intermediate for which these modes are uncontracted.
Since the leaves at level $l=N-1$ have two uncontracted modes, they will include each $\tsr{M}^{(i,n)}_p$ for $i<n$ 
and have ${N \choose 2 }$ tensors overall.
At level $l$ it then suffices to compute ${l+1 \choose 2}$ tensors $\tsr{M}^{(i,j,l+2,l+3,\ldots, N)}, \forall i,j\in \inti{1}{l+1}, i< j$.
Each $\tsr{M}^{(i,j,l+2,l+3,\ldots, N)}$ can be computed by contraction of $\tsr{M}^{(s,t,v,l+2,l+3,\ldots, N)}$ and $\mat{A}^{(w)}$ where $\{s,t,v\}=\{i,j,w\}$ with $w =\max_{w\in \{l-1,l,l+1\}\setminus\{i,j\}}(w)$ and $s<t<v$.

\begin{table}[!htb]
\caption{Cost comparison between pairwise perturbation algorithm and ALS dimension tree algorithm for CP and Tucker decompositions.} 
\label{table:compare}
\centering
\begin{adjustbox}{width=0.55\textwidth}
\begin{tabular}{|c|c|c|c|}
  \hline
 & DT ALS
 & PP initialization step
 & PP approximated step
 \\ 
  \hline
CP & $4s^NR$ & $6s^NR$ & {$2N^2(s^2R+sR^2)$} \\ 
  \hline
  Tucker & $4s^NR$ 
  & $6s^NR$
  & $2N^2s^2R^{N-1}$  \\ 
   \hline
\end{tabular}
\end{adjustbox}
\end{table} 
 
The construction of pairwise perturbation operators for CP decomposition costs
\begin{align*} 
2R\sum_{l=2}^{N-1}{l+1 \choose 2}s^{N-l+2} = 6s^NR+12s^{N-1}R+O\left(s^{N-2}R^2\right).
\end{align*}
The cost to form pairwise perturbation operators for Tucker decomposition is
\begin{align*} 
2\sum_{l=2}^{N-1}{l+1 \choose 2}s^{N-l+2}R^{l-1}
= 6s^NR + 12s^{N-1}R^{2} + O\left(s^{N-2}R^3\right).
\end{align*}
We summarize the leading order computational costs for both algorithms in Table~\ref{table:compare}. The PP initialization step, which involves the PP operator construction and does one more first level contraction, is computationally 1.5X more expensive than the ALS algorithm.

As for the memory footprint,
ALS with the best choice of dimension tree requires intermediate tensors of size $O\left(s^{\lceil N/2\rceil}R\right)$. As an example, for the order four case shown in Figure~\ref{fig:als_cp4}, the first and second level contractions are combined to save memory, so that $\tsr{M}^{(3,4)}$ and $\tsr{M}^{(1,2)}$ are stored, both of size $O\left(s^{2}R\right)$. The PP dimension tree described above and in Figure~\ref{fig:pp_cp4} needs at least $O\left(s^{N-1}R\right)$ auxiliary memory to store the first level contraction results. The memory needed for PP can be reduced similar to ALS. 
For example, when calculating the PP operator $\tsr{M}_p^{(1,3)}$ for an order four tensor, we can 
bypass the first level contraction and save its memory via directly performing a contraction between the input tensor and the Khatri-Rao product output $\mat{A}^{(1)} \odot \mat{A}^{(3)}$.
Combining the first $l\leq N-2$ levels of contractions requires $O\left(s^{N-l}R+N^2s^2R\right)$ auxiliary memory, but incurs a cost of $O\left(l^2s^{N-1}R\right)$.


\section{Error Analysis}
\label{sec:error}
In this section, we formally bound the approximation error of the pairwise perturbation algorithm relative to ALS.
We show that quadratic optimization problems computed by pairwise perturbation differ only slightly from ALS so long as the factor matrices have not changed significantly since the construction of the pairwise perturbation operators.

\subsection{CP-ALS}
\label{sec:err-cp}

\begin{algorithm}
    \caption{\textbf{CP-ALS}: Reinterpreted ALS procedure for CP decomposition}
\label{alg:cp_als2}
\begin{algorithmic}[1]
\small
\STATE{\textbf{Input: }Tensor $\tsr{X}\in\mathbb{R}^{s_1\times\cdots\times s_N}$, 
stopping criteria \modify{$\Delta$}}
\STATE{Initialize $[\![ \textbf{A}^{(1)}, \ldots , \textbf{A}^{(N)} ]\!]$ as uniformly distributed random matrices within $[0,1]$, initialize $\mat{G}^{(n)},\delta\mat{A}^{(n)}\leftarrow \mat{A}^{(n)}$, $\mat{S}^{(n)}\leftarrow \mat{A}^{(n)T}\mat{A}^{(n)}$ for $i\in\{1,\ldots,N\}$}
\FOR{\texttt{$n\in \inti{1}{N} $}}
\STATE{Update $ \textbf{\M}^{(n)}$ based on the dimension tree algorithm shown in Figure~\ref{fig:dt}}
\ENDFOR
\WHILE{$\sum_{i=1}^{N}{\fnrm{\mat{G}^{(i)}}}>\modify{\Delta\|\tsr{X}\|_F}$}
\FOR{\texttt{$n\in \inti{1}{N} $}}
\STATE{$\boldsymbol{\Gamma}^{(n)}\leftarrow\textbf{S}^{(1)}\ast\cdots\ast \textbf{S}^{(n-1)}\ast \textbf{S}^{(n+1)}\ast\cdots\ast \textbf{S}^{(N)} $}
\STATE{$   \textbf{A}^{(n)}_\text{new} \leftarrow \textbf{\M}^{(n)}\boldsymbol{\Gamma}^{(n)}{}^{\dagger} $}
\STATE{$   \delta \mat{A}^{(n)} = \textbf{A}^{(n)}_\text{new} - \textbf{A}^{(n)}$}
\STATE{$   \mat{G}^{(n)}\leftarrow -\delta \mat{A}^{(n)}\mat{\Gamma}^{(n)}$}
\STATE{$   \textbf{A}^{(n)} \leftarrow\textbf{A}^{(n)}_\text{new}$}
\vspace{.01in}
\STATE{$   \textbf{S}^{(n)} \leftarrow {\textbf{A}^{(n)}{}^T}\textbf{A}^{(n)} $}
\FOR{\texttt{$m\in \inti{1}{N}, m\neq n $}}
    \STATE{Update $\mat{M}^{(m)}$ as $\mat{M}^{(m)}(x,k)
    =\mat{M}^{(m)}(x,k) + \sum_{y=1}^{s_n}\tsr{M}^{(m,n)}(x,y,k) \delta\mat{A}^{(n)}(y,k)$}
\ENDFOR
\ENDFOR
\ENDWHILE
\RETURN $[\![ \textbf{A}^{(1)}, \ldots , \textbf{A}^{(N)} ]\!]$
\end{algorithmic}
\end{algorithm}

To bound the error of pairwise perturbation, we view the ALS procedure for CP decomposition in terms of pairwise updates (Algorithm~\ref{alg:cp_als2}), pushing updates to least-squares problems of all tensors as soon as any one of them is updated.
This reformulation is algebraically equivalent to Algorithm~\ref{alg:cp_als}, but makes oracle-like use of $\tsr{M}^{(m,n)}$ (Equation~\ref{eq:tensors-cp}), recomputing which would increase the computational cost.
We can bound the error of the way pairwise perturbation propagates updates to any right-hand side $\mat{M}^{(m)}$ due to changes in any one of the other factor matrices $\delta\mat{A}^{(n)}$. We define the update $\mat{H}^{(m,n)}$ in terms of its columns,
\[
\vcr{h}^{(m,n)}_k(x) =
 \sum_{y=1}^{s_n} \tsr{M}^{(m,n)}(x,y,k) \delta{\mat{A}^{(n)}(y,k)}, \quad \text{where} \quad \delta \mat{A}^{(n)} = \mat{A}_\text{new}^{(n)}-\mat{A}^{(n)}.
\]
{Note that $\delta \mat{A}^{(n)}$ denotes the update of $n$th factor between two neighboring sweeps, which should be distinguished from $d\mat{A}^{(n)}$, denoting the perturbation of $n$th factor in PP.}
Based on the definition, the update of each $\mat{M}^{(m)}$ after an ALS sweep is the summation of $\mat{H}^{(m,n)}$ expressed as 
$
{\delta}\mat{M}^{(m)} = \sum_{n=1,n\neq m}^{N}\mat{H}^{(m,n)}. 
$

For simplicity, we first perform an error analysis for the case where the second order correction terms $\mat{V}^{(n,i,j)}$ are not included in PP. 
In Theorem~\ref{thm:cpgenerrbound}, we prove that when the column-wise norm of $d\mat{A}^{(n)}=\mat{A}^{(n)}-\mat{A}_{p}^{(n)}$ relative to the norm of $\mat{A}^{(n)}$ for $n\in\inti{1}{N}$ is small, the absolute error of column-wise results for $\mat{H}^{(m,n)}$ calculated from pairwise perturbation with respect to that calculated from exact ALS is also small. Corollary~\ref{thm:cperrbound_N3} provides a simple relative error bound for third-order tensors. 
Overall, these bounds demonstrate that pairwise perturbation should generally compute updates with small relative error with respect to the magnitude of the perturbation of the factor matrices since the setup of the pairwise operators.
However, this relative error can be amplified during other steps of ALS, which are ill-conditioned, i.e., can suffer from catastrophic cancellation (the same would hold for round-off error).

We then perform an error analysis for the case where the second order correction terms $\mat{V}^{(n,i,j)}$ are included in PP in Theorem~\ref{thm:cpgenerrbound_withsecond}. 
We show that the second order corrections can tighten the leading order error by a factor related to the CP decomposition accuracy.


\begin{thm}
\label{thm:cpgenerrbound}
 For $k \in \{1,\ldots,R\}$, if $\vnrm{d\vcr{a}^{(l)}_k}/\vnrm{\vcr{a}^{(l)}_k}\leq
 \epsilon<1$ for all $l\in\{1,\ldots, N\}$,
the pairwise perturbation algorithm without second order corrections computes the update $\Tilde{\mat{H}}^{(1,N)}$ with columnwise error,
\begin{align*}
\vnrm{\Tilde{\vcr{h}}_k^{(1,N)}-\vcr{h}_k^{(1,N)}}  =& O(N\epsilon)\tnrm{{\tsr{\hat{T}}}}\prod_{j=2}^{N-1}\vnrm{\vcr{a}^{(j)}_k},
\end{align*}
where $\mat{H}^{(1,N)}$ is the update to the matrix $\mat{M}^{(1)}$ due to the change $\delta\mat{A}^{(N)}$ performed by a regular ALS sweep,  and $ \tsr{\hat{T}}=\tsr{X} \times_N \delta\mat{a}_k^{(N)T}$.
Analogous bounds hold for $\mat{H}^{(m,n)}$ for any $m,n\in\{1,\ldots, N\}$, $m\neq n$.
\end{thm}
\begin{proof}
The ALS update and approximated update are
\begin{equation}
    \vcr{h}^{(1,N)}_k =   {\tsr{\hat{T}}} \bigtimes_{i\in \{2,\ldots,N-1\}} \vcr{a}_k^{(i)T} \quad \text{and} \quad \Tilde{\vcr{h}}^{(1,N)}_k =   {\tsr{\hat{T}}} \bigtimes_{i\in \{2,\ldots,N-1\}} \left(\vcr{a}_k^{(i)T}-d\vcr{a}_k^{(i)T}\right).
\end{equation}
We can expand the error as
\begin{equation}\label{eq:err_expression_cpbound}
 \Tilde{\vcr{h}}_k^{(1,N)}{-} \vcr{h}^{(1,N)}_k{=} \sum_{S\subset \{2,\ldots,N-1\},S\neq \emptyset}{\tsr{\hat{T}}}  \bigtimes_{i\in \{2,\ldots,N-1\}} \vcr{v}_k^{(i)T}, \text{ where } \vcr{v}_k^{(i)}=\begin{cases} -d\vcr{a}_k^{(i)} &{:} i \in S \\ \vcr{a}_k^{(i)} &{:} i \notin S \end{cases}.   
\end{equation}
Consequently, we can upper-bound the error due to terms with $|S|=d$ by
\[
{N-2 \choose d} \epsilon^d \tnrm{{\tsr{\hat{T}}}}\prod_{j=2}^{N-1}\vnrm{\vcr{a}^{(j)}_k}
=
O(N\epsilon)^d \tnrm{{\tsr{\hat{T}}}}\prod_{j=2}^{N-1}\vnrm{\vcr{a}^{(j)}_k}.
\]
Therefore, the error bound when $|S|=d$ scales as $O(N\epsilon)^d$, and the leading order error is $O(N\epsilon)$.
\end{proof}
Note that this error bound involves $\tsr{\hat{T}}$, which is small in norm due to being constructed from contraction with $\delta\vcr{a}_k^{(N)}$.
Thus, the error norm generally scales as $O(\epsilon^2)$ relative to the norm of the original tensor $\tsr{X}$, since $O(N\epsilon)\tnrm{{\tsr{\hat{T}}}}\prod_{j=2}^{N-1}\vnrm{\vcr{a}^{(j)}_k}
=O(N\epsilon^2)\tnrm{{\tsr{X}}}\prod_{j=2}^{N-1}\vnrm{\vcr{a}^{(j)}_k}$.

\begin{cor}
\label{thm:cperrbound_N3}
For $N=3$, using the bounds from the proof of Theorem~\ref{thm:cpgenerrbound}, under the same assumptions, we obtain the absolute error bound,
\begin{align*}
\vnrm{\Tilde{\vcr{h}}_k^{(1,3)}-\vcr{h}_k^{(1,3)}}  \leq \tnrm{\mat{\hat{T}}}\vnrm{\vcr{a}^{(2)}_k}\epsilon,
\end{align*}
where $\mat{\hat{T}}=\tsr{X} \times_{3} \delta\mat{a}_k^{{(3)T}}$. Further, since $\vcr{h}_k^{(1,3)}=\mat{\hat{T}}\vcr{a}^{(2)}_k$, the relative error is bounded by
\begin{align*}
\frac{\vnrm{\Tilde{\vcr{h}}_k^{(1,3)}-\vcr{h}_k^{(1,3)}}}{\vnrm{\vcr{h}_k^{(1,3)}}}  \leq \kappa(\mat{\hat{T}})\epsilon.
\end{align*}
\end{cor}
From Theorem~\ref{thm:cpgenerrbound}, we can conclude that the relative error in computing any column update $\mat{h}^{(i,j)}_k$ is $O(\epsilon)$ when $\epsilon\ll 1$ and the correct update is sufficiently large, e.g., for $i=1$ and $j=N$, $\vnrm{\vcr{h}^{(1,N)}_k}=\Omega\Big(\tnrm{{\tsr{\hat{T}}}}\prod_{i=2}^{N-1}\vnrm{\vcr{a}^{(l)}_k}\Big)$.
When this is the case, we can also bound the error of the update to the columns of the right-hand sides ${\delta}\mat{M}^{(n)}$ formed in ALS, so long as the sum of the updates $\mat{H}^{(n,m)}$ for $m\neq n$ is not too small in norm relative to each update matrix.

We now perform analysis for the case where the second order corrections $\mat{V}^{(n,i,j)}$ are included in PP.
\begin{thm}
\label{thm:cpgenerrbound_withsecond}
 For $k \in \{1,\ldots,R\}$, if $\vnrm{d\vcr{a}^{(l)}_k}/\vnrm{\vcr{a}^{(l)}_k}\leq
 \epsilon<1$ for all $l\in\{1,\ldots, N\}$,
the pairwise perturbation algorithm with second order correction terms computes the update term $\Tilde{\mat{H}}^{(1,N)}$ with columnwise error,
\begin{align*}
\vnrm{\Tilde{\vcr{h}}_k^{(1,N)}-\vcr{h}_k^{(1,N)}}  =&
O(N\epsilon) \tnrm{\tsr{\hat{P}} - \tsr{\hat{T}}}\prod_{j=2}^{N-1}\vnrm{\vcr{a}^{(j)}_k} + O\left( (N\epsilon)^2\right)\tnrm{{\tsr{\hat{T}}}}\prod_{j=2}^{N-1}\vnrm{\vcr{a}^{(j)}_k} ,
\end{align*}
where $\tsr{\hat{P}}=\tsr{Z} \times_N \delta\mat{a}_k^{(N)T}$, and $\tsr{Z}$ denotes the approximate CP decomposition of $\tsr{X}$, $\cpbrak{\mat{A}^{(1)}, \ldots, \mat{A}^{(N)}}$.
$\mat{H}^{(1,N)}$ is the update to the matrix $\mat{M}^{(1)}$ due to the change $\delta\mat{A}^{(N)}$ performed by a regular ALS sweep,  and $ \tsr{\hat{T}}=\tsr{X} \times_N \delta\mat{a}_k^{(N)T}$.
Analogous bounds hold for $\mat{H}^{(m,n)}$ for any $m,n\in\{1,\ldots, N\}$, $m\neq n$.
\end{thm}
\begin{proof}
The ALS approximated update is
\begin{equation}
\Tilde{\vcr{h}}^{(1,N)}_k =   \tsr{\hat{T}} \bigtimes_{i\in \{2,\ldots,N-1\}} \left(\vcr{a}_k^{(i)T}-d\vcr{a}_k^{(i)T}\right)
+
\sum_{i\in\{2,\ldots,N-1\}}\tsr{\hat{P}} \times_i d\vcr{a}_k^{(i)T} \bigtimes_{j\in \{2,\ldots,N-1\}, j\neq i} \vcr{a}_k^{(j)T}.
\end{equation}
We can expand the error as
{\small
\begin{align*}
   \Tilde{\vcr{h}}_k^{(1,N)}{-} \vcr{h}^{(1,N)}_k &= \sum_{S\subset \{2,\ldots,N-1\},|S|\geq 2}{\tsr{\hat{T}}}  \bigtimes_{i\in \{2,\ldots,N-1\}} \vcr{v}_k^{(i)T} + 
\sum_{i\in\{2,\ldots,N-1\}}\left(\tsr{\hat{P}} - \tsr{\hat{T}}\right) \times_i d\vcr{a}_k^{(i)T} \bigtimes_{j\in \{2,\ldots,N-1\}, j\neq i} \vcr{a}_k^{(j)T}
,
\end{align*}
}
where $\vcr{v}_k^{(i)}=-d\vcr{a}_k^{(i)}$ if $i \in S$ and $\vcr{v}_k^{(i)}=\vcr{a}_k^{(i)}$ otherwise. By the same analysis as in Theorem~\ref{thm:cpgenerrbound}, the error due to each term with $|S|=d$, $d\geq 2$ can be bounded as $O(N\epsilon)^d \tnrm{{\tsr{\hat{T}}}}\prod_{j=2}^{N-1}\vnrm{\vcr{a}^{(j)}_k}.$ We can then upper-bound the error due to the second term by
\begin{equation}\label{eq:second_order_tern_cp_thm}
    O(N\epsilon) \tnrm{\tsr{\hat{P}} - \tsr{\hat{T}}}\prod_{j=2}^{N-1}\vnrm{\vcr{a}^{(j)}_k},
\end{equation}
thus completing the proof.
\end{proof}
From Theorem~\ref{thm:cpgenerrbound_withsecond}, we can conclude that when the approximate CP decomposition is close to $\tsr{X}$, the term expressed in Equation \ref{eq:second_order_tern_cp_thm} will have small magnitude, making the absolute error second order accurate in terms of $\epsilon$. 

In Appendix~\ref{subsec:cp_bound_cond}, we also obtain relative error bounds on MTTKRPs (the right-hand sides in the linear least squares subproblems). 
However, this error bound is relative to the condition number of $\tsr{X}$ (defined in Appendix~\ref{subsec:cond}), which is infinite for sufficiently large tensors.

\subsection{Tucker-ALS}
\label{sec:error-tucker}
For Tucker decomposition, the pairwise perturbation approximation satisfies better bounds than for CP decomposition, due to the orthogonality of the factor matrices. 
We can not only obtain the similar bound as Theorem~\ref{thm:cpgenerrbound}, but also obtain stronger results assuming that either the residual of the Tucker decomposition is bounded (it suffices that the decomposition achieves one digit of accuracy in residual) or that the ratio of rank to dimension is not too large.
We demonstrate that
\begin{itemize}[leftmargin=*]
\item {similar to Algorithm~\ref{alg:cp_als2} and Theorem~\ref{thm:cpgenerrbound}, when we view the ALS procedure for Tucker decomposition of equidimensional tensors in terms of pairwise updates, we can bound the error of updates to any right-hand side $\tsr{Y}^{(m)}$ due to changes in any one of the other factor matrices $\delta\mat{A}^{(n)}$. We define the update $\tsr{J}^{(m,n)}$ as
\[
\tsr{J}^{(m,n)} =
 \tsr{Y}^{(m,n)} \times_n \delta\mat{A}^{(n){T}}, \quad \text{where} \quad \delta \mat{A}^{(n)} = \mat{A}_\text{new}^{(n)}-\mat{A}^{(n)}.
\]
The columnwise absolute error bound for MTTKRP holds for $\tsr{J}^{(m,n)}$ when the column-wise 2-norm relative perturbations of the input matrices are bounded by $O(\epsilon)$ (Theorem~\ref{thm:tuckergenerrbound}),}
\item{the relative error of $\tsr{Y}^{(m)}$ for $m\in \inti{1}{N}$ satisfies the bound of $O\left(\epsilon^2\right)$, so long as the residual of Tucker decomposition is small (Theorem~\ref{thm:tuckerbound4}), }
\item{the relative error of $\tsr{Y}^{(m)}$ for $m\in \inti{1}{N}$ is bounded in Frobenius norm  by $O\left(\epsilon^2\right)$ for a fixed problem size assuming that HOSVD is performed to initialize Tucker-ALS (Theorem~\ref{thm:tuckerboundf})}. 
\end{itemize}

\begin{thm}
\label{thm:tuckergenerrbound}
For an order $N$ tensor $\tsr{X}$ with dimension sizes $s$, if 
$\vnrm{d\vcr{a}^{(n)}_k}$
$/\vnrm{\vcr{a}^{(n)}_k}\leq\epsilon<1$ for all $n\in\{1,\ldots, N\}, k \in \inti{1}{R}$, the pairwise perturbation algorithm computes update $\tsr{J}^{(1,N)}$ with error, 
{
\begin{align*}
\Big \|\vcr{\Tilde{j}}_{i_2,\ldots,i_{N}}^{(1,N)}-\vcr{j}_{i_2,\ldots,i_{N}}^{(1,N)}\Big\|_2
=& O(N\epsilon)\tnrm{{\tsr{\hat{T}}}}\prod_{j=2}^{N-1}\vnrm{\vcr{a}^{(j)}_k},
\end{align*}
}
where 
$\tsr{\hat{T}}=\tsr{X} \times_N \delta\mat{a}_{i_N}^{(N){T}}$ and $\vcr{j}^{(1,N)}_{i_2,\ldots,i_{N}}(x) = \tsr{J}^{(1,N)}(x,i_2,\ldots,i_N)$. 
\end{thm}
\begin{proof}
The proof is similar to that of Theorem~\ref{thm:cpgenerrbound}. The ALS update and approximated update after a change $\delta \mat{A}^{(N)}$ are
\[\vcr{j}^{(1,N)}_{i_2,\ldots,i_{N}} = {\tsr{\hat{T}}} \bigtimes_{j=2}^{N-1} \vcr{a}_{i_j}^{(j){T}} \quad \text{and} \quad 
  \vcr{\Tilde{j}}^{(1,N)}_{i_2,\ldots,i_{N}} = {\tsr{\hat{T}}} \bigtimes_{j=2}^{N-1} \left(\vcr{a}_{i_j}^{(j){T}} - d\vcr{a}_{i_j}^{(j){T}}\right).\]
The error bound proceeds by analogy to the proof of Theorem~\ref{thm:cpgenerrbound}.
\end{proof}

Using Lemma~\ref{lem:norm}, we prove in Lemma~\ref{lem:normequal} that after contracting a tensor with a matrix with orthonormal columns, whose row length is higher or equal to the column length, the contracted tensor norm is the same as the original tensor norm.
\begin{lem}
\label{lem:normequal}
Given tensor $\tsr{G}\in\mathbb{R}^{r_1\times\cdots\times r_N}$, the mode-$n$ product for any $n\in\{1,\ldots , N\}$, with a matrix with orthonormal columns $\mat{M} \in \mathbb{R}^{s\times r_n}$, $r_n\leq s$, satisfies $\tnrm{\tsr{G}}= \tnrm{\tsr{G}\times_n\mat{M}}$.
\end{lem}
\begin{proof}
Based on the submultiplicative property of the tensor norm (Lemma~\ref{lem:norm}),
\begin{align*}
&\tnrm{\tsr{G}} = \tnrm{\tsr{G}\times_n(\mat{M}^T\mat{M})} = \tnrm{\tsr{G}\times_n\mat{M}\times_n\mat{M}^T} \leq \tnrm{\tsr{G}\times_n\mat{M}}\tnrm{\mat{M}^T} {=} \tnrm{\tsr{G}\times_n\mat{M}}, 
\end{align*}
and simultaneously,
$
\tnrm{\tsr{G}\times_n\mat{M}} \leq \tnrm{\tsr{G}}\tnrm{\mat{M}} {=} \tnrm{\tsr{G}}.$
\end{proof}

Using Lemma~\ref{lem:normequal}, we prove in Theorem~\ref{thm:tuckerbound4} that when the relative error of  the matrices $\mat{A}^{(n)}$ for $n\in\inti{1}{N}$ is small and the residual of the Tucker decomposition is loosely bounded, the relative error bound for the $\tsr{Y}^{(n)}$ is independent of the tensor condition number defined in Section~\ref{sec:cond_glb}.
\begin{thm} 
\label{thm:tuckerbound4}
Given tensor $\tsr{X}\in\mathbb{R}^{s_1\times\cdots\times s_N}$, if $\tnrm{d\mat{A}^{(n)}}\leq\epsilon\ll 1 \text{ for } n\in\inti{1}{N}$ and \\ 
    $\tnrm{\tsr{X} - \cpbrak{\tsr{G};\mat{A}^{(1)},\mat{A}^{(2)},\ldots,\mat{A}^{(N)}}}\leq \frac{1}{3}\tnrm{\tsr{X}}$, $\Tilde{\tsr{Y}}^{(n)}$ is constructed with error,
\[
\frac{\tnrm{\Tilde{\tsr{Y}}^{(n)}-\tsr{Y}^{(n)}}}{\tnrm{\tsr{Y}^{(n)}}} = O\left(\epsilon^2\right).
\]
\end{thm} 
\begin{proof}
\begin{align*}
\frac{\tnrm{\Tilde{\tsr{Y}}^{(n)}-\tsr{Y}^{(n)}}}{\tnrm{\tsr{Y}^{(n)}}} 
&{\leq} {N \choose 2}\max_{i,j}\frac{\tnrm{\tsr{Y}_p^{(i,j,n)}\times_id\mat{A}^{(i)T}\times_jd\mat{A}^{(j)T}}}{\tnrm{\tsr{Y}^{(n)}}} \leq {N \choose 2}\max_{i,j}\frac{ \tnrm{\tsr{Y}_p^{(i,j,n)}} \tnrm{d\mat{A}^{(i)}} \tnrm{d\mat{A}^{(j)}}}{\tnrm{\tsr{Y}^{(n)}}}  .
\end{align*}
Let $\Tilde{\tsr{X}}=\cpbrak{\tsr{G};\mat{A}^{(1)},\mat{A}^{(2)},\ldots,\mat{A}^{(N)}}, \tsr{R} = \tsr{X}-\Tilde{\tsr{X}}$. 
Define the tensors $\tsr{Z}^{(i,j,n)}$ by contraction of $\tsr{R}$ with all except three factor matrices,
 \begin{align*}
\tsr{Z}^{(i,j,n)}
    =\tsr{R}\bigtimes_{r\in \inti{1}{N}\setminus\{i,j,n\}} \mat{A}^{(r)}{}^T.
 \end{align*}
For $ \tnrm{\tsr{X} - \Tilde{\tsr{X}}}=\tnrm{\tsr{R}}\leq \frac{1}{3}\tnrm{\tsr{X}}, $ 
we have 
$
\frac{2}{3}\tnrm{\tsr{X}} 
\leq\tnrm{\Tilde{\tsr{X}}}
\leq \frac{4}{3}\tnrm{\tsr{X}} .
$ 
Based on Lemma~\ref{lem:normequal},
\begin{align*}
\tnrm{\tsr{Y}^{(n)}} &= \tnrm{\tsr{G}\times_n\mat{A}^{(n)}+\tsr{Z}^{(i,j,n)}\times_i\mat{A}^{(i)T}\times_j\mat{A}^{(j)T}} 
\geq \tnrm{\tsr{G}}-\tnrm{\tsr{Z}^{(i,j,n)}}\tnrm{\mat{A}^{(i)T}}\tnrm{\mat{A}^{(j)T}}
\\&\geq \tnrm{\tsr{G}}-\tnrm{\tsr{R}} \geq \frac{1}{3}\tnrm{\tsr{X}}.
\end{align*}
Additionally,
\begin{align*}
\tnrm{\tsr{Y}^{(i,j,n)}} = \tnrm{\tsr{G}\times_i\mat{A}^{(i)}\times_j\mat{A}^{(j)}\times_n\mat{A}^{(n)}+\tsr{Z}^{(i,j,n)}} \leq \tnrm{\tsr{G}}+\tnrm{\tsr{R}}\leq \frac{5}{3}\tnrm{\tsr{X}}.
\end{align*}
Therefore, 
\begin{align*}
\frac{\tnrm{\Tilde{\tsr{Y}}^{(n)}-\tsr{Y}^{(n)}}}{\tnrm{\tsr{Y}^{(n)}}} {\leq} {N \choose 2}\max_{i,j}\frac{ \tnrm{\tsr{Y}^{(i,j,n)}_p} \tnrm{d\mat{A}^{(i)}} \tnrm{d\mat{A}^{(j)}}}{\tnrm{\tsr{Y}^{(n)}}} 
{\leq}   {N \choose 2}\frac{\frac{5}{3}\tnrm{\tsr{X}}\epsilon^2}{\frac{1}{3}\tnrm{\tsr{X}}} {=} O\left(\epsilon^2\right).
\end{align*}
\end{proof}

We now derive a Frobenius norm error bound that is independent of residual norm and tensor condition number, 
and is based the ratio of the tensor dimensions and the Tucker rank.
We arrive at this result (Theorem~\ref{thm:tuckerboundf}) by obtaining a lower bound on the residual achieved by the HOSVD (Lemmas~\ref{lem:core1d} and \ref{lem:lowerbound}).
\begin{lem}
\label{lem:core1d}
Given tensor $\tsr{X}\in\mathbb{R}^{s_1\times\cdots\times s_N}$ and matrix $\mat{A}\in\mathbb{R}^{R\times s_n}$, where 
$R<\max\left\{s_n, \prod_{i=1, i\neq n}^N s_i\right\}$
and $\mat{A}$ consists of $R$ leading left singular vectors of $\mat{X}_{(n)}$. Let ${\tsr{Z}} = \tsr{X}\times_n \mat{A}$, $\fnrm{\tsr{X}} \geq {\fnrm{\tsr{Z}}} \geq \sqrt{\frac{R}{s_n}}\fnrm{\tsr{X}}$.
\end{lem}
\begin{proof}
The singular values of $\mat{A}\mat{X}_{(n)}$ are the first $R$ singular values of $\mat{X}_{(n)}$. 
Since the square of the Frobenius norm of a matrix is the sum of the squares of the singular values, $\fnrm{{\tsr{Z}}}^2 =\fnrm{\mat{A}\mat{X}_{(n)}}^2\geq \left(R/s_n\right)\fnrm{\mat{X}_{(n)}}^2= \left(R/s_n\right)\fnrm{\tsr{X}}^2$ and $\fnrm{{\tsr{Z}}}\leq \fnrm{\tsr{X}}$.
\end{proof}

\begin{lem}
\label{lem:lowerbound}
For any equidimensional order $N$ tensor $\tsr{X}$ with size $s$, $\fnrm{\tsr{Y}^{(n)}} \geq\left(\frac{R}{s}\right)^{N/2}\fnrm{\tsr{X}}$ if Tucker-ALS starts from an interlaced HOSVD. 
\end{lem}
\begin{proof}
In Tucker-ALS, $\fnrm{\tsr{G}}$ is strictly increasing after each Tucker iteration, where $\tsr{G}$ is $\tsr{X}$'s HOSVD core tensor. 
Since the interlaced SVD computes each $\mat{A}^{(n)}$ from the truncated SVD of the product of $\tsr{X}$ and the first $n-1$ factor matrices, we can apply Lemma~\ref{lem:core1d} $N$ times,
\begin{align*}
\fnrm{\tsr{X}\times_1\textbf{A}^{(1)\T}\cdots\times_{N-1}\textbf{A}^{(N-1)\T}} \geq
&\fnrm{\tsr{G}}
\geq \sqrt{\frac Rs}\fnrm{\tsr{X}\times_1\textbf{A}^{(1)\T}\cdots\times_{N-1}\textbf{A}^{(N-1)\T}} , \\
&\quad \vdots \\
\fnrm{\tsr{X}} \geq &\fnrm{\tsr{G}} \geq (R/s)^{N/2}\fnrm{\tsr{X}}.
\end{align*}
\end{proof}

\begin{thm}
\label{thm:tuckerboundf}
Given any equidimensional order $N$ tensor $\tsr{X}$ with size $s$, if 
$\fnrm{d\mat{A}^{(n)}}\leq\epsilon \text{ for } n\in [1,N]$, $\Tilde{\tsr{Y}}^{(n)}$ is constructed with error,
\[
\frac{\fnrm{\Tilde{\tsr{Y}}^{(n)}-\tsr{Y}^{(n)}}}{\fnrm{\tsr{Y}^{(n)}}} = O\Big(\epsilon^2\Big(\frac{s}{R}\Big)^{N/2}\Big),
\]
assuming that HOSVD is used to initialize Tucker-ALS and the residual associated with factor matrices $\mat{A}^{(1)},\ldots,\mat{A}^{(n)}$ is no higher than that attained by HOSVD.
\end{thm} 
\begin{proof}
\begin{align*}
\frac{\fnrm{\Tilde{\tsr{Y}}^{(n)}-\tsr{Y}^{(n)}}}{\fnrm{\tsr{Y}^{(n)}}} 
{\leq} {N \choose 2}\max_{i,j}\frac{\fnrm{\tsr{Y}_p^{(i,j,n)}\times_id\mat{A}^{(i)T}\times_jd\mat{A}^{(j)T}}}{\fnrm{\tsr{Y}^{(n)}}} .
\end{align*}
From Lemma~\ref{lem:lowerbound}, we have   
\begin{align*}
 \frac{\fnrm{\tsr{Y}^{(i,j,n)}_p\times_id\mat{A}^{(i)T}\times_jd\mat{A}^{(j)T}}}{\fnrm{\tsr{Y}^{(n)}}} 
\leq
\frac{\fnrm{\tsr{X}}\fnrm{d\mat{A}^{(i)}}\fnrm{d\mat{A}^{(j)}}}{(\frac{R}{s})^{N/2}\fnrm{\tsr{X}}} .
\end{align*}
Consequently, we can bound the relative error by
\[
\frac{\fnrm{\Tilde{\tsr{Y}}^{(n)}-\tsr{Y}^{(n)}}}{\fnrm{\tsr{Y}^{(n)}}} 
\leq {N \choose 2}(s/R)^{N/2}\max_{i,j} \fnrm{d\mat{A}^{(i)}}\fnrm{d\mat{A}^{(j)}}
=O\Big(\epsilon^2\Big(\frac{s}{R}\Big)^{N/2}\Big).
\]

\end{proof}


\section{Experiments}
\label{sec:exp}
We evaluate the performance of the pairwise perturbation algorithms on both synthetic tensors and application datasets. The synthetic experiments enable us to test tensors with known factors and to measure how effectively the algorithm works across many problem instances. We also consider publicly available tensor datasets as well as tensors of interest for quantum chemistry calculations and demonstrate the effectiveness of our algorithms on practical problems. We focus on the experiments on CP decomposition, because for many cases in Tucker decomposition, performing HOSVD and running CP decomposition on a much smaller core tensor is sufficient for getting accurate results. 

We use the metrics \textit{relative residual} and \textit{fitness} to evaluate the convergence of the decomposition. Let $\Tilde{\tsr{X}}$ denote the tensor reconstructed by the factor matrices and the core tensor, the relative residual and fitness are defined as follows,
\begin{align*}
    r = \frac{\|\tsr{X} - \Tilde{\tsr{X}}\|_F}{\|\tsr{X}\|_F}, \quad\quad
    f = 1 - r.
\end{align*}

{
We compare the performance of our own implementations of regular ALS with dimension trees to the pairwise perturbation algorithm. Both algorithms are implemented in Python with NumPy for sequential calculation and with Cyclops Tensor Framework (v1.5.5)~\cite{solomonik2014massively}, which is a distributed-memory library for matrix/tensor contractions that uses MPI for interprocessor communication and OpenMP for threading.
We also make use of a wrapper Cyclops provides for ScaLAPACK~\cite{Dongarra:1997:SUG:265932} routines to solve symmetric positive definite linear systems of equations and compute the SVD\footnote{All of our code is available at \url{https://github.com/LinjianMa/tensor_decompositions}.}.

The experimental results are collected on the Stampede2 supercomputer located at the University of Texas at Austin.
We leverage the Knight's Landing (KNL) nodes exclusively, each of which consists of 68 cores, 96 GB of DDR RAM, and 16 GB of MCDRAM.
These nodes are connected via a 100 Gb/sec fat-tree Omni-Path interconnect.
For both NumPy and Cyclops implementations,
we use Intel compilers and the MKL library for threaded BLAS routines, including batched BLAS routines, which are efficient for Khatri-Rao products arising in MTTKRP in CP decomposition, and employ the HPTT library~\cite{springer2017hptt} for high-performance tensor transposition.
All storage and computation assumes the tensors are dense.
}

\subsection{Sequential Experimental Results}
\label{sec:numpyexp}

{ 
We collect the sequential results on one KNL node on Stampede2, leveraging 64 threads for MKL and HPTT routines.

We compare the per-sweep time of the ALS dimension tree to the pairwise perturbation initialization and approximated sweep in Figure~\ref{fig:numpybench}. 
Each initialization sweep constructs the PP operators and updates all the factor matrices, while an approximated sweep computes approximate updates to all the factor matrices using the PP operators constructed in the last initialization sweep.
We also provide the reference per-sweep time of the ALS implementation from MATLAB Tensor Toolbox~\cite{kolda2006matlab}. As can be seen, both ALS sweep times on top of NumPy and Cyclops are comparable to the Tensor Toolbox. 
For both decompositions and all the configurations, the time of an PP initialization sweep is 1.5-2.0X the time of a dimension tree based ALS sweep, while the approximated steps can have up to $6.3$X speed-up for an order three tensor and $33.0$X speed-up for an order six tensor for CP, and up to $10.6$X speed-up for an order 6 tensor for Tucker. In addition, larger speed-up can be achieved with the increase of dimension size $s$ and the tensor order $N$, which is consistent with Table~\ref{table:compare}.

\begin{figure}[!ht]
\centering
\subfloat[CPD, $N=3$, $ R=s$]{\label{fig:bencha}\includegraphics[width=0.38\textwidth, keepaspectratio]{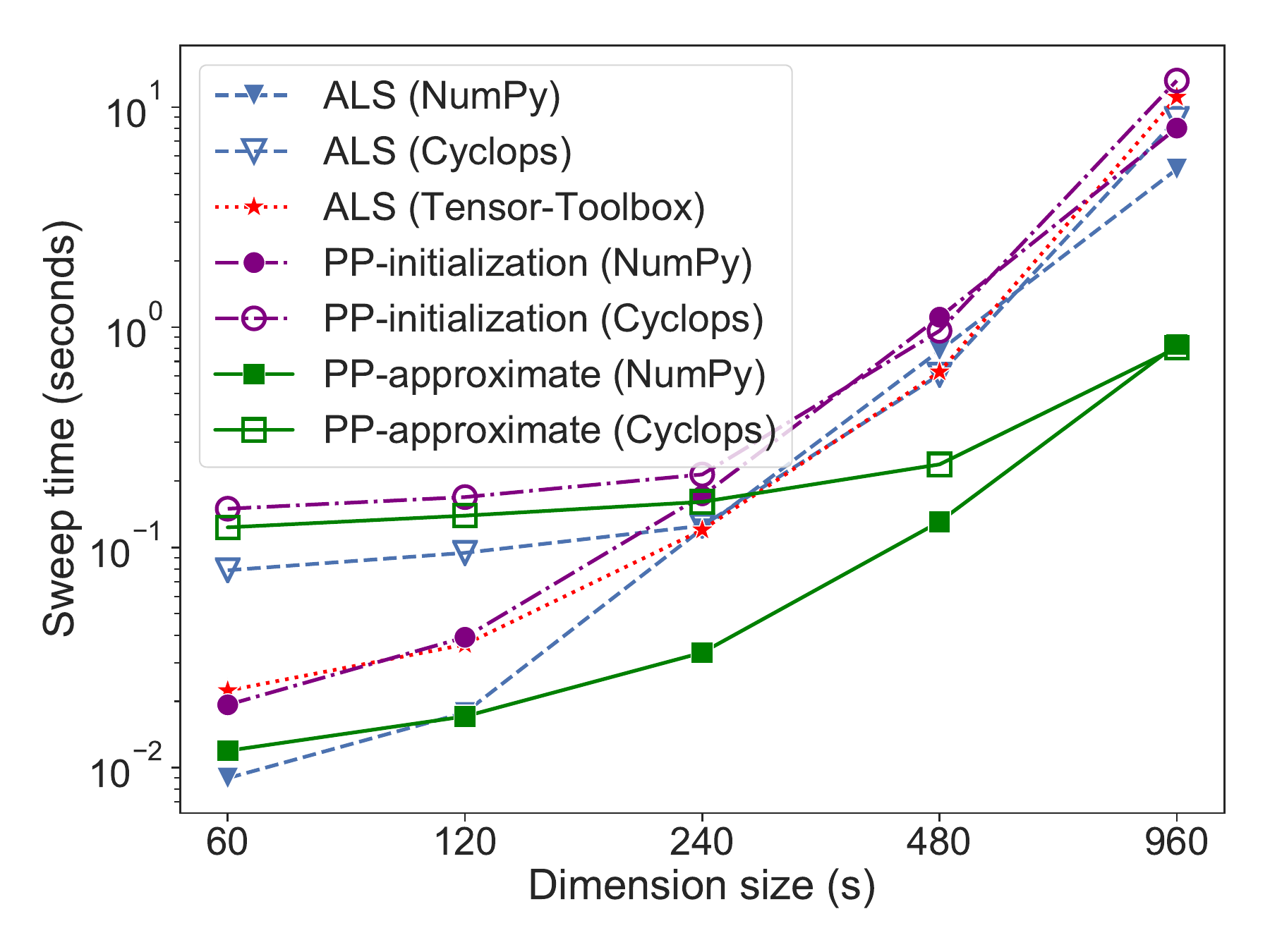}}
\subfloat[CPD, $s=960^{3/N}$, $ R=50$]{\label{fig:benchb}\includegraphics[width=0.38\textwidth, keepaspectratio]{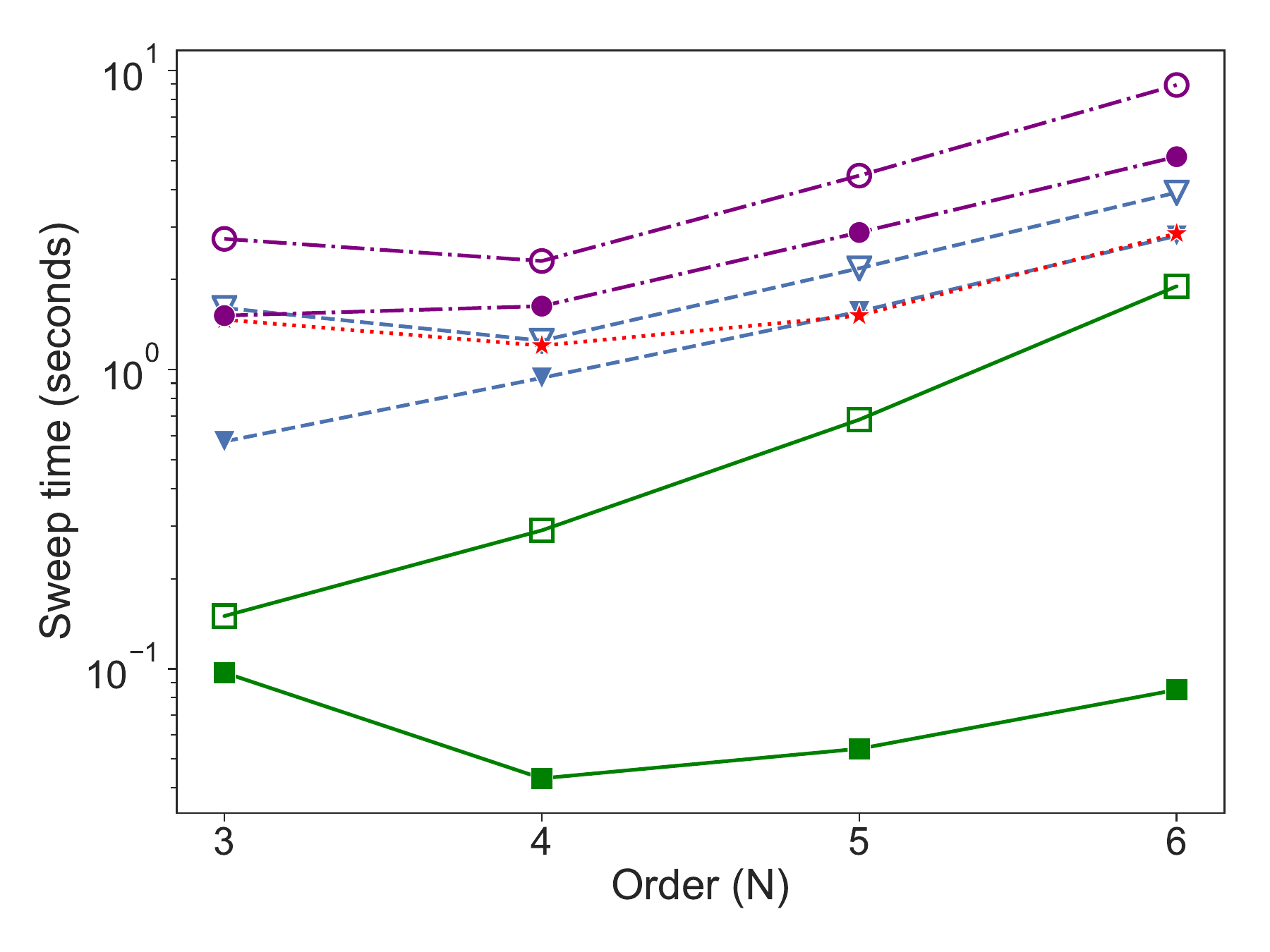}}

\subfloat[Tucker, $N=3$, $ R=0.05s$]{\label{fig:benchc}\includegraphics[width=0.38\textwidth, keepaspectratio]{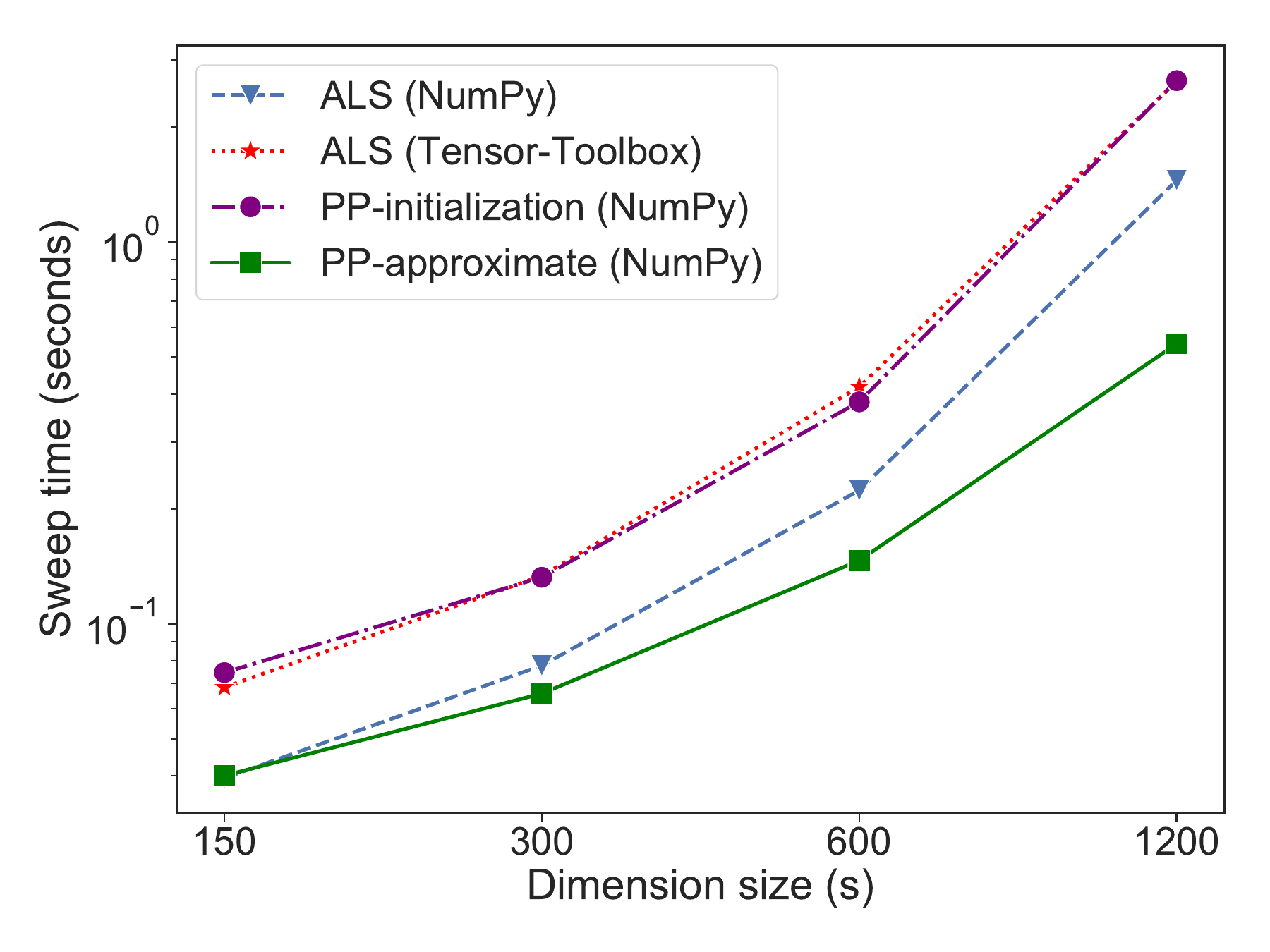}}
\subfloat[Tucker, $s=1200^{3/N}$, $ R=5$]{\label{fig:benchd}\includegraphics[width=0.38\textwidth, keepaspectratio]{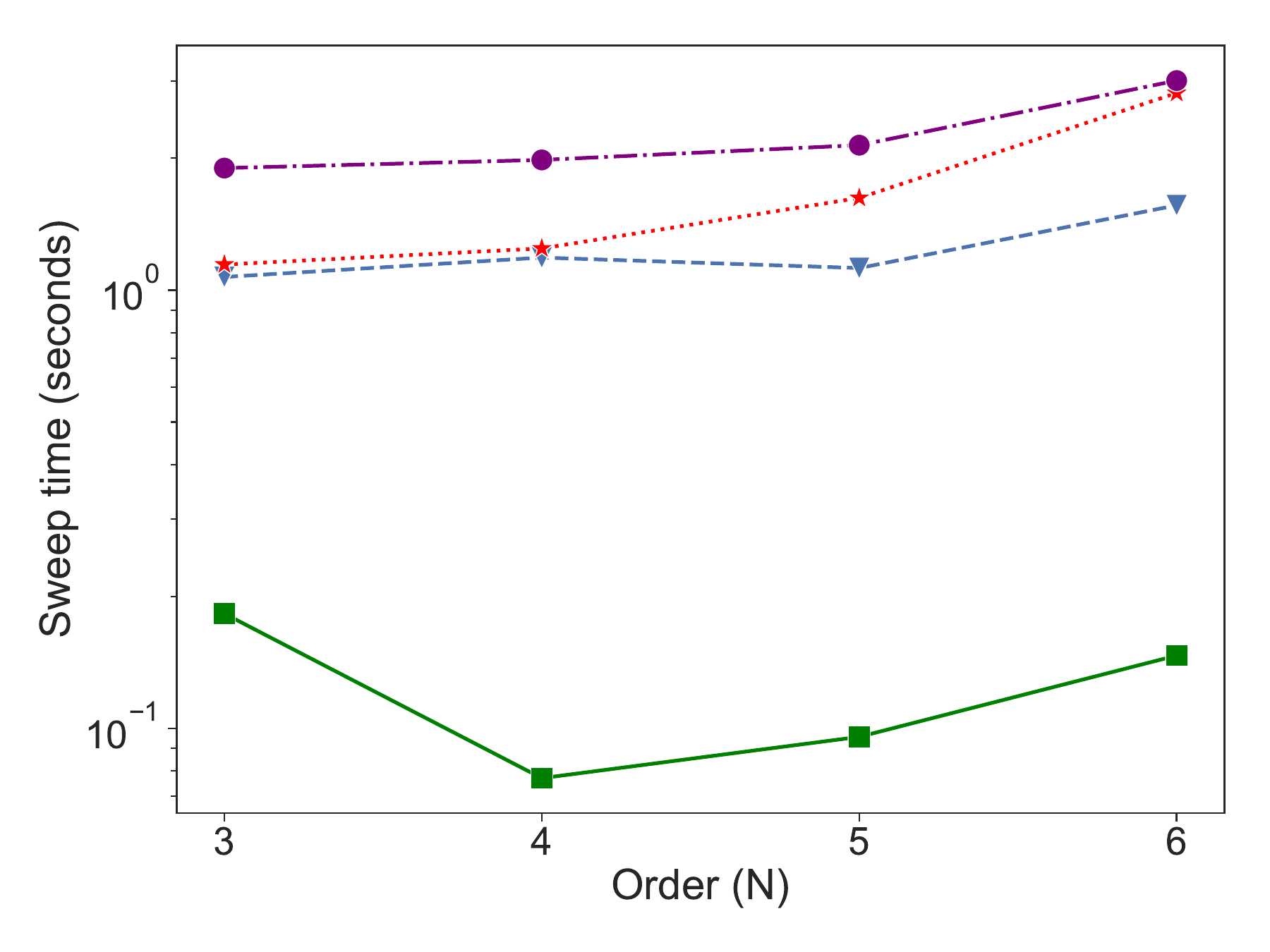}}
\caption{Sequential ALS sweep time comparison for both CP and Tucker decompositions. Results are taken as the mean time across 5 sweeps. The line label of (b) is the same as (a), of (d) is the same as (c). In \textbf{(a)(c)}, we vary the dimension size and the decomposition rank, and fix the input tensor order. In \textbf{(b)(d)}, we vary the input tensor order, and fix the input tensor size and the decomposition rank.
}
\label{fig:numpybench}
\end{figure}

We use five different tensors to test the sequential performance of pairwise perturbation. Sequential performance results are collected using NumPy, as NumPy has better sequential performance than Cyclops, as shown in Figures~\ref{fig:bencha} and \ref{fig:benchb}. For all the experiments, the pairwise perturbation tolerance is set as 0.1 for CP decomposition, and set as 0.3 for Tucker decomposition.
}
\begin{enumerate}[leftmargin=*]
    \item \label{tsr:col} \textbf{Tensors with random collinearity}~\cite{battaglino2017practical}. We create tensors based on known randomly-generated  factor matrices $\mat{A}^{(n)}$. The factor matrices $\mat{A}^{(n)} \in \mathbb{R}^{s\times R}$ are randomly generated so that the columns have collinearity defined based on a scalar $C$ (selected randomly for the tensor from a given interval $[a,b)$), so that
$$
\frac{\left\langle \mat{a}_i^{(n)}, \mat{a}_j^{(n)}\right\rangle}{\vnrm{\mat{a}_i^{(n)}}\vnrm{\mat{a}_j^{(n)}}}
=C, \quad \forall i,j \in \{1, \ldots, R\}, i\neq j. 
$$
Higher collinearity corresponds to greater overlap between columns within each factor matrix, which makes the convergence of CP-ALS slower~\cite{rajih2008enhanced}.

\item \label{tsr:random}\textbf{Tensors made by random matrices}. We create tensors based on known uniformly distributed randomly generated factor matrices $\mat{A}^{(n)}\in [0,1]^{s\times R}$, 
   \[
    \tsr{X} = \cpbrak{ \mat{A}^{(1)}, \cdots , \mat{A}^{(N)} }.
    \]
In the experiments, we set $R$ to be the same as the decomposition rank.

\item \label{tsr:quantum}\textbf{Quantum chemistry tensor.} We also test on the density fitting intermediate tensor arising in quantum chemistry, which is the Cholesky factor of the two-electron integral tensor~\cite{hohenstein:044103,hummel2017low}. For an order 4 two-electron integral tensor $\tsr{T}$, its Cholesky factor is an order 3 tensor $\tsr{D}$, with their relations shown as follows:
\[
\tsr{T}(a,b,c,d) = \sum_{s=1}^{P}\tsr{D}(a,b,s)\tsr{D}(c,d,s),
\]
where $P$ is the third mode dimension size of $\tsr{D}$.
CP decomposition can be performed on $\tsr{D}$ to provide the compressed form of the density fitting intermediate and can be used to speed up post Hartree-Fork calculations~\cite{hohenstein2012communication}. We generate the density fitting tensor via the PySCF library~\cite{sun2018pyscf}, which represents the compressed restricted Hartree-Fock wave function of an 8 water molecule chain system with a STO3G basis set. The generated tensor has size $904\times 56\times 56$. We set the CP rank to be 400.
\item \textbf{COIL dataset}. COIL-100 is an image-recognition data set that contains images of objects in different poses~\cite{nenecolumbia} and has been used previously as a tensor decomposition benchmark~\cite{battaglino2017practical,zhou2014decomposition}. There are 100 different object classes, each of which is imaged from 72 different angles. Each image has $128 \times 128$ pixels in three color channels. Transferring the data into tensor format, we have a $128\times 128\times 3\times 7200$ tensor. We fix the CP decomposition rank to be 15 and the Tucker decomposition rank to be $10\times 10\times 3\times 50$.
\item \textbf{Time-Lapse hyperspectral radiance images}. We consider the 3D hyperspectral imaging dataset called “Souto wood pile”~\cite{nascimento2016spatial}. The dataset is usually used on the benchmark of nonnegative tensor decomposition~\cite{liavas2017nesterov,ballard2018parallel}. The hyperspectral data consists of a tensor with dimensions $1024\times1344\times33\times 9$. We fix the CP decomposition rank to be 50 and the Tucker decomposition rank to be $100\times 100\times 3\times 3$.

\end{enumerate}

{The order three tensors are tested to justify the relative error bound shown in Section~\ref{sec:err-cp}. The performance of PP on higher order CP decompositions is also considered. We focus on the high rank CP decomposition, because for the cases with rank $R < s$, Tucker decomposition or HOSVD can be used to effectively compress the input tensor from dimensions of size $s$ to $R$, and then CP decomposition can be performed.}

We test the synthetic tensors for CP decomposition.
These tensors are all generated based on known factor matrices whose column sizes are equal to the decomposition rank, so these tensors have exact decompositions. For Tensor~\ref{tsr:col}, we test on {both order three tensors with both dimension sizes $s$ and decomposition rank $R$ equal to 400 and order four tensors with $s=R=120$}, and test the performance of pairwise perturbation on tensors with different collinearity for the exact input factor matrices. For Tensor~\ref{tsr:random}, we test on order three tensors with $s=R$, and test the performance of pairwise perturbation with different dimension size and corresponding rank. 

\begin{figure}[!ht]   
\centering

\subfloat[$N=3$, $s=R=400$]{\includegraphics[width=0.38\textwidth, keepaspectratio]{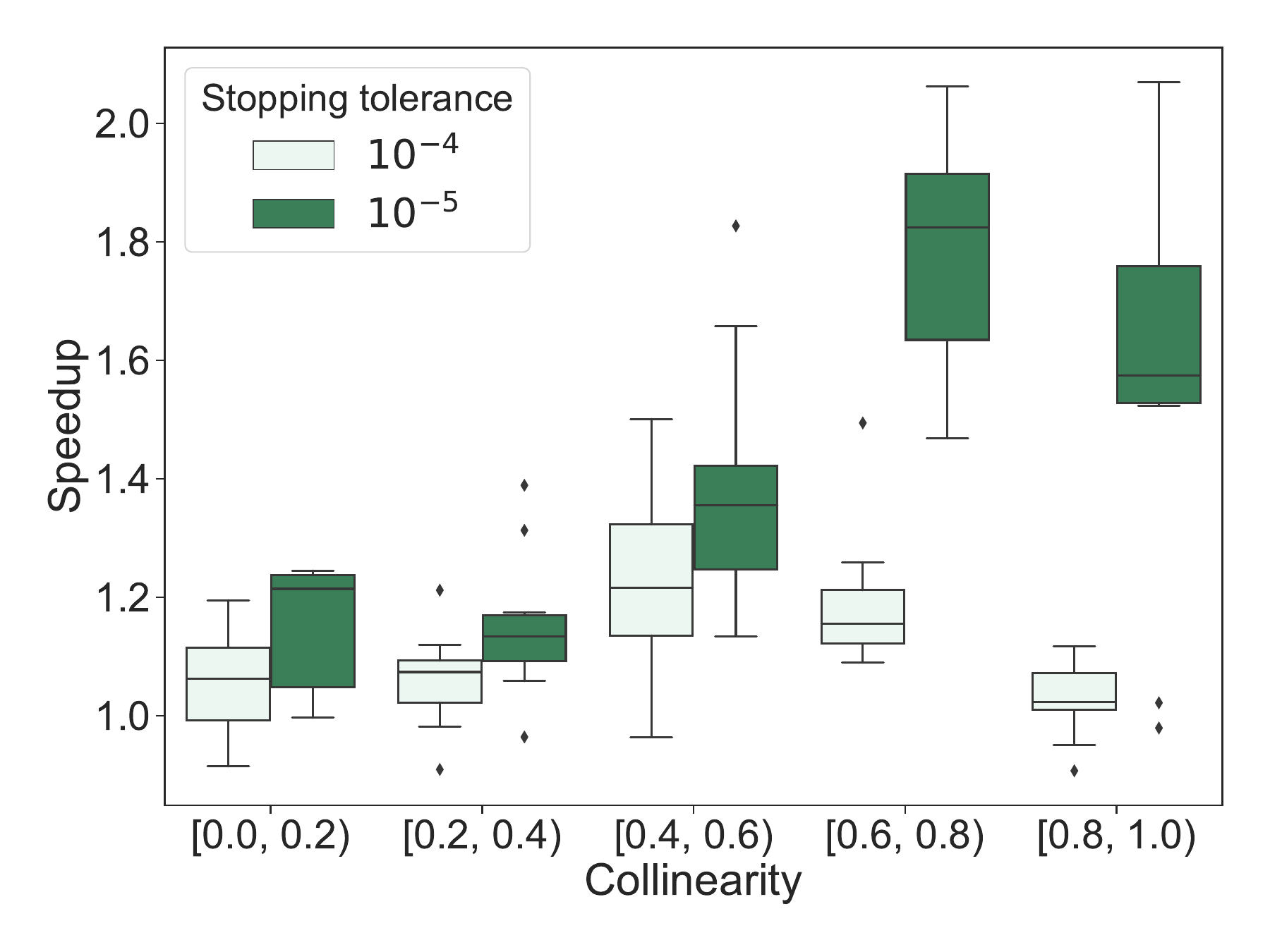}\label{cprandom:sub1}}
\subfloat[$N=4$, $s=R=120$]{\includegraphics[width=0.38\textwidth, keepaspectratio]{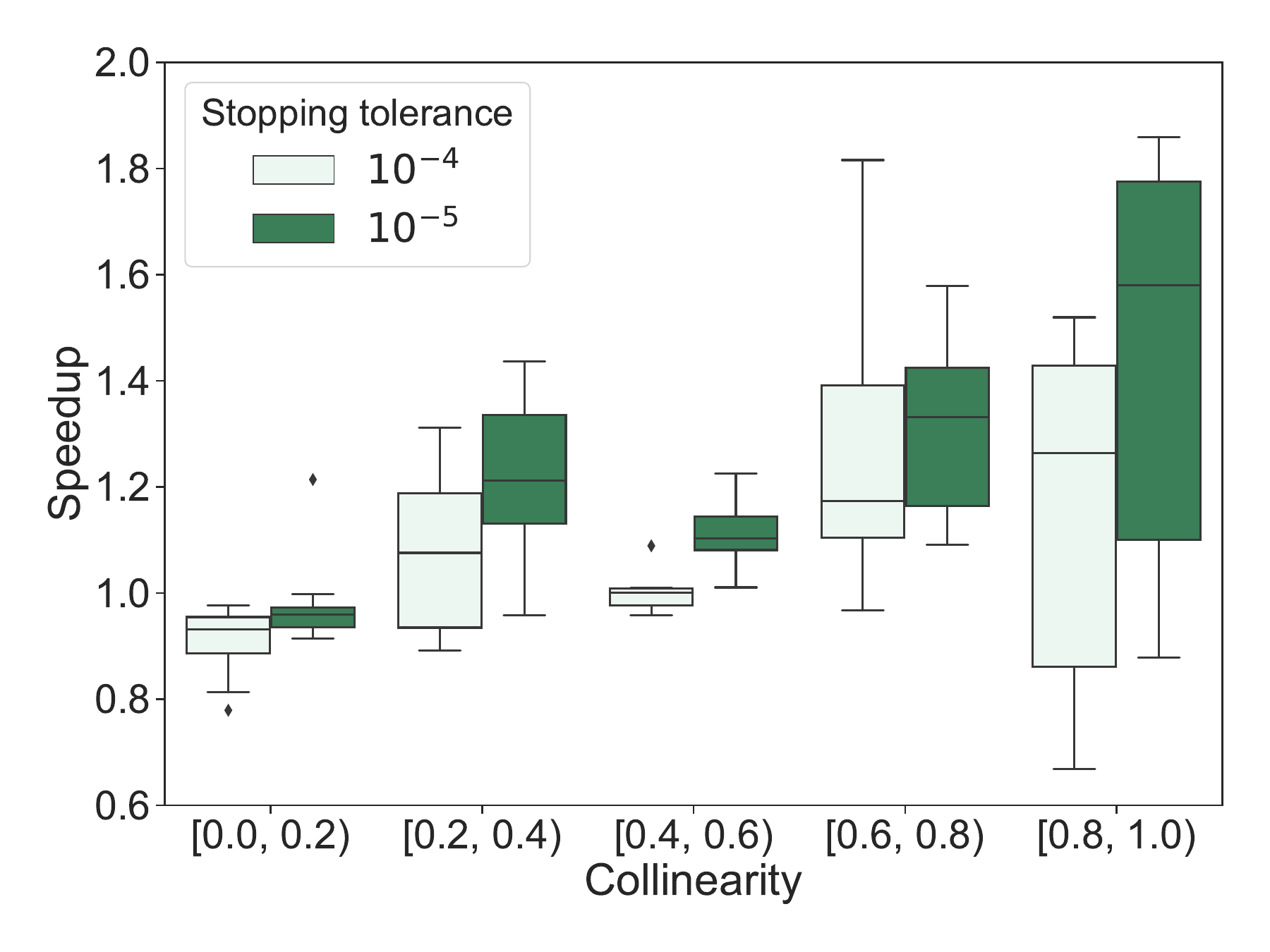}\label{cprandom:sub2}}

\subfloat[$N=3$, $s=R=400$, collinearity$\in[0.6, 0.8)$]{\includegraphics[width=0.38\textwidth, keepaspectratio]{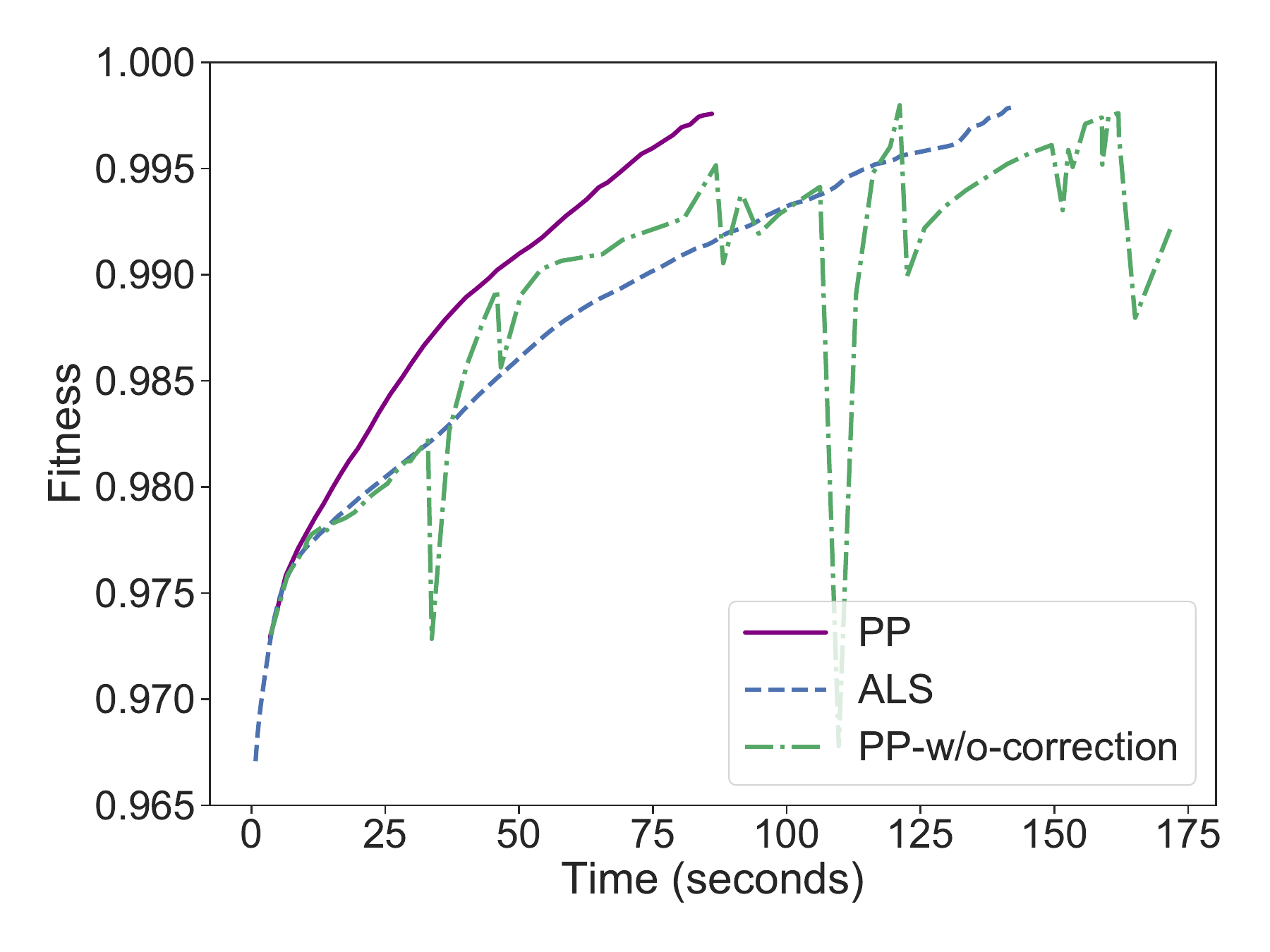}\label{cprandom:sub3}}
\subfloat[Random tensors, $N=3$, $s=R$]{\includegraphics[width=0.38\textwidth, keepaspectratio]{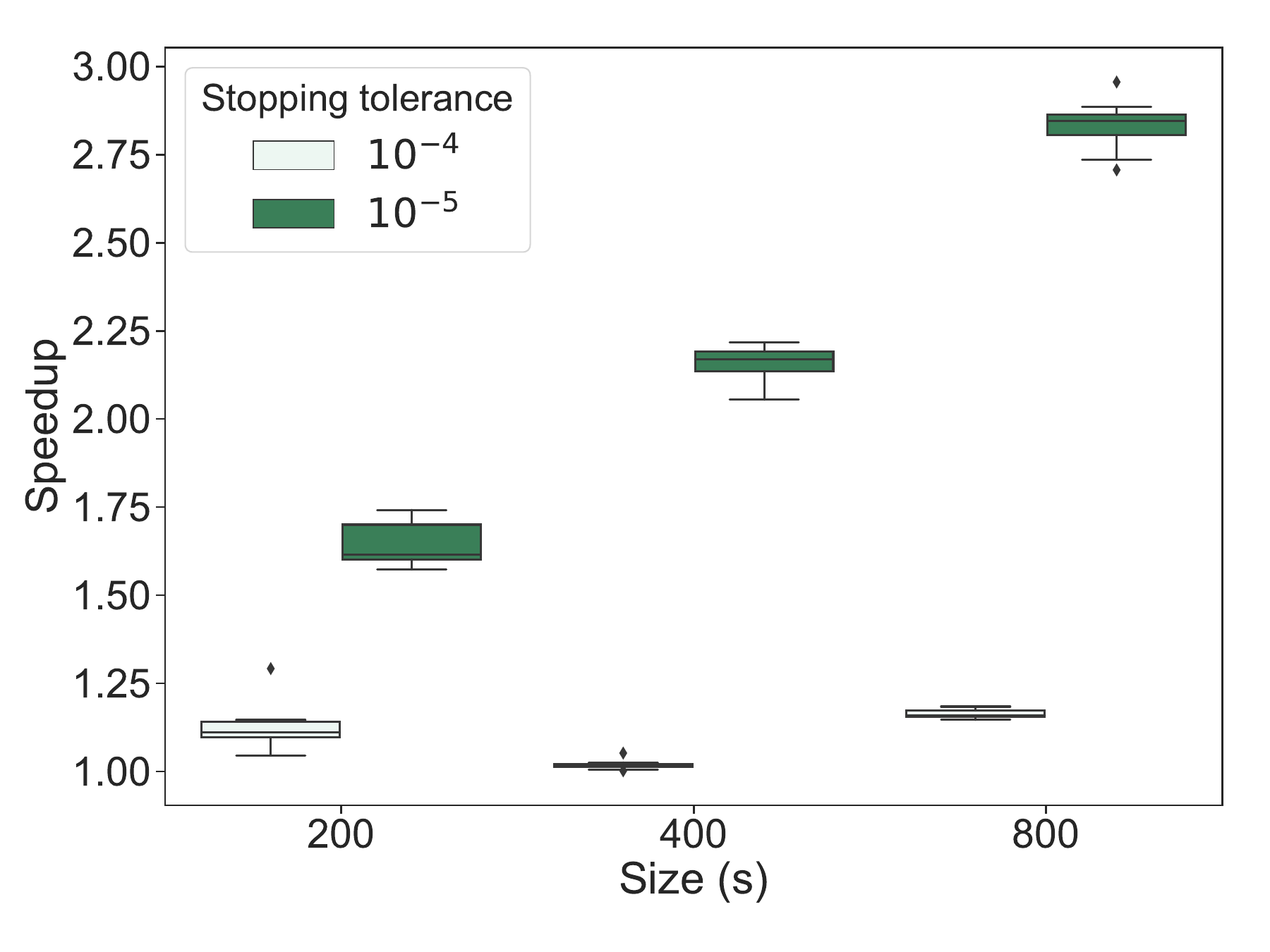}\label{cprandom:sub4}}

\caption[Optional caption for list of figures 5-8]{{\textbf{(a)(b)} Box plot of the relation between PP speed-up and input collinearity ranges for tensors with specific collinearity. \textbf{(c)} Fitness-time relation for the decomposition of one tensor with specific collinearity. \textbf{(d)} Box plot of the relation between PP speed-up and size in each mode for order 3 random tensors. } For all the box plots, each box is based on 10 experiments with different random seeds. 
Each box shows the 25th-75th quartiles, the median is indicated by a horizontal line inside the box, and outliers are displayed as dots.
}
\label{fig:randomcp}
\end{figure}

We display the speed-ups of pairwise perturbation compared to the dimension tree algorithm for synthetic tensors in Figure~\ref{fig:randomcp}. {Figures~\ref{cprandom:sub1} and~\ref{cprandom:sub2} show} the speed-up distribution with different exact factor matrices collinearity. We stop the algorithm when the stopping tolerance (defined as the fitness difference between two neighboring sweeps) is reached.
It can be seen that for both order three and order four tensors, PP achieves up to 2.0X speed-up, and high speed-up is achieved with tighter stopping tolerance. 
We find that the stricter stopping tolerance of $10^{-5}$ is valuable, as generally it permits about one more digit of accuracy to be achieved in fitness compared to a tolerance of $10^{-4}$. In addition,  experiments with a $10^{-4}$ stopping tolerance sometimes stop at transient swamps~\cite{mohlenkamp2019dynamics} with high decomposition residual, where ALS makes small progress for a period but the residual norm decreases more rapidly afterwards. 
In addition, PP tends to have higher speed-ups with relatively high collinearity. This is because tensors with high collinearity will converge in more sweeps, and more PP approximated sweeps are activated as can be seen in Table~\ref{table:randomstat}. PP starts working early for almost all the experiments, as can be observed in Figure~\ref{cprandom:sub3}, where PP starts to have speed-up when the fitness is around 0.975 and the experiment time is less than 20 seconds, and in Table~\ref{table:randomstat}, where almost all the PP initialization steps start within 20 sweeps. In addition, the fitness increases monotonically in Figure~\ref{cprandom:sub3}, indicating that PP controls the approximation error well.

Figure~\ref{cprandom:sub3} also illustrates the importance of the second-order correction term, $\mat{V}^{(n,i,j)}$, in Equation~\ref{eq:ppupdate}. We set the PP tolerance to be 0.02 for the PP experiment without corrections, which results in more conservative use of PP approximate steps than with the 0.1 tolerance we use for PP with the second-order correction. As can be seen, without the correction, PP suffers from more instability and no speed-up is achieved for this experiment. Therefore, for all other experiments, the correction terms are included as part of PP.

Figure~\ref{cprandom:sub4} shows the speed-up distribution with different dimension size for order three tensors made by random factor matrices.
It can be seen from the figure than PP achieves up to 3.0X speed-up, and PP has larger speed-ups on larger tensors, consistent with the cost analysis.
\begin{table}[!ht]
\caption{
Detailed statistics of the results shown in Figure~\ref{fig:randomcp}. From left to right: the tensor configuration (col stands for collinearity), number of exact ALS sweeps within the PP algorithm, number of PP initialization sweeps, number of PP approximated sweeps, index of sweep when PP is first initialized (approximation begins), the fitness when PP is first initialized, and the final fitness of the experiment. All the data are the average statistics from ten experiments. 
} 
\label{table:randomstat}
\centering
\begin{adjustbox}{width=.8\textwidth}
\begin{tabular}{|c|c|c|c|c|c|c|}
  \hline
  Configuration
 & Num-ALS
 & Num-PP-init
 & Num-PP-approx
 & PP-init-sweep
 & PP-init-fit
 & Final-fit
 \\ 
  \hline
N=3, col$\in[0.0, 0.2)$ & 19.9 & 2.5 & 11.4 & 12.7 & 0.8203 & 0.9330 \\ 
  \hline
N=3, col$\in[0.2, 0.4)$ & 49.1 & 18.4 & 35.3 & 7.7 & 0.7937 & 0.9991 \\ 
  \hline
N=3, col$\in[0.4, 0.6)$ & 60.8 & 52.9 & 149.1 & 8.8 & 0.9345 & 0.9999 \\ 
  \hline
N=3, col$\in[0.6, 0.8)$ & 54.8 & 50.1 & 252.1 & 5.7 & 0.9751 & 0.9962\\ 
  \hline
N=3, col$\in[0.8, 1.0)$ & 12.8 & 9.4 & 51.1 & 4.3 & 0.9940 & 0.9966\\ 
  \hline
N=4, col$\in[0.0, 0.2)$ & 20.1 & 3.3 & 2.4 & 13.7 & 0.6802 & 0.8235\\ 
  \hline
N=4, col$\in[0.2, 0.4)$ & 15.4 & 1.9 & 5.6 & 14.0 & 0.9525 & 0.9945\\ 
  \hline
N=4, col$\in[0.4, 0.6)$ & 34.0 & 7.5 & 13.5 & 22.6 & 0.9477 & 0.9935 \\ 
  \hline
N=4, col$\in[0.6, 0.8)$ & 46.1 & 29.3 & 73.3 & 9.1 & 0.9365 & 0.9990 \\ 
  \hline
N=4, col$\in[0.8, 1.0)$ & 47.5 & 26.4 & 62.4 & 6.2 & 0.9831 & 0.9963 \\ 
   \hline
\end{tabular}
\end{adjustbox}
\end{table}

\begin{figure}[!ht]   
\centering

\subfloat[Fitness-time relation]{\includegraphics[width=0.38\textwidth, keepaspectratio]{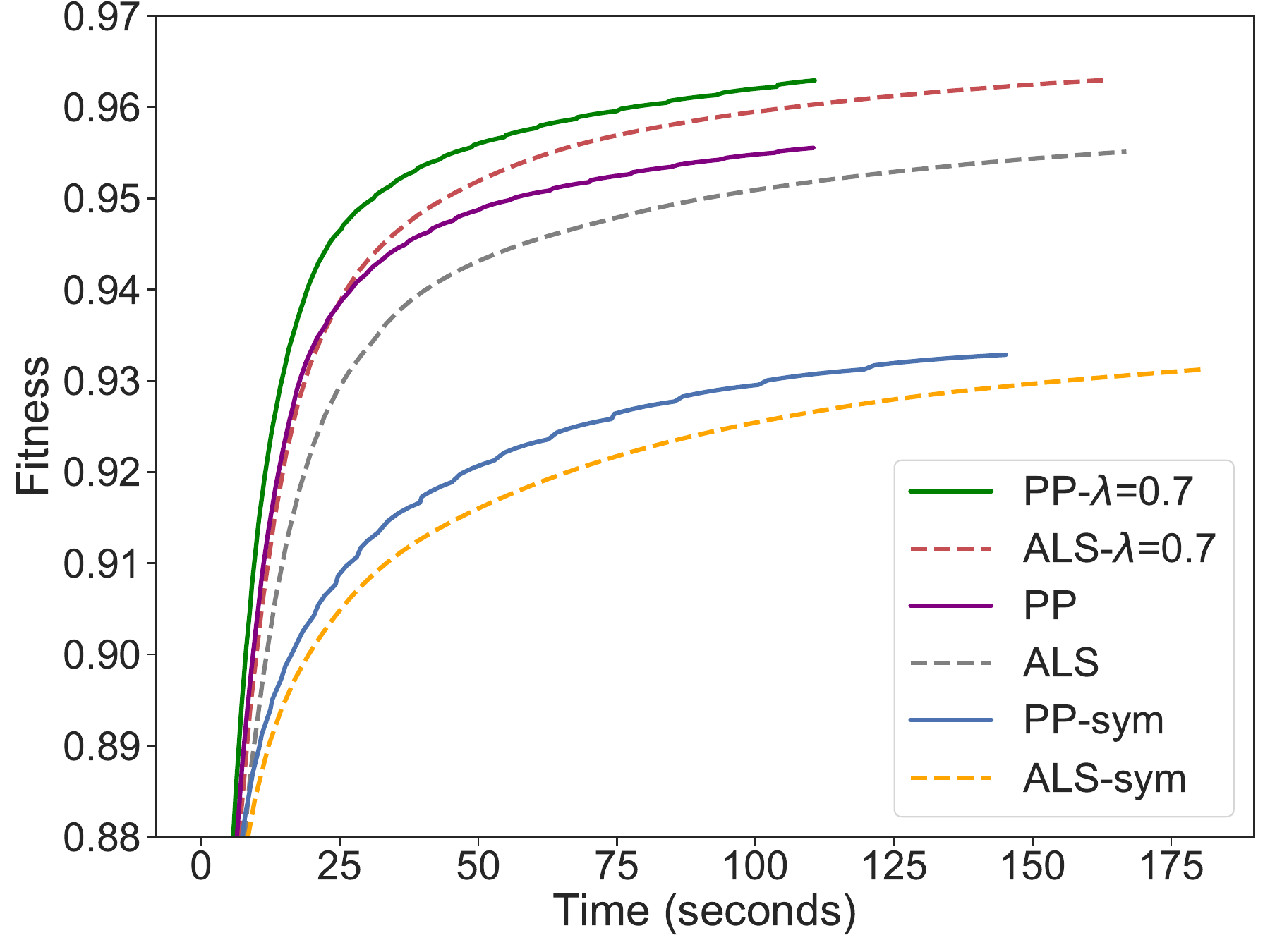}\label{scf:sub1}}
\subfloat[Fitness-sweep relation]{\includegraphics[width=0.38\textwidth, keepaspectratio]{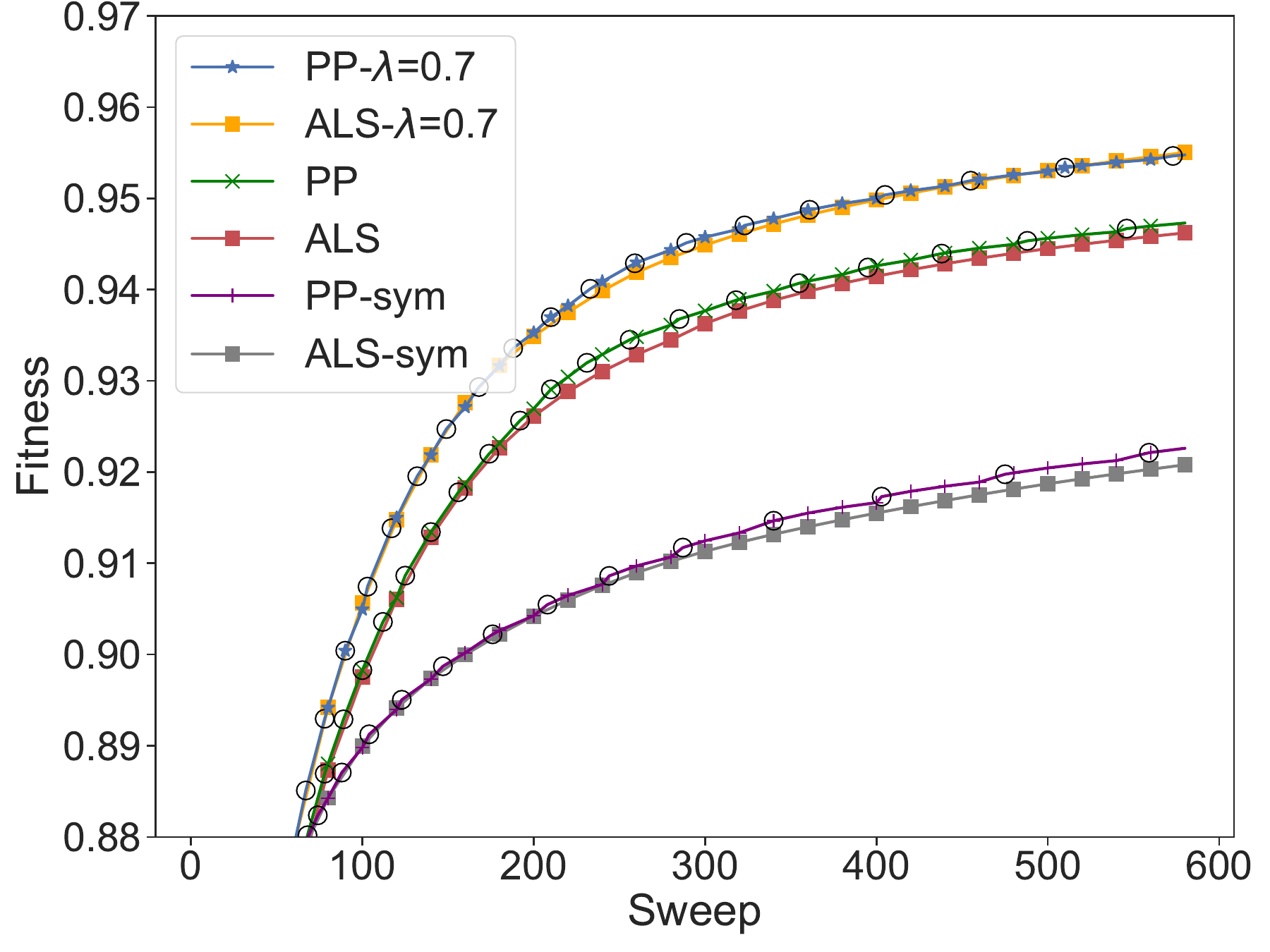}\label{scf:sub2}}

\caption[]{
Comparison of PP and the dimension tree algorithm for CP decomposition on the quantum chemistry tensor with different variants. PP-sym/ALS-sym denotes the decomposition with symmetry constraint. PP-$\lambda$=0.7/ALS-$\lambda$=0.7 denotes the decomposition with step size chosen to be 0.7. (b) shows detailed fitness-sweep relation for part of the sweeps. In (b),
squares on the
dimension tree lines represent the results per 20 sweeps (including all PP initialization, PP approximated and ALS sweeps), and the black circles on pairwise
perturbation lines represent the time when pairwise perturbation re-initializes. 
}
\label{fig:scf}
\end{figure}

\begin{table}[!htb]
\caption{
{
Detailed statistics of different experiments. From left to right: the tensor type, number of ALS sweeps until PP experiments are finished, number of PP initialization sweeps, number of PP approximated sweeps, the average time of each ALS sweep, the average time of each PP initialization sweep, and average time of each PP approximated sweep.
}
} 
\label{table:tensorstat}
\centering
\begin{adjustbox}{width=.8\textwidth}
\begin{tabular}{|c|c|c|c|c|c|c|}
  \hline
  Tensor
 & Num-ALS
 & Num-PP-init
 & Num-PP-approx
 & Time-ALS
 & Time-PP-init
 & Time-PP-approx
 \\ 
  \hline
Chemistry (Figure~\ref{fig:scf}) & 44 & 40 & 1416 & 0.1116 & 0.1655 & 0.0703 \\ 
  \hline
Coil (Figure~\ref{real:sub1}) & 31 & 22 & 147 & 2.357 & 3.660 & 0.0648 \\ 
  \hline
TimeLapse (Figure~\ref{real:sub2}) & 23 & 16 & 161 & 0.4087 & 0.9236 & 0.0562 \\ 
  \hline
Chemistry (Figure~\ref{fig:scf_par}) & 88 & 54 & 1358 & 5.338 & 9.608 & 2.254\\ 
  \hline
\end{tabular}
\end{adjustbox}
\end{table}

We also test the performance of pairwise perturbation on CP decomposition of the quantum chemistry tensor, as is shown in Figure~\ref{fig:scf}, {with detailed statistics shown in Table~\ref{table:tensorstat}.} In addition to the original ALS algorithm, we consider two other ALS variants for this problem: the ALS algorithm with different update step size, and the ALS algorithm with a symmetry constraint~\cite{hummel2017low}.
The algorithm with different update step size updates the factor matrices $\mat{A}^{(n)}$ based on
\[
\mat{A}^{(n)}_{new} = (1-\lambda)\mat{A}^{(n)} + \lambda \mat{M}^{(n)}\boldsymbol{\Gamma}^{(n)}{}^{\dagger},
\]
where $\lambda$ is the update step size. A good choice of $\lambda$ can help achieving better convergence.
The symmetry constrained algorithm considers the input tensor is symmetric in the two equidimensional modes and restricts the two factor matrices for these two modes to be the same:
$
    \tsr{X} = [\![ \mat{A}, \mat{B} , \mat{B}
    ]\!].
$
We update $\mat{A}$ the same as the original ALS step, and update $\mat{B}$ with the update step size $\lambda=0.8$ to avoid divergence.
 
As is shown in Figure~\ref{scf:sub1}, for all the variants of ALS algorithms, PP performs better than the dimension tree algorithm, achieving {1.25-1.52X} speed-up. {All the experiments are stopped after 1500 sweeps.} It can also be observed in Figure~\ref{scf:sub2} that PP usually restarts {once approximately every 40 sweeps}, and for each sweep, the fitness of both ALS and PP are almost the same, indicating that PP controls the approximation error well.

\begin{figure}[!ht]   
\centering

\subfloat[CP decomposition of Coil Dataset]{\includegraphics[width=0.38\textwidth, keepaspectratio]{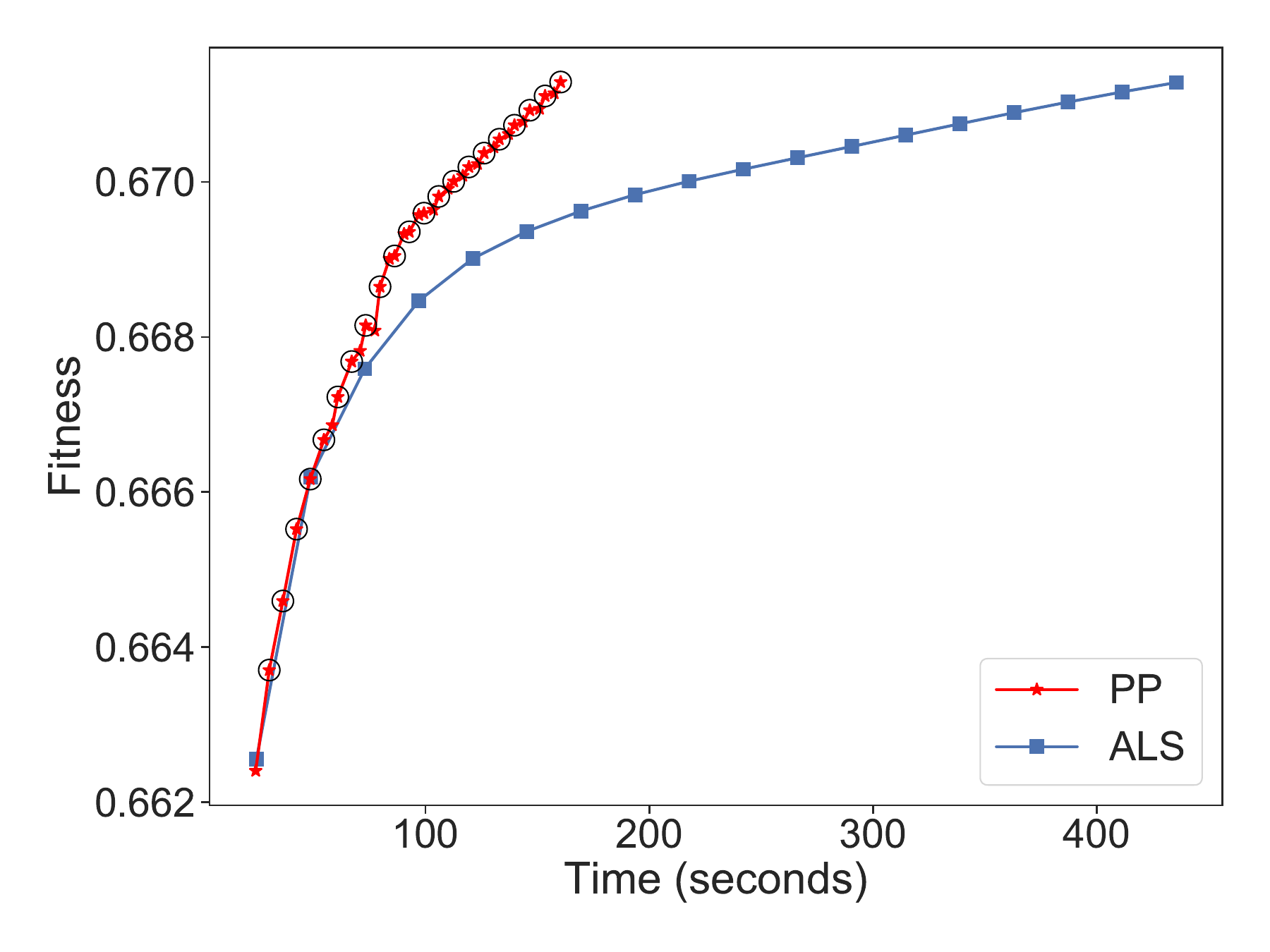}\label{real:sub1}}
\subfloat[CP decomposition of Time-Lapse Dataset]{\includegraphics[width=0.38\textwidth, keepaspectratio]{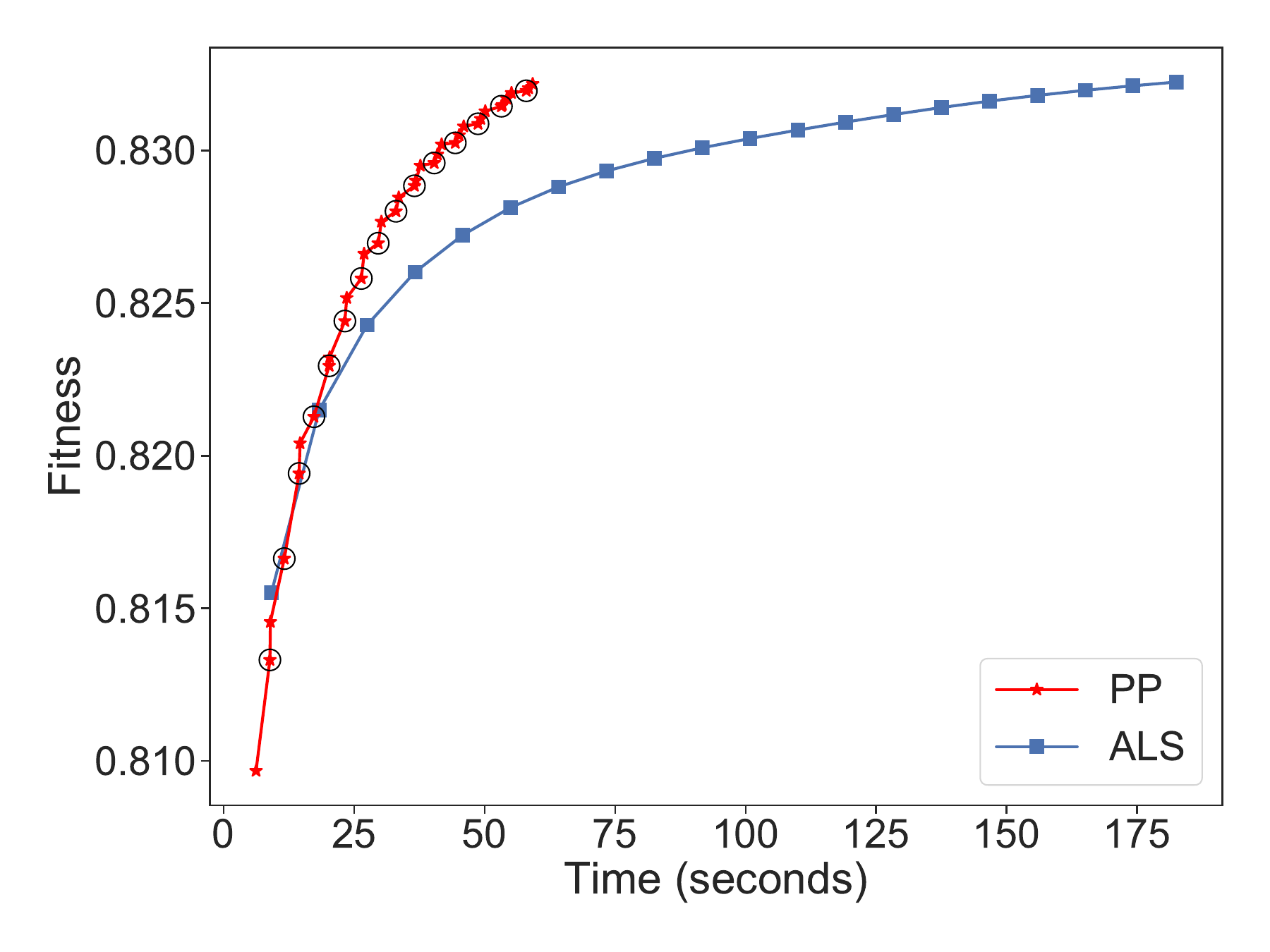}\label{real:sub2}}

\subfloat[Tucker decomposition of Coil Dataset]{\includegraphics[width=0.38\textwidth, keepaspectratio]{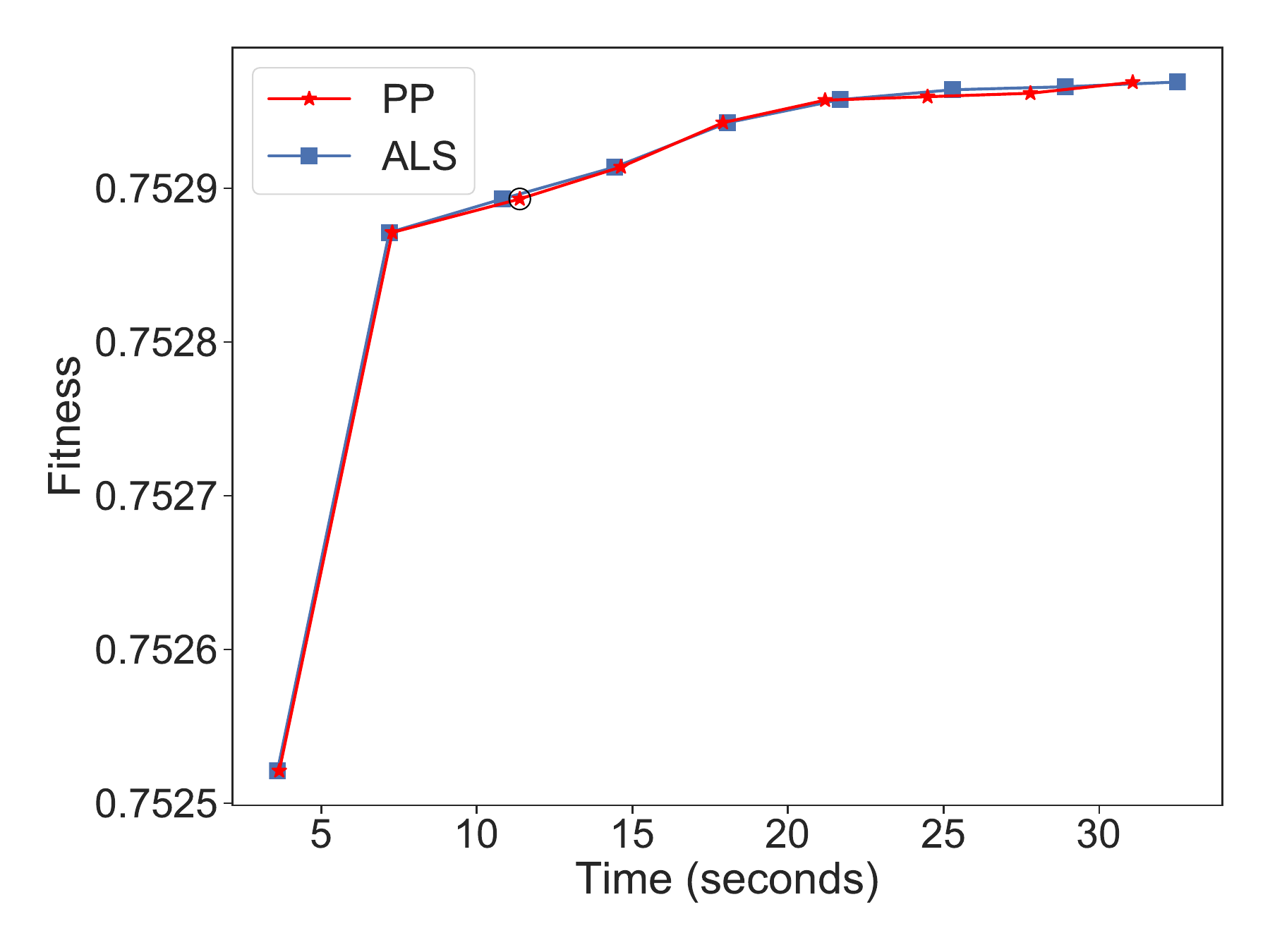}\label{real:sub3}}
\subfloat[Tucker decomposition of Time-Lapse Dataset]{\includegraphics[width=0.38\textwidth, keepaspectratio]{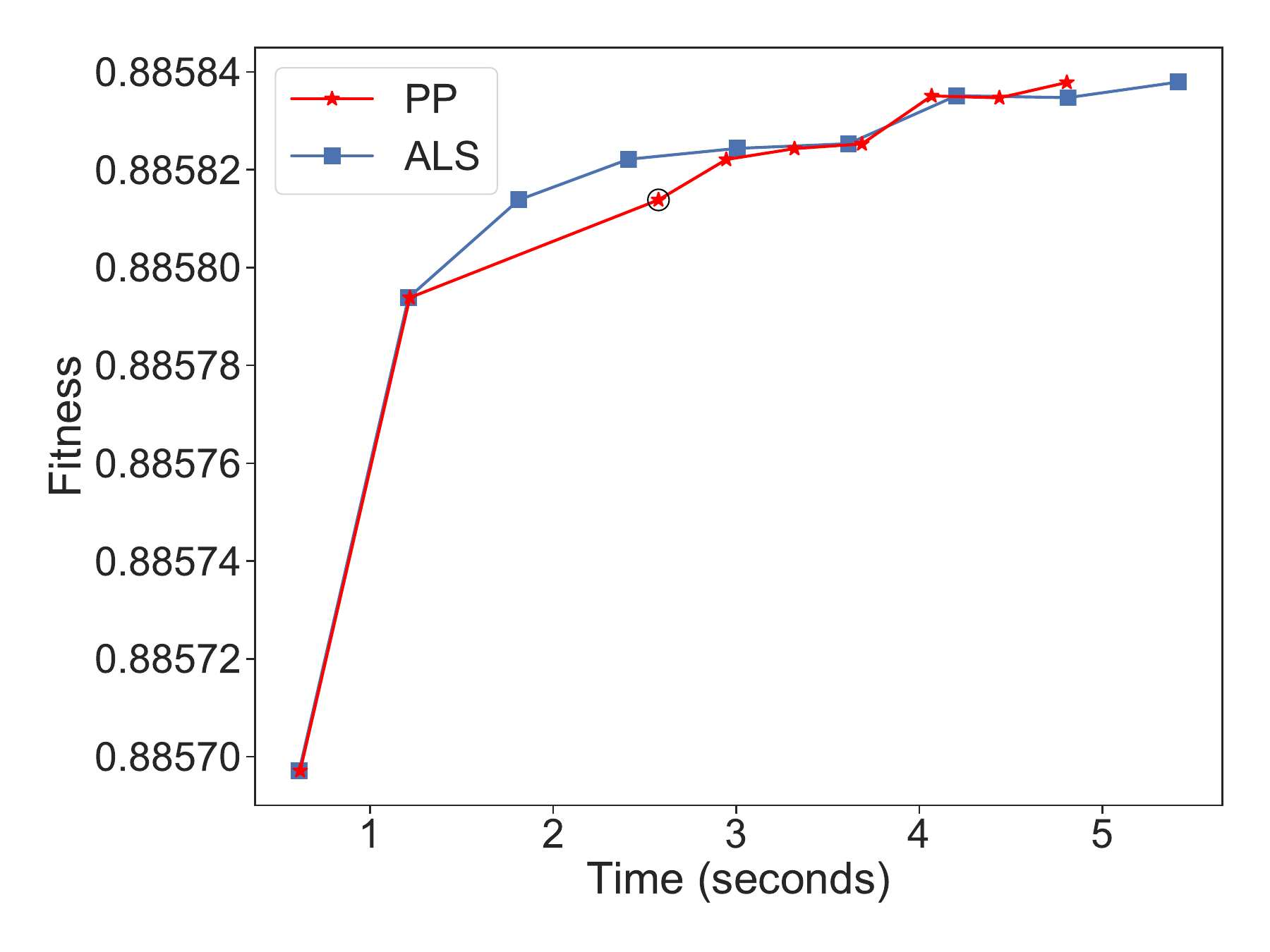}\label{real:sub4}}

\caption[]{
{
Experimental results on image datasets between pairwise perturbation and ALS for CP and Tucker decompositions. Each dot on the ALS/PP lines represents the results per 10 sweeps for CP and per sweep for Tucker decomposition (including all PP initialization, PP approximated and ALS sweeps), and the black circles on pairwise perturbation lines represent the time when pairwise perturbation restarts. }
}
\label{fig:real}
\end{figure}

We test the performance of pairwise perturbation on real image datasets with {NumPy} in Figure~\ref{fig:real}, {with detailed statistics shown in Table~\ref{table:tensorstat}.} We display the fitness and execution time for CP decomposition of the two image datasets in Figure~\ref{real:sub1},~\ref{real:sub2}. We observe that pairwise perturbation achieves a lower execution time for them. {The speed-up for the Coil Dataset is 2.72X and for the Time-Lapse Dataset is 3.1X.}

Pairwise perturbation is also used to speedup HOOI procedure in Tucker decomposition. However, as noted in other work~\cite{austin2016parallel}, we observed that ALS sweeps do not significantly lower the residual beyond what is achieved by the first sweep (HOSVD). 
We display the fitness
and the execution time for Tucker decomposition of the two real datasets in Figure~\ref{real:sub3},~\ref{real:sub4}. The speed-up for the Coil Dataset is {1.05X} and for the Time-Lapse Dataset is {1.13X}. The reason for no obvious speed-up for the Coil Dataset is that the tensor is not equidimensional (one dimension is 7200, while others are all smaller or equal to 128). Therefore, when updating the factor matrix with a dimension of 7200, the number of operations necessary to construct the SVD input for PP are similar to that for the dimension tree Tucker algorithm. For the Time-Lapse Dataset, the tensor dimensions are more evenly distributed (two dimensions are greater than 1000), and we observe a greater speed-up. We conclude that the proposed Tucker PP algorithm performs better when used on the tensors whose dimensions are approximately equal.

\subsection{Parallel Performance}

\begin{figure}[!ht]
\centering

\subfloat[Strong scaling of CP decomposition]{\label{fig:par1}\includegraphics[width=0.38\textwidth, keepaspectratio]{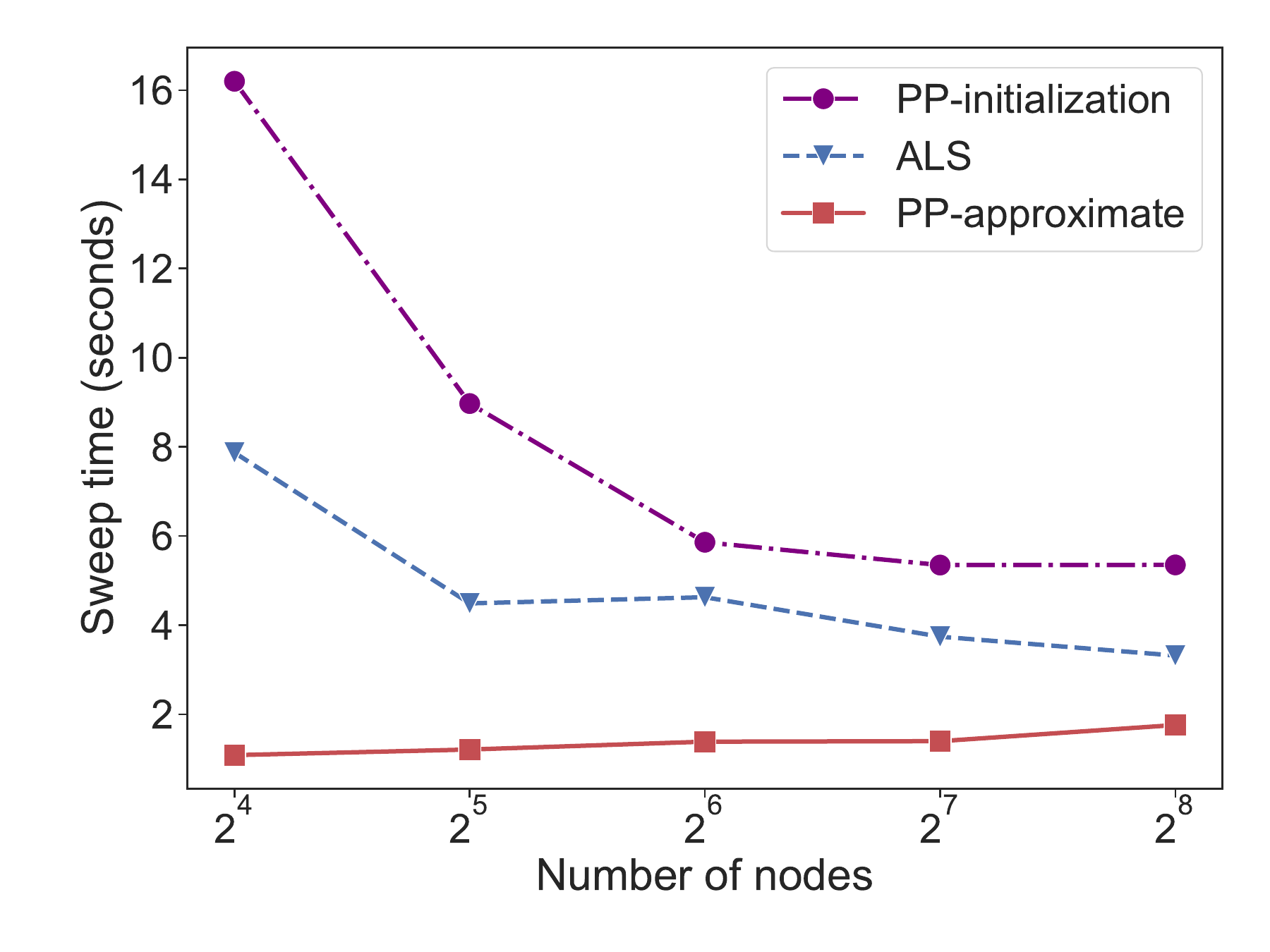}}
\subfloat[Weak scaling of CP decomposition]{\label{fig:par2}\includegraphics[width=0.38\textwidth, keepaspectratio]{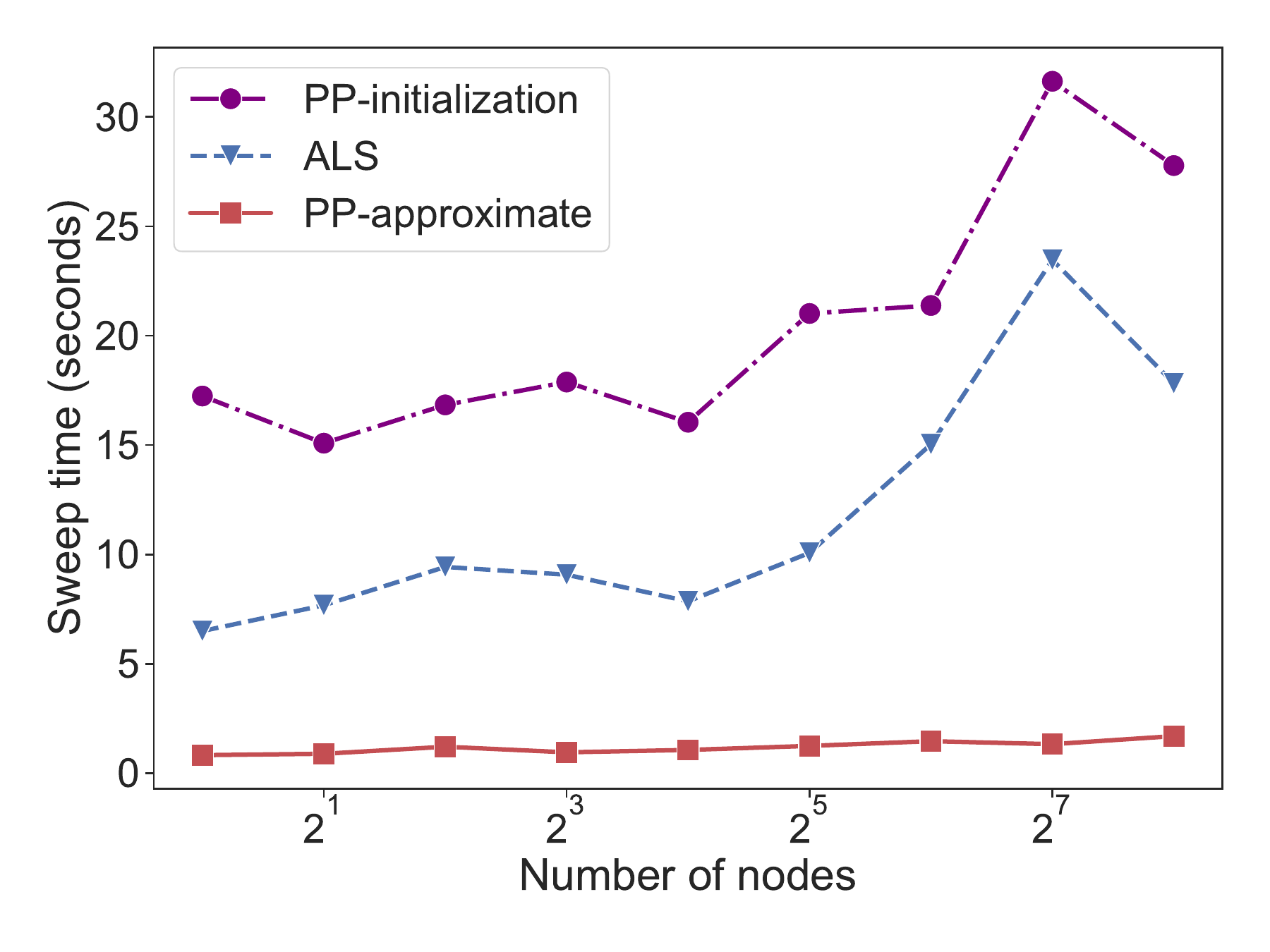}}
\caption{Benchmark results for ALS sweeps with Cyclops, taken as the mean time across 5 sweeps.}
\label{fig:bench}
\end{figure}


We perform a parallel scaling analysis to compare the simulation time for one ALS sweep of the dimension tree algorithm to the initialization and the approximated step of the pairwise perturbation algorithm with Cyclops in Figure~\ref{fig:bench}. 
Parallelism is used to accelerate the tensor contractions via calling Cyclops kernels as well as the linear system solve via calling ScaLAPACK kernels.
The Cyclops library reduces each tensor contraction to a matrix multiplication. For the PP initialization step, this approach either keeps the input tensor in place, performs local multiplications, and afterwards performs a reduction on the output tensor when the rank $R$ is small, or performs a general 3D parallel matrix multiplication when $R$ is high. For the PP approximated step, this approach parallelizes small-sized batched matrix-vector products and result in over-parallelization. We direct readers to the reference \cite{ma2020efficient} for a detailed communication cost analysis and a more communication efficient algorithm for parallel pairwise perturbation.

We use 8 processes per node and 8 threads per process for the benchmark experiments.
The pairwise perturbation initialization step includes the construction of the pairwise perturbation operators, and is therefore much slower than the approximated steps.
For strong scaling, we consider order $N=6$ tensors with dimension $s=50$ and rank $R=6$ CP and Tucker decompositions.
For weak scaling, on $p$ processors, we consider order $N=6$ tensors with dimension $s=\lfloor 32 p^{1/6}\rfloor$ and rank $R=\lfloor 4p^{1/6}\rfloor$. 

For weak scaling, Figure~\ref{fig:bench} shows that with the increase of number of nodes, the step time for all three steps increases. The approximated step time of pairwise perturbation is always much faster (7.8 and 10.5 times faster on 1 node and 256 nodes, respectively, compared to the dimension tree based ALS step time) than the two other steps, showing the good scalability of pairwise perturbation. For strong scaling, the figure shows that the approximated step time of pairwise perturbation increases with the number of nodes, while the two other step times decrease. The PP approximated step is much cheaper computationally and becomes dominated by communication with increasing node counts, thereby slowing down in step time. For the two other steps, the matrix calculation time will be decreased a lot with the increase of node number, thereby the step time is decreased. Overall, we observe that the potential performance benefit of pairwise perturbation is greater for weak scaling.

\subsection{Parallel Experimental Results}
\label{sec:par}

We test the parallel performance of pairwise perturbation with Cyclops on a quantum chemistry tensor. Similar to Section~\ref{sec:numpyexp}, we generate the order three density fitting tensor representing the compressed restricted Hartree-Fock wave function of an 40 water molecule chain system with a STO3G basis set. The generated tensor has size $4520\times 280\times 280$. We set the CP rank to be 1800.
\begin{figure}[!ht]   
\centering
\subfloat[Fitness-time relation]{\includegraphics[width=0.37\textwidth, keepaspectratio]{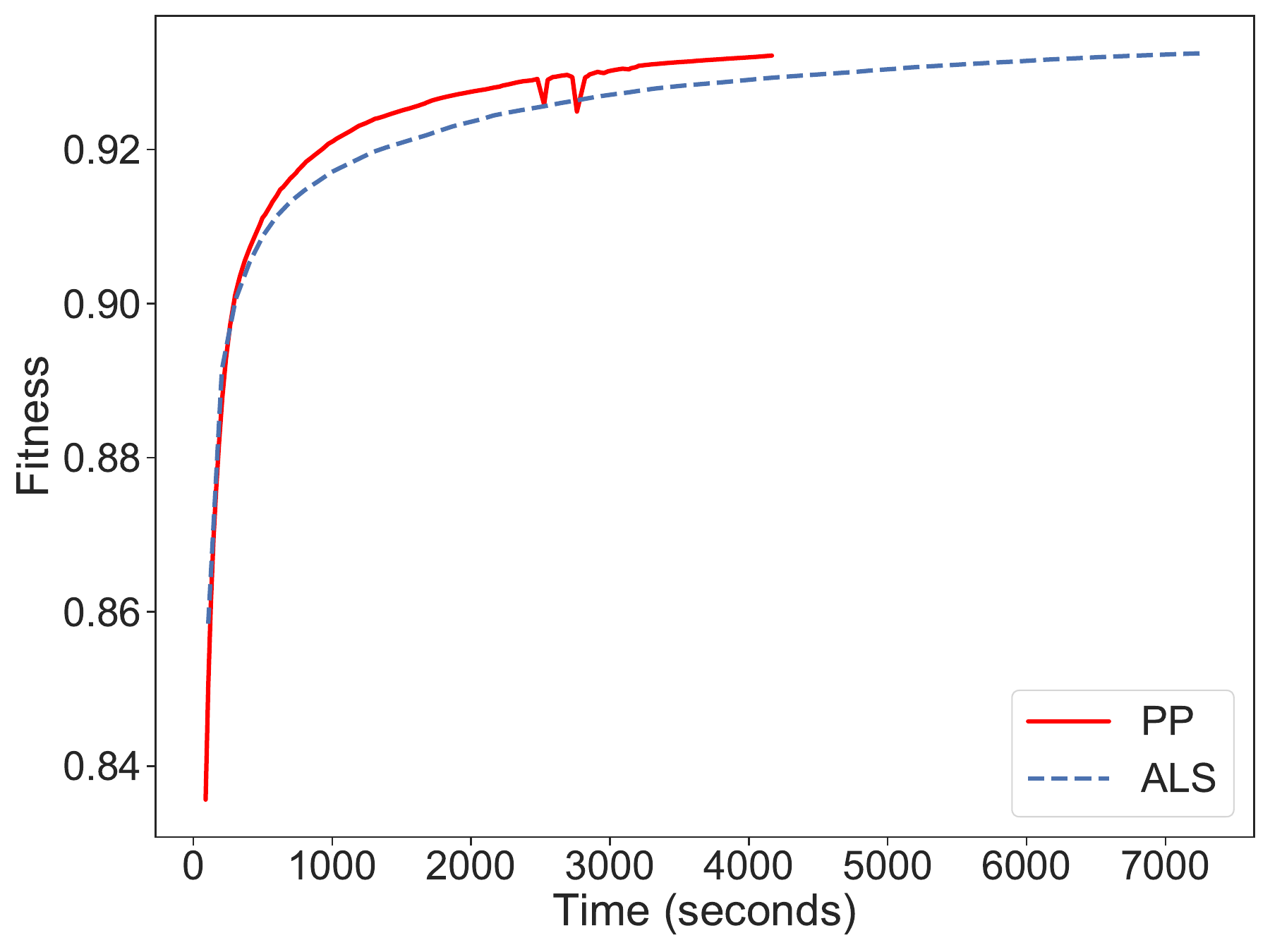}\label{scfpar:sub1}}
\subfloat[Fitness-sweep relation]{\includegraphics[width=0.38\textwidth, keepaspectratio]{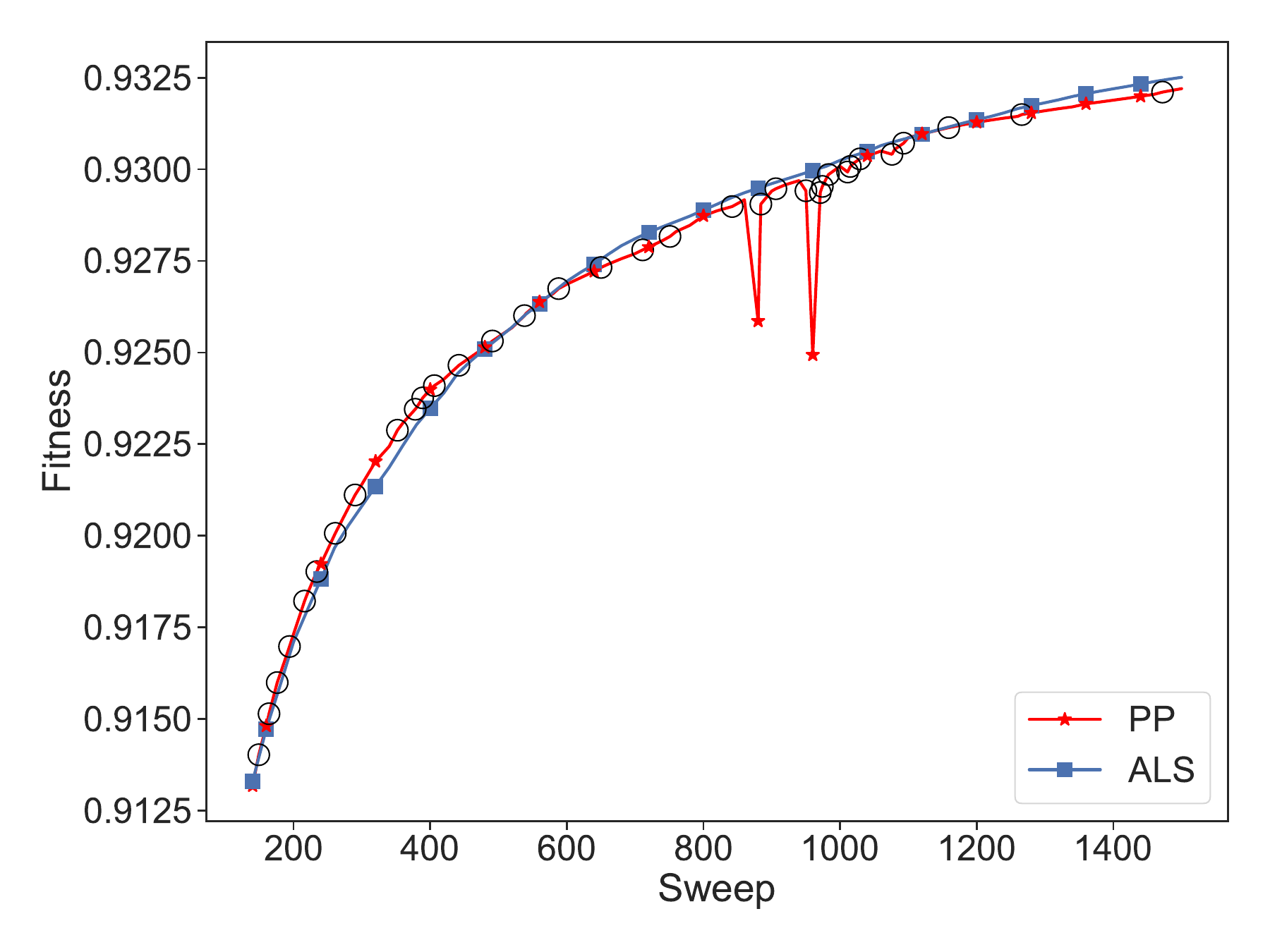}\label{scfpar:sub2}}
\caption[]{
{
Comparison of PP and the dimension tree algorithm for CP decomposition on the quantum chemistry tensor with Cyclops. (b) shows detailed fitness-sweep relation for part of the sweeps. In (b),
squares on the
dimension tree lines represent the results per 20 sweeps (including all PP initialization, PP approximated and ALS sweeps), and the black circles on pairwise
perturbation lines represent the time when pairwise perturbation re-initializes.
}
}
\label{fig:scf_par}
\end{figure}
We show the parallel performance with Cyclops for the quantum chemistry tensor in Figure~\ref{fig:scf_par}, {with detailed statistics shown in Table~\ref{table:tensorstat}.} We perform experiments on 4 KNL nodes, leveraging 64 processors on each node. For the PP experiment, after first level contractions of the PP dimension tree, we redistribute the resulting tensor across all the processes so that it is partitioned in the rank mode, which makes the PP approximated steps faster.
{It can be seen that PP performs better than the dimension tree algorithm, achieving 1.75X speed-up to reach a fitness of 0.933. It can also be observed in Figure~\ref{scfpar:sub2} that for most of the sweeps, the fitness of both the dimension tree algorithm and PP are almost the same, indicating that PP controls the approximation error well.}

\section{Conclusion}
\label{sec:conclu}
We have provided the pairwise perturbation algorithm for both CP and Tucker decompositions for dense tensors. The advantage of this algorithm is that it uses perturbative corrections rather than recomputing the tensor contractions to set up the quadratic optimization subproblems. Our error analysis demonstrates that it is accurate when the factor matrices change little. Specifically, our implementation of pairwise perturbation shows speed-ups for CP-ALS of up to {3.1X} across all synthetic and application data with respect to the best known method for exact CP-ALS with the NumPy-based sequential implementation.


We leave analysis and benchmarking of pairwise perturbation with sparse tensors for future work. Since contraction between the input tensor and the first factor matrix requires fewer operations, there is less likely to be a benefit in using pairwise perturbation.
Additionally, it is likely of interest to investigate more efficient adaptations of pairwise perturbation for non-equidimensional tensors and to experiment with alternative schemes for switching between regular ALS and pairwise perturbation.

\section{Acknowledgments}
\label{sec:ack}
The authors are grateful to Daniel Kressner for pointing out the connection to the Hurwitz problem, to Fan Huang for finding the $8\times 8 \times 8$ perfectly conditioned tensor, and to Nick Vannieuwenhoven for helpful comments.
The authors were supported by the US NSF OAC SSI program, award No.\ 1931258.
This work used the Extreme Science and Engineering Discovery Environment (XSEDE), which is supported by National Science Foundation grant number ACI-1548562.
We used XSEDE to employ Stampede2 at the Texas Advanced Computing Center (TACC) through allocation TG-CCR180006.

\bibliographystyle{abbrv}
\bibliography{draft_MTTKRP}

\begin{thebibliography}{10}

\bibitem{acar2011scalable}
E.~Acar, D.~M. Dunlavy, and T.~G. Kolda.
\newblock A scalable optimization approach for fitting canonical tensor
  decompositions.
\newblock {\em Journal of Chemometrics}, 25(2):67--86, 2011.

\bibitem{adams1962vector}
J.~F. Adams.
\newblock Vector fields on spheres.
\newblock {\em Annals of Mathematics}, pages 603--632, 1962.

\bibitem{adams1965matrices}
J.~F. Adams, P.~D. Lax, and R.~S. Phillips.
\newblock On matrices whose real linear combinations are nonsingular.
\newblock {\em Proceedings of the American Mathematical Society},
  16(2):318--322, 1965.

\bibitem{anandkumar2014tensor}
A.~Anandkumar, R.~Ge, D.~J. Hsu, S.~M. Kakade, and M.~Telgarsky.
\newblock Tensor decompositions for learning latent variable models.
\newblock {\em Journal of Machine Learning Research}, 15(1):2773--2832, 2014.

\bibitem{andersson1998improving}
C.~A. Andersson and R.~Bro.
\newblock Improving the speed of multi-way algorithms: Part {I}. {Tucker3}.
\newblock {\em Chemometrics and intelligent laboratory systems},
  42(1-2):93--103, 1998.

\bibitem{austin2016parallel}
W.~Austin, G.~Ballard, and T.~G. Kolda.
\newblock Parallel tensor compression for large-scale scientific data.
\newblock In {\em Parallel and Distributed Processing Symposium, 2016 IEEE
  International}, pages 912--922. IEEE, 2016.

\bibitem{ballard2018parallel}
G.~Ballard, K.~Hayashi, and R.~Kannan.
\newblock Parallel nonnegative {CP} decomposition of dense tensors.
\newblock {\em arXiv preprint arXiv:1806.07985}, 2018.

\bibitem{ballard2017communication}
G.~Ballard, N.~Knight, and K.~Rouse.
\newblock Communication lower bounds for matricized tensor times {Khatri-Rao}
  product.
\newblock In {\em 2018 IEEE International Parallel and Distributed Processing
  Symposium (IPDPS)}, pages 557--567. IEEE, 2018.

\bibitem{battaglino2017practical}
C.~Battaglino, G.~Ballard, and T.~G. Kolda.
\newblock A practical randomized {CP} tensor decomposition.
\newblock {\em SIAM Journal on Matrix Analysis and Applications},
  39(2):876--901, 2018.

\bibitem{benedikt2013tensor}
U.~Benedikt, H.~Auer, M.~Espig, W.~Hackbusch, and A.~A. Auer.
\newblock Tensor representation techniques in post-{H}artree--{F}ock methods:
  matrix product state tensor format.
\newblock {\em Molecular Physics}, 111(16-17):2398--2413, 2013.

\bibitem{Dongarra:1997:SUG:265932}
L.~S. Blackford, J.~Choi, A.~Cleary, E.~D'Azeuedo, J.~Demmel, I.~Dhillon,
  S.~Hammarling, G.~Henry, A.~Petitet, K.~Stanley, D.~Walker, and R.~C. Whaley.
\newblock {\em {ScaLAPACK} {U}ser's {G}uide}.
\newblock Society for Industrial and Applied Mathematics, Philadelphia, PA,
  USA, 1997.

\bibitem{carroll1970analysis}
J.~D. Carroll and J.-J. Chang.
\newblock Analysis of individual differences in multidimensional scaling via an
  {N}-way generalization of {E}ckart-{Y}oung decomposition.
\newblock {\em Psychometrika}, 35(3):283--319, 1970.

\bibitem{chakaravarthy2017optimizing}
V.~T. Chakaravarthy, J.~W. Choi, D.~J. Joseph, X.~Liu, P.~Murali, Y.~Sabharwal,
  and D.~Sreedhar.
\newblock On optimizing distributed {Tucker} decomposition for dense tensors.
\newblock In {\em Parallel and Distributed Processing Symposium (IPDPS), 2017
  IEEE International}, pages 1038--1047. IEEE, 2017.

\bibitem{choi2018high}
J.~Choi, X.~Liu, and V.~Chakaravarthy.
\newblock High-performance dense {Tucker} decomposition on {GPU} clusters.
\newblock In {\em Proceedings of the International Conference for High
  Performance Computing, Networking, Storage, and Analysis}, page~42. IEEE
  Press, 2018.

\bibitem{cichocki2016tensor}
A.~Cichocki, N.~Lee, I.~Oseledets, A.-H. Phan, Q.~Zhao, and D.~P. Mandic.
\newblock Tensor networks for dimensionality reduction and large-scale
  optimization: Part 1 low-rank tensor decompositions.
\newblock {\em Foundations and Trends in Machine Learning}, 9(4-5):249--429,
  2016.

\bibitem{de2000multilinear}
L.~De~Lathauwer, B.~De~Moor, and J.~Vandewalle.
\newblock A multilinear singular value decomposition.
\newblock {\em SIAM journal on Matrix Analysis and Applications},
  21(4):1253--1278, 2000.

\bibitem{de2000best}
L.~De~Lathauwer, B.~De~Moor, and J.~Vandewalle.
\newblock On the best rank-1 and rank-(r1, r2,..., rn) approximation of
  higher-order tensors.
\newblock {\em SIAM journal on Matrix Analysis and Applications},
  21(4):1324--1342, 2000.

\bibitem{friedland2013best}
S.~Friedland, V.~Mehrmann, R.~Pajarola, and S.~K. Suter.
\newblock On best rank one approximation of tensors.
\newblock {\em Numerical Linear Algebra with Applications}, 20(6):942--955,
  2013.

\bibitem{grasedyck2013literature}
L.~Grasedyck, D.~Kressner, and C.~Tobler.
\newblock A literature survey of low-rank tensor approximation techniques.
\newblock {\em GAMM-Mitteilungen}, 36(1):53--78, 2013.

\bibitem{hackbusch2012tensor}
W.~Hackbusch.
\newblock {\em Tensor spaces and numerical tensor calculus}, volume~42.
\newblock Springer Science \& Business Media, 2012.

\bibitem{hao2014nonnegative}
N.~Hao, L.~Horesh, and M.~Kilmer.
\newblock Nonnegative tensor decomposition.
\newblock In {\em Compressed Sensing \& Sparse Filtering}, pages 123--148.
  Springer, 2014.

\bibitem{harris2020array}
C.~R. Harris, K.~J. Millman, S.~J. van~der Walt, R.~Gommers, P.~Virtanen,
  D.~Cournapeau, E.~Wieser, J.~Taylor, S.~Berg, N.~J. Smith, R.~Kern, M.~Picus,
  S.~Hoyer, M.~H. van Kerkwijk, M.~Brett, A.~Haldane, J.~F. del R{'{\i}}o,
  M.~Wiebe, P.~Peterson, P.~G{'{e}}rard-Marchant, K.~Sheppard, T.~Reddy,
  W.~Weckesser, H.~Abbasi, C.~Gohlke, and T.~E. Oliphant.
\newblock Array programming with {NumPy}.
\newblock {\em Nature}, 585(7825):357--362, Sept. 2020.

\bibitem{harshman1970foundations}
R.~A. Harshman.
\newblock Foundations of the {PARAFAC} procedure: models and conditions for an
  explanatory multimodal factor analysis.
\newblock 1970.

\bibitem{hayashi2017shared}
K.~Hayashi, G.~Ballard, J.~Jiang, and M.~Tobia.
\newblock Shared memory parallelization of {MTTKRP} for dense tensors.
\newblock {\em arXiv preprint arXiv:1708.08976}, 2017.

\bibitem{Hillar:2013}
C.~J. Hillar and L.-H. Lim.
\newblock Most tensor problems are {NP}-hard.
\newblock {\em J. ACM}, 60(6):45:1--45:39, Nov. 2013.

\bibitem{hitchcock1927expression}
F.~L. Hitchcock.
\newblock The expression of a tensor or a polyadic as a sum of products.
\newblock {\em Studies in Applied Mathematics}, 6(1-4):164--189, 1927.

\bibitem{hohenstein:044103}
E.~G. Hohenstein, R.~M. Parrish, and T.~J. Mart\'{\i}nez.
\newblock Tensor hypercontraction density fitting. {I}. {Q}uartic scaling
  second- and third-order {M}{\o}ller-{P}lesset perturbation theory.
\newblock {\em The Journal of Chemical Physics}, 137(4):044103, 2012.

\bibitem{hohenstein2012communication}
E.~G. Hohenstein, R.~M. Parrish, C.~D. Sherrill, and T.~J. Mart{\'\i}nez.
\newblock Communication: {T}ensor hypercontraction. {III}. {L}east-squares
  tensor hypercontraction for the determination of correlated wavefunctions,
  2012.

\bibitem{HUCKLE2013750}
T.~Huckle, K.~Waldherr, and T.~Schulte-Herbr{\"u}ggen.
\newblock Computations in quantum tensor networks.
\newblock {\em Linear Algebra and its Applications}, 438(2):750 -- 781, 2013.
\newblock Tensors and Multilinear Algebra.

\bibitem{hummel2017low}
F.~Hummel, T.~Tsatsoulis, and A.~Gr{\"u}neis.
\newblock Low rank factorization of the {Coulomb} integrals for periodic
  coupled cluster theory.
\newblock {\em The Journal of chemical physics}, 146(12):124105, 2017.

\bibitem{Hurwitz1898}
A.~Hurwitz.
\newblock {\"Uber} die {Composition} der quadratischen {Formen} von belibig
  vielen {Variablen}.
\newblock {\em Nachrichten von der Gesellschaft der Wissenschaften zu
  G\"ottingen, Mathematisch-Physikalische Klasse}, 1898:309--316, 1898.

\bibitem{karlsson2016parallel}
L.~Karlsson, D.~Kressner, and A.~Uschmajew.
\newblock Parallel algorithms for tensor completion in the {CP} format.
\newblock {\em Parallel Computing}, 57:222--234, 2016.

\bibitem{kaya2017high}
O.~Kaya.
\newblock {\em High performance parallel algorithms for tensor decompositions}.
\newblock PhD thesis, 2017.

\bibitem{kaya2019computing}
O.~Kaya and Y.~Robert.
\newblock Computing dense tensor decompositions with optimal dimension trees.
\newblock {\em Algorithmica}, 81(5):2092--2121, 2019.

\bibitem{kaya2016high}
O.~Kaya and B.~U{\c{c}}ar.
\newblock High performance parallel algorithms for the {Tucker} decomposition
  of sparse tensors.
\newblock In {\em Parallel Processing (ICPP), 2016 45th International
  Conference on}, pages 103--112. IEEE, 2016.

\bibitem{kolda2006matlab}
T.~G. Kolda and B.~W. Bader.
\newblock Matlab tensor toolbox.
\newblock Technical report, Sandia National Laboratories, 2006.

\bibitem{kolda2009tensor}
T.~G. Kolda and B.~W. Bader.
\newblock Tensor decompositions and applications.
\newblock {\em SIAM review}, 51(3):455--500, 2009.

\bibitem{li2017model}
J.~Li, J.~Choi, I.~Perros, J.~Sun, and R.~Vuduc.
\newblock Model-driven sparse {CP} decomposition for higher-order tensors.
\newblock In {\em 2017 IEEE international parallel and distributed processing
  symposium (IPDPS)}, pages 1048--1057. IEEE, 2017.

\bibitem{liavas2017nesterov}
A.~P. Liavas, G.~Kostoulas, G.~Lourakis, K.~Huang, and N.~D. Sidiropoulos.
\newblock Nesterov-based alternating optimization for nonnegative tensor
  factorization: algorithm and parallel implementation.
\newblock {\em IEEE Trans. Signal Process}, 66:944--953, 2017.

\bibitem{lim2005singular}
L.-H. Lim.
\newblock Singular values and eigenvalues of tensors: a variational approach.
\newblock In {\em Computational Advances in Multi-Sensor Adaptive Processing,
  2005 1st IEEE International Workshop on}, pages 129--132. IEEE, 2005.

\bibitem{liu2013tensor}
J.~Liu, P.~Musialski, P.~Wonka, and J.~Ye.
\newblock Tensor completion for estimating missing values in visual data.
\newblock {\em IEEE Transactions on Pattern Analysis and Machine Intelligence},
  35(1):208--220, 2013.

\bibitem{ma2020efficient}
L.~Ma and E.~Solomonik.
\newblock Efficient parallel {CP} decomposition with pairwise perturbation and
  multi-sweep dimension tree.
\newblock {\em arXiv preprint arXiv:2010.12056}, 2020.

\bibitem{mohlenkamp2019dynamics}
M.~J. Mohlenkamp.
\newblock The dynamics of swamps in the canonical tensor approximation problem.
\newblock {\em SIAM Journal on Applied Dynamical Systems}, 18(3):1293--1333,
  2019.

\bibitem{1593664}
J.~G. Nagy and M.~E. Kilmer.
\newblock Kronecker product approximation for preconditioning in
  three-dimensional imaging applications.
\newblock {\em IEEE Transactions on Image Processing}, 15(3):604--613, March
  2006.

\bibitem{nascimento2016spatial}
S.~M. Nascimento, K.~Amano, and D.~H. Foster.
\newblock Spatial distributions of local illumination color in natural scenes.
\newblock {\em Vision Research}, 120:39--44, 2016.

\bibitem{nenecolumbia}
S.~A. Nene, S.~K. Nayar, and H.~Murase.
\newblock Columbia object image library (coil-100).

\bibitem{ORUS2014117}
R.~Or{\'u}s.
\newblock A practical introduction to tensor networks: Matrix product states
  and projected entangled pair states.
\newblock {\em Annals of Physics}, 349:117 -- 158, 2014.

\bibitem{oseledets2009breaking}
I.~V. Oseledets and E.~E. Tyrtyshnikov.
\newblock Breaking the curse of dimensionality, or how to use {SVD} in many
  dimensions.
\newblock {\em SIAM Journal on Scientific Computing}, 31(5):3744--3759, 2009.

\bibitem{pazner2018approximate}
W.~Pazner and P.-O. Persson.
\newblock Approximate tensor-product preconditioners for very high order
  discontinuous {G}alerkin methods.
\newblock {\em Journal of Computational Physics}, 354:344--369, 2018.

\bibitem{perros2015sparse}
I.~Perros, R.~Chen, R.~Vuduc, and J.~Sun.
\newblock Sparse hierarchical {Tucker} factorization and its application to
  healthcare.
\newblock In {\em Data Mining (ICDM), 2015 IEEE International Conference on},
  pages 943--948. IEEE, 2015.

\bibitem{phan2013fast}
A.-H. Phan, P.~Tichavsk{\`y}, and A.~Cichocki.
\newblock Fast alternating {LS} algorithms for high order {CANDECOMP/PARAFAC}
  tensor factorizations.
\newblock {\em IEEE Transactions on Signal Processing}, 61(19):4834--4846,
  2013.

\bibitem{radon1922lineare}
J.~Radon.
\newblock Lineare {Scharen} orthogonaler {Matrizen}.
\newblock In {\em Abhandlungen aus dem Mathematischen Seminar der
  Universit{\"a}t Hamburg}, volume~1, pages 1--14. Springer, 1922.

\bibitem{rajih2008enhanced}
M.~Rajih, P.~Comon, and R.~A. Harshman.
\newblock Enhanced line search: a novel method to accelerate {PARAFAC}.
\newblock {\em SIAM journal on matrix analysis and applications},
  30(3):1128--1147, 2008.

\bibitem{schatz2014exploiting}
M.~D. Schatz, T.~M. Low, R.~A. van~de Geijn, and T.~G. Kolda.
\newblock Exploiting symmetry in tensors for high performance: multiplication
  with symmetric tensors.
\newblock {\em SIAM Journal on Scientific Computing}, 36(5):C453--C479, 2014.

\bibitem{solomonik2014massively}
E.~Solomonik, D.~Matthews, J.~R. Hammond, J.~F. Stanton, and J.~Demmel.
\newblock A massively parallel tensor contraction framework for coupled-cluster
  computations.
\newblock {\em Journal of Parallel and Distributed Computing},
  74(12):3176--3190, 2014.

\bibitem{springer2017hptt}
P.~Springer, T.~Su, and P.~Bientinesi.
\newblock {HPTT}: a high-performance tensor transposition {C++} library.
\newblock In {\em Proceedings of the 4th ACM SIGPLAN International Workshop on
  Libraries, Languages, and Compilers for Array Programming}, pages 56--62.
  ACM, 2017.

\bibitem{sun2018pyscf}
Q.~Sun, T.~C. Berkelbach, N.~S. Blunt, G.~H. Booth, S.~Guo, Z.~Li, J.~Liu,
  J.~D. McClain, E.~R. Sayfutyarova, S.~Sharma, et~al.
\newblock {PySCF}: the {P}ython-based simulations of chemistry framework.
\newblock {\em Wiley Interdisciplinary Reviews: Computational Molecular
  Science}, 8(1):e1340, 2018.

\bibitem{tucker1966some}
L.~R. Tucker.
\newblock Some mathematical notes on three-mode factor analysis.
\newblock {\em Psychometrika}, 31(3):279--311, 1966.

\bibitem{vannieuwenhoven2015computing}
N.~Vannieuwenhoven, K.~Meerbergen, and R.~Vandebril.
\newblock Computing the gradient in optimization algorithms for the {CP}
  decomposition in constant memory through tensor blocking.
\newblock {\em SIAM Journal on Scientific Computing}, 37(3):C415--C438, 2015.

\bibitem{vannieuwenhoven2012new}
N.~Vannieuwenhoven, R.~Vandebril, and K.~Meerbergen.
\newblock A new truncation strategy for the higher-order singular value
  decomposition.
\newblock {\em SIAM Journal on Scientific Computing}, 34(2):A1027--A1052, 2012.

\bibitem{zhou2014decomposition}
G.~Zhou, A.~Cichocki, and S.~Xie.
\newblock Decomposition of big tensors with low multilinear rank.
\newblock {\em arXiv preprint arXiv:1412.1885}, 2014.

\end{thebibliography}

\section{Appendix: Error Bounds based on a Tensor Condition Number}
\label{sec:cond_glb}

We provide relative error bounds for the pairwise perturbation algorithm for both CP-ALS and Tucker-ALS for tensors that are `well-conditioned', in a sense that is defined in this appendix.
However, results related to the Hurwitz problem regarding multiplicative relations of quadratic forms~\cite{Hurwitz1898}, imply that equidimensional order three tensors can have a bounded condition number only if their dimension is $s\in \{1,2,4,8\}$.
We provide families of tensors with unit condition number with such dimensions.
The results shed further light on the stability of MTTKRP as well as ALS, and yield nontrivial bounds for small tensors.
For factorization of large tensors, the bounds proven in this section are not meaningful, since their condition number is necessarily infinite for at least one ordering of modes.

\subsection{Tensor Condition Number}
\label{subsec:cond}
We consider a notion of tensor condition number that corresponds to a global bound on the conditioning of the multilinear vector-valued function, 
$\fvcr{g}_{\tsr{T}} : \otimes_{i=2}^{N} \mathbb{R}^{s_{i}} \to \mathbb{R}^{s_1}$ associated with the product of the tensor with vectors along all except the first mode,
\[{
\fvcr{g}_{\tsr{T}}\left(\vcr{x}^{(2)},\ldots, \vcr{x}^{(N)}\right) = 
\tsr{T} \bigtimes_{i\in\{2,\ldots, N\}}\vcr{x}^{(i)T},
}
\]
where $\tsr{T}$ is contracted with {$\vcr{x}^{(i)}$} along its $i$th mode.
The norm and condition number are given by extrema of the norm amplification of $\fvcr{g}_{\tsr{T}}$, which are described by the amplification function $\fvcr{f}_{\tsr{T}} : \otimes_{i=2}^{N} \mathbb{R}^{s_{i}} \to \mathbb{R}$,
\[
{
\fvcr{f}_{\tsr{T}}\left(\vcr{x}^{(2)},\ldots, \vcr{x}^{(N)}\right) = \frac{\vnrm{\fvcr g_{\tsr{T}}(\vcr{x}^{(2)},\ldots, \vcr{x}^{(N)})}}{\vnrm{\vcr{x}^{(2)}}\cdots \vnrm{\vcr{x}^{(N)}}}.
}
\]
The spectral norm of the tensor corresponds to its supremum,
\[\tnrm{\tsr T} =\sup\{\fvcr f_{\tsr{T}}\}.\] 
The tensor condition number can be defined as
\[\kappa(\tsr T) = \sup\{\fvcr{f}_{\tsr{T}}\}/\inf\{\fvcr{f}_{\tsr{T}}\},\]
which enables quantification of the worst-case relative amplification of error with respect to input for the product of a tensor with vectors along all except the first mode.
In particular, $\kappa(\tsr{T})$ provides an upper bound on the relative norm of the perturbation of $\fvcr{g}_{\tsr{T}}$ with respect to the relative norm of any perturbation to any input vector.

For a matrix $\mat{M}\in\mathbb{R}^{s_1\times s_2}$, if $s_1>s_2$ the above notion of condition number gives $\kappa(\mat{M})=\frac{\sigma_\text{max}(\mat{M})}{\sigma_\text{min}(\mat{M})}$ where $\sigma_\text{min}(\mat{M})$ is the smallest singular value of $\mat{M}$ in the reduced SVD, while if $s_1<s_2$, then $\kappa(\mat{M})=\infty$.
When tensor dimensions are unequal, the condition number is infinite if the first dimension is not the largest, so for some $i$, $s_i>s_1$.
Aside from this condition, the ordering of modes of $\tsr{T}$ does not affect the condition number, since for any $m>1$, the supremum/infimum of $\fvcr{f}_{\tsr{T}}$ over the domain of unit vectors are for some choice of $\vcr{x}^{(2)},\ldots, \vcr{x}^{(m-1)}, \vcr{x}^{(m+1)},\ldots,\vcr{x}^{(N)}$ the maximum/minimum singular values of
\[
{
\mat{K} = 
\tsr{T} \bigtimes_{i\in\{2,\ldots, m-1, m+1, \ldots, N\}}\vcr{x}^{(i)T}.
}
\]

\subsection{Well-Conditioned Tensors}

We provide two examples of order three tensors that have unit condition number.
Other perfectly conditioned tensors can be obtained by multiplying the above tensors by an orthogonal matrix along any mode (we prove below that such transformations preserve condition number).
The first example has $s_i=2$, and yields a Givens rotation when contracted with a vector along the last mode.
It is composed of two slices:
\begin{align*}
\begin{bmatrix} 1 &  \\  & 1 \end{bmatrix} \quad \text{and} \quad \begin{bmatrix}  & 1 \\ -1 &  \end{bmatrix}.
\end{align*}
The second example has $s_i=4$ and is composed of four slices:
\begin{align*}
\begin{bmatrix} 1 & & &  \\ & 1 & & \\ & & 1 & \\ & & & -1 \end{bmatrix},
\begin{bmatrix} & 1 & &  \\ -1 & & & \\ & & & 1 \\ & & 1 &  \end{bmatrix},
\begin{bmatrix} & & & 1  \\ & & 1 & \\ & -1 & & \\ 1 & & & \end{bmatrix}, 
\begin{bmatrix} & & -1 &  \\ & & & 1 \\ 1 & & & \\ & 1 & & \end{bmatrix}. 
\end{align*}
Finally, for $s_i=8$, we provide an example by giving matrices $\mat{M}$ and $\mat{N}$, so that the tensor has nonzeros $\tsr{T}(i,j,\mat{M}(i,j))=\mat{N}(i,j)$ for each entry in $\mat{M}$,
\begin{align*}
\mat{M} = \begin{bmatrix}
1 & 2 & 3 & 4 & 5 & 6 & 7 & 8 \\
2 & 1 & 4 & 3 & 6 & 5 & 8 & 7 \\
3 & 4 & 1 & 2 & 7 & 8 & 5 & 6 \\
4 & 3 & 2 & 1 & 8 & 7 & 6 & 5 \\
5 & 6 & 7 & 8 & 1 & 2 & 3 & 4 \\
6 & 5 & 8 & 7 & 2 & 1 & 4 & 3 \\
7 & 8 & 5 & 6 & 3 & 4 & 1 & 2 \\
8 & 7 & 6 & 5 & 4 & 3 & 2 & 1
\end{bmatrix},  \mat{N} = \begin{bmatrix}
 1 &  1 &  1 &  1 &  1 &  1 &  1 &  1 \\
-1 &  1 &  1 & -1 &  1 & -1 & -1 &  1 \\
-1 & -1 &  1 &  1 &  1 &  1 & -1 & -1 \\
-1 &  1 & -1 &  1 &  1 & -1 &  1 & -1 \\
-1 & -1 & -1 & -1 &  1 &  1 &  1 &  1 \\
-1 &  1 & -1 &  1 & -1 &  1 & -1 &  1 \\
-1 &  1 &  1 & -1 & -1 &  1 &  1 & -1 \\
-1 & -1 &  1 &  1 & -1 & -1 &  1 &  1
\end{bmatrix}.
\end{align*}
The fact that the latter two tensors have unit condition number can be verified by symbolic algebraic manipulation or numerical tests.

These tensors provide solutions to special cases of the Hurwitz problem~\cite{Hurwitz1898}, which seeks bilinear forms $z_1,\ldots, z_n$ in variables $x_1,\ldots, x_l$ and $y_1,\ldots, y_m$ such that
\[\left(x_1^2+\cdots+x_l^2\right)\left(y_1^2+\cdots+ y_m^2\right) = z_1^2 + \cdots + z_n^2.\]
Consequently, if for $\tsr{T}$ and any vectors $\vcr{{x}}$, $\vcr{{y}}$,
\[
\frac{\vnrm{\tsr{T}\times_2 \vcr{{x}}^{{T}} \times_3 \vcr{y}^{{T}}}}{\vnrm{\vcr x}\vnrm{\vcr y}} =  1 \quad \Rightarrow \quad \vnrm{\tsr{T}\times_2 \vcr{{x}}^{{T}} \times_3 \vcr{y}^{{T}}}^2 = \vnrm{\vcr x}^2 \vnrm{\vcr y}^2,
\] 
so we can define bilinear forms,
\[z_i = \sum_{j}\sum_{k} \tsr{T}(i,j,k)x_jy_k,\]
that provide a solution to the Hurwitz problem.
Such equidimensional tensors with unit condition number exist for dimension $s\in\{1,2,4,8\}$~\cite{radon1922lineare}, corresponding to the Hurwitz problem with $l=m=n=s$.
However, solutions to the Hurwitz problem with $l=m=n$ cannot exist for any other dimension.
Furthermore, tight bounds exist on the dimension $s_3$ for a tensor of dimensions $s\times s\times s_3$ to have bounded condition number ($\inf\{\fvcr{f}_{\tsr{T}}\}> 0$).
This problem is equivalent to finding $s_3$ matrices of dimension $s\times s$, such that any nonzero linear combination thereof is invertible.
Factorizing $s=2^{4a+b}c$, where $b\in\{0,1,2,3\}$ and $c$ is odd, $s_3\leq 8a+2^b$~\cite{adams1962vector,adams1965matrices}.

\subsection{Properties of the Tensor Condition Number}

In our analysis, we make use of the following submultiplicativity property of the tensor condition number with respect to multilinear multiplication (the property also generalizes to pairs of arbitrary order tensors contracted over one mode).
\begin{lem}
\label{lem:kappamat}
For any $\tsr{T}{\in}\mathbb{R}^{s_1\times\cdots\times s_N}$ and matrix $\mat{M}$, if  {$\tsr{V}{=}\tsr{T}\times_N\mat{M}^T$} then  $\kappa(\tsr{V}){\leq} \kappa(\tsr{T})\kappa(\mat{M})$.
\end{lem}
\begin{proof}
Assume $\kappa(\tsr{V})> \kappa(\tsr{T})\kappa(\mat{M})$, then there exist unit vectors {$\vcr{x}^{(2)},\ldots, \vcr{x}^{(N)}$ and $\vcr{y}^{(2)},\ldots, \vcr{y}^{(N)}$} such that
\[
\kappa(\tsr{T})\kappa(\mat{M})<\kappa(\tsr{V})
=\frac{
{\vnrm{\tsr{V} \bigtimes_{i\in\{2,\ldots, N\}}\vcr{x}^{(i)T}}}
}{
{\vnrm{\tsr{V} \bigtimes_{i\in\{2,\ldots, N\}}\vcr{y}^{(i)T}}}
} 
=\frac{
{\vnrm{\tsr{T} \bigtimes_{i\in\{2,\ldots, N-1\}}\vcr{x}^{(i)T} \times_N \left(\mat{M}\vcr{x}^{(N)}\right)^T }}
}{
{\vnrm{\tsr{T} \bigtimes_{i\in\{2,\ldots, N-1\}}\vcr{y}^{(i)T} \times_N \left(\mat{M}\vcr{y}^{(N)}\right)^T }}
}.
\]
\normalsize
Let 
\(
\vcr{u} = \mat{M}{\vcr{x}^{(N)}} 
\)
and
\(
\vcr{v} = \mat{M}{\vcr{y}^{(N)}} 
\), so $\vnrm{\vcr{u}}/\vnrm{\vcr{v}}\leq \kappa(\mat{M})$,
yielding a contradiction,
\begin{align*}
\kappa(\tsr{V})
\leq\frac{
{\vnrm{\tsr{T} \bigtimes_{i\in\{2,\ldots, N-1\}}\vcr{x}^{(i)T} \times_N (\vcr{u}/\vnrm{\vcr{u}})^T }}
}{
{\vnrm{\tsr{T} \bigtimes_{i\in\{2,\ldots, N-1\}}\vcr{y}^{(i)T} \times_N (\vcr{v}/\vnrm{\vcr{v}})^T }}
}\kappa(\mat{M})
\leq
\kappa(\tsr{T})\kappa(\mat{M}).
\end{align*}
\end{proof}
Applying Lemma~\ref{lem:kappamat} with a vector, i.e. when $\mat{M}\in\mathbb{R}^{s_N\times 1}$ and so has condition number $\kappa(\vcr{M})=1$, implies $\kappa\left({\tsr{T}\times_N\mat{M}^T}\right)\leq \kappa(\tsr{T})$.
By an analogous argument to the proof of Lemma~\ref{lem:kappamat}, we can also conclude that the norm and infimum of such a product of $\tsr{T}$ with unit vectors are bounded by those of $\tsr{T}$, giving the following corollary.
\begin{cor}
\label{cor:kappared}
For any $\tsr{T}\in\mathbb{R}^{s_1\times\cdots\times s_N}$, vector $\vcr{u} \in \mathbb{R}^{s_n}$, and any 
$n\in\inti{1}{N}$ such that $\exists m \in\inti{1}{N}$ with $s_m\geq s_n$ and $m\neq n$, if
\(\tsr{V} = \tsr{T} \times_n \vcr{u}^{{T}},\)
then $\tnrm{\tsr{V}} \leq \vnrm{\vcr{u}}\tnrm{\tsr{T}}$, $\inf\{\fvcr{f}_{\tsr{V}}\} \geq  \vnrm{\vcr{u}}\inf\{\fvcr{f}_{\tsr{T}}\}$, and
$\kappa(\tsr{V})\leq \kappa(\tsr{T})$.
\end{cor}
For an orthogonal matrix $\mat{M}$, Lemma~\ref{lem:kappamat} can be applied in both directions, namely for $\tsr{V}={\tsr{T}\times_N\mat{M}^T}$ and $\tsr{T}={\tsr{V}\times_N\mat{M}}$, so we observe that $\kappa(\tsr{V})=\kappa(\tsr{T})$.
Using this fact, we demonstrate in the following theorem that any tensor $\tsr{T}$ can be transformed by orthogonal matrices along each mode, so that one of its fibers has norm $\tnrm{\tsr{T}}/\kappa(\tsr{T})$.
\begin{theorem}
\label{thm:fibersimilar}
For any $\tsr{T}\in\mathbb{R}^{s_1\times\cdots\times s_N}$, there exist orthogonal matrices  {$\mat{Q}^{(2)}\ldots \mat{Q}^{(N)}$}, with ${\mat{Q}^{(i)}}\in\mathbb{R}^{s_{i}\times s_{i}}$, such that {$\tsr{V}=\tsr{T} \times_2 \mat{Q}^{(2)}\cdots \times_N\mat{Q}^{(N)}$} satisfies $\kappa(\tsr{V})=\kappa(\tsr{T})$, $\tnrm{\tsr{V}}=\tnrm{\tsr{T}}$, and the first fiber of $\tsr{V}$, i.e. the vector $\vcr{v}$ with $\vcr{v}(i)=\tsr{V}(i,0,\ldots,0)$, satisfies $\vnrm{\vcr{v}}=\tnrm{\tsr T}/\kappa(\tsr T)$.
\end{theorem}
\begin{proof}
Given a tensor $\tsr{T}$ with infinite condition number, there must exist $N-1$ unit vectors {$\vcr{x}^{(2)},\ldots, \vcr{x}^{(N)}$}, such that 
$\vnrm{
{
\tsr{T} \bigtimes_{i\in\{2,\ldots, N\}}\vcr{x}^{(i)T}}
}=\tnrm{\tsr{T}}/\kappa(\tsr{T})$.
We define $N-1$ orthogonal matrices {$\mat{Q}^{(2)}, \ldots, \mat{Q}^{(N)}$} such that {$\mat{Q}^{(i)T}\vcr{x}^{(i)}=\vcr{e}_i$}. 
We can then contract $\tsr{T}$ with these matrices along the last $N-1$ modes, resulting in $\tsr{V}$, with the same condition number as $\tsr{T}$ (by Lemma~\ref{lem:kappamat}) and the same norm (by a similar argument).
Then, we have that the first fiber of $\tsr{V}$ is
\[
{
\vcr{v} = 
\tsr{V} \bigtimes_{i\in\{2,\ldots, N\}}\vcr{e}_i^{T}
=
\tsr{T} \bigtimes_{i\in\{2,\ldots, N\}}\vcr{x}^{(i)T},
}
\]
and consequently $\vnrm{\vcr{v}}=\tnrm{\tsr{T}}/\kappa(\tsr{T})$.
\end{proof}
By Theorem~\ref{thm:fibersimilar}, the condition number of a tensor is infinity if and only if it can be transformed by products with orthogonal matrices along the last $N-1$ modes into a tensor with a zero fiber.
Further, any tensor $\tsr{T}$ may be perturbed to have infinite condition number by adding to it some $\delta\tsr{T}$ with relative norm $\tnrm{\delta\tsr{ T}}/\tnrm{\tsr{T}}=1/\kappa(\tsr{T})$.

\subsection{PP-CP-ALS Error Bound using Tensor Condition Number}
\label{subsec:cp_bound_cond}

For CP decomposition, we obtain condition-number-dependent column-wise error bounds on $\mat{M}^{(n)}$ (the right-hand sides in the linear least squares subproblems), based on the magnitude of the relative perturbation to $\mat{A}^{(n)}$ since the formation of the pairwise perturbation operators.

\begin{thm}
\label{thm:cperrbound2}
If $\frac{\tnrm{d\vcr{a}^{(n)}_k}}{\tnrm{\vcr{a}^{(n)}_k}}\leq\epsilon \ll 1 \text{ for } n\in \inti{1}{N}, k \in\inti{1}{R}$ and $s_m\leq s_n$ for any $m\in\inti{1}{N}$, the pairwise perturbation algorithm without second order corrections computes $\Tilde{\mat{M}}^{(n)}$ with column-wise error,
\[\frac{\tnrm{\Tilde{\vcr{m}}_k^{(n)}-\vcr{m}_k^{(n)}}}{\tnrm{\vcr{m}^{(n)}_k}} = O\left(\epsilon^2\kappa(\tsr{X})\right),\]
where $\mat{M}^{(n)}$ is the matrix given by a regular ALS sweep.
\end{thm}
\begin{proof}
We bound the error due to second order perturbations in $d\mat{A}^{(1)},\ldots,d\mat{A}^{(n)}$, by similar analysis, higher-order perturbations would lead to errors smaller by a factor of $O(\text{poly}(N)\epsilon)$ and are consequently negligible if $\epsilon \ll 1$.
Consider the order four tensors $\tsr{M}^{(i,j,n)}$ (Equation~\ref{eq:tensors-cp}) based on the current factor matrices $\mat{A}^{(1)},\ldots, \mat{A}^{(N)}$ and the pairwise perturbation operators $\tsr{M}^{(i,j,n)}_p$ based on past factor matrices $\mat{A}_p^{(1)},$
$\ldots, \mat{A}_p^{(N)}$.
The contribution of second order terms to the error is
\[\Tilde{\textbf{m}}^{(n)}_k(x)-\textbf{m}^{(n)}_k(x) \approx \sum_{\substack{{i,j\in \{1,\ldots,n-1,n+1,\ldots,N\}} \\{ i\neq j}}} \sum_{y=1}^s\sum_{z=1}^s\tsr{M}_p^{(i,j,n)}(x,y,z,k)d\vcr{a}_k^{(i)}(y) d\vcr{a}_k^{(j)}(z).\]
This absolute error has magnitude,
\[
\vnrm{\Tilde{\vcr{m}}^{(n)}_k-\vcr{m}^{(n)}_k} \leq {N \choose 2} \max_{i,j} \tnrm{\tsr{M}_p^{(i,j,n)}(:,:,:,k)} \vnrm{d\vcr{a}^{(i)}_k} \vnrm{d\vcr{a}^{(j)}_k}.
\]
Using the fact that for any $i,j$ we can express $\vcr{m}^{(n)}_k$ as
\[\vcr{m}^{(n)}_k(x) = \sum_{y=1}^s \sum_{z=1}^s\tsr{M}^{(i,j,n)}(x,y,z,k) \vcr{a}^{(i)}_k(y) \vcr{a}^{(j)}_k(z),\]
we can lower bound the magnitude of the answer with respect to any $\tsr{M}^{(i,j,n)}$,
\begin{align*}
\vnrm{\vcr{m}^{(n)}_k} \geq \tinf{\tsr{M}^{(i,j,n)}(:,:,:,k)}\vnrm{\vcr{a}^{(i)}_k} \vnrm{\vcr{a}^{(j)}_k}.
\end{align*}
Combining the upper bound on the absolute error with the lower  bound on norm,
\begin{align*}
\frac{\vnrm{\Tilde{\vcr{m}}^{(n)}_k-\vcr{m}^{(n)}_k}}{\vnrm{\vcr{m}^{(n)}_k}} \leq {N \choose 2}\max_{i,j} \frac{ \tnrm{\tsr{M}_p^{(i,j,n)}(:,:,:,k)} \vnrm{d\vcr{a}^{(i)}_k} \vnrm{d\vcr{a}^{(j)}_k}}{\tinf{\tsr{M}^{(i,j,n)}(:,:,:,k)}\vnrm{\vcr{a}^{(i)}_k} \vnrm{\vcr{a}^{(j)}_k}}.
\end{align*}
Lemma~\ref{lem:kappamat} implies that  for any $i,j,k$,
\[\tnrm{\tsr{M}_p^{(i,j,n)}(:,:,:,k)}\leq \tnrm{\tsr{X}} \prod_{l\in \inti{1}{N} \setminus \{i,j,n\}} \tnrm{\mat{A}^{(l)}_p(:,k)}\]
and that 
\[\tinf{\tsr{M}^{(i,j,n)}(:,:,:,k)}\geq \tinf{\tsr{X}}\prod_{l\in \inti{1}{N} \setminus\{i,j,n\}} \tnrm{\mat{A}^{(l)}(:,k)}.\]
Since, $\tnrm{\mat{A}^{(l)}_p(:,k)}\leq (1+\epsilon) \tnrm{\mat{A}^{(l)}(:,k)}$, we obtain the bound,
\begin{align*}
\frac{\vnrm{\Tilde{\vcr{m}}^{(n)}_k-\vcr{m}^{(n)}_k}}{\vnrm{\vcr{m}^{(n)}_k}}
\leq {N\choose 2}\kappa(\tsr{X})(1+\epsilon)^{N-3}\epsilon^2\approx {N\choose 2}\kappa(\tsr{X})\epsilon^2.
\end{align*}

\end{proof}
This error bound is relative to the condition number of $\tsr{X}$, which means the bound is sensitive to the input tensor and that the error may be unbounded if $\tsr{X}$ has an exact CP decomposition of rank at most $\min_{i}s_i$.

\subsection{PP-Tucker-ALS Error Bound using Tensor Condition Number}

For Tucker decomposition, we again obtain bounds based on the perturbation to $\mat{A}^{(n)}$, this time for $\tsr{Y}^{(n)}$ (the tensors on whose matricizations a truncated SVD is performed).
Using Lemma~\ref{lem:kappamat}, we prove in Theorem~\ref{thm:tuckerbound2} that when the tensor has same length in each mode and the relative error of the matrices $\mat{A}^{(n)}$ for $n\in\inti{1}{N}$ is small, the relative error for the $\tilde{\tsr{Y}}^{(n)}$ is also small.
\begin{thm}
\label{thm:tuckerbound2}
Given an order $N$ equidimensional tensor $\tsr{X}$ with size $s$, if
$\tnrm{d\mat{A}^{(n)}}\leq\epsilon\ll 1 \text{ for } n\in \inti{1}{N}$, 
the pairwise perturbation algorithm computes
$\Tilde{\tsr{Y}}^{(n)}$ with error,
\[
\frac{\tnrm{\Tilde{\tsr{Y}}^{(n)}-\tsr{Y}^{(n)}}}{\tnrm{\tsr{Y}^{(n)}}} = O\left(\epsilon^2\kappa(\tsr{X})\right).
\]
\end{thm} 
\begin{proof} As in Theorem~\ref{thm:cperrbound2}, we bound the error due to second-order terms,
\begin{align*}
\frac{\tnrm{\Tilde{\tsr{Y}}^{(n)}-\tsr{Y}^{(n)}}}{\tnrm{\tsr{Y}^{(n)}}} 
= {N \choose 2}\max_{i,j}\frac{\tnrm{\tsr{Y}^{(i,j,n)}_p\times_id\mat{A}^{(i)T}\times_jd\mat{A}^{(j)T}}}{\tnrm{\tsr{Y}^{(i,j,n)}\times_i\mat{A}^{(i)T}\times_j\mat{A}^{(j)T}}} .
\end{align*}
From Lemma ~\ref{lem:kappamat}, we have   
$$
 \frac{\tnrm{\tsr{Y}^{(i,j,n)}_p\times_id\mat{A}^{(i)T}\times_jd\mat{A}^{(j)T}}}{\tnrm{\tsr{Y}^{(i,j,n)}\times_i\mat{A}^{(i)T}\times_j\mat{A}^{(j)T}}} \leq
\frac{\tnrm{\tsr{Y}^{(i,j,n)}_p}\tnrm{d\mat{A}^{(i)}}\tnrm{d\mat{A}^{(j)}}}{\tinf{\tsr{Y}^{(i,j,n)}}\tnrm{\mat{A}^{(i)}}\tnrm{\mat{A}^{(j)}}} .
$$
Since $\mat{A}^{(i)}$ and $\mat{A}^{(j)}$ are both matrices with orthonormal columns,
\[
\frac{\tnrm{\Tilde{\tsr{Y}}^{(n)}-\tsr{Y}^{(n)}}}{\tnrm{\tsr{Y}^{(n)}}} 
\leq {N \choose 2}\max_{i,j} \frac{\tnrm{\tsr{Y}^{(i,j,n)}_p}\tnrm{d\mat{A}^{(i)}}\tnrm{d\mat{A}^{(j)}}}{\tinf{\tsr{Y}^{(i,j,n)}}}=O\left(\epsilon^2\kappa(\tsr{X})\right).
\]
\end{proof}

\end{document}